\documentclass[]{article}
\usepackage{colortbl} 

\usepackage[utf8]{inputenc}
\usepackage{amsmath,amsfonts,amssymb,amsthm}
	\numberwithin{equation}{section} 
\usepackage{xcolor}
\usepackage{enumerate}
\usepackage{geometry}
\usepackage[unicode]{hyperref}   

\newtheorem{theorem}{\protect\theoremname}[section]
\newtheorem{definition}[theorem]{\protect\definitionname}
\newtheorem{lemma}[theorem]{\protect\lemmaname}
\newtheorem{proposition}[theorem]{\protect\propositionname}
\newtheorem{example}[theorem]{\protect\examplename}
\newtheorem{corollary}[theorem]{\protect\corollaryname}
\newtheorem{remark}[theorem]{\protect\remarkname}
\newtheorem{problem}[theorem]{\protect\problemname}
\newtheorem{notation}[theorem]{\protect\notationname}
\newtheorem{claim}[theorem]{\protect\claimname}
\newtheorem{conjecture}[theorem]{\protect\conjecturename}

\providecommand{\corollaryname}{Corollary}
\providecommand{\claimname}{Claim}
\providecommand{\definitionname}{Definition}
\providecommand{\lemmaname}{Lemma}
\providecommand{\notationname}{Notation}
\providecommand{\remarkname}{Remark}
\providecommand{\problemname}{Problem}
\providecommand{\propositionname}{Proposition}
\providecommand{\examplename}{Example}
\providecommand{\theoremname}{Theorem}
\providecommand{\conjecturename}{Conjecture}

\newcommand{\Q}{\mathbb{Q}}

\newcommand{\slantfrac}[2]{\,^#1\!/_#2}
\newcommand{\mc}{\mathcal}

\newcommand{\fsd}{F_{\sigma\delta}}

\newcommand{\seq}{\omega^{<\omega}}
\newcommand{\Fa}{\mc F_\alpha}
\newcommand{\ext}{\hat{\ }}

\newcommand{\baire}{\omega^\omega}

\newcommand{\Amg}[2]{\textnormal{Amg}\left(#1,#2\right)}
\newcommand{\Compl}[2]{\textnormal{Compl}\left(#1,#2\right)}
\newcommand{\rank}{r_{\textnormal{iie}}}
\newcommand{\Tr}{\textnormal{Tr}}
\newcommand{\D}{D_{\textnormal{iie}}}
\newcommand{\cltr}[1]{\textnormal{cl}_{\Tr}\left(#1\right)}
\newcommand{\ims}[2]{\textnormal{ims}_{#1}\left(#2\right)}

\newcounter{vkNoteCounter}

\newcounter{okNoteCounter}

\renewcommand{\fsd}{\mathcal{F}_{\sigma\delta}}

\begin{document}

\begin{titlepage}
\title{Complexities and Representations of $\mc F$-Borel Spaces}
\author{Vojtěch Kovařík\\
	E-mail: kovarikv@karlin.mff.cuni.cz 	\\
	\\
	Charles University \\
	Faculty of Mathematics and Physics \\
	Department of Mathematical Analysis \\
	\\
	Sokolovská 83, Karlín \\
	Praha 8, 186 00 \\
	Czech Republic}

\date{}

\maketitle

\renewcommand{\thefootnote}{}



\renewcommand{\thefootnote}{\arabic{footnote}}
\setcounter{footnote}{0}

\begin{abstract}
We investigate the~$\mc F$-Borel complexity of topological spaces in their different compactifications.
We provide a~simple proof of the~fact that a~space can have arbitrarily many different complexities in different compactifications.
We also develop a~theory of representations of $\mc F$-Borel sets,
and show how to apply this theory to prove that the~complexity of hereditarily Lindelöf spaces is absolute (that is, it is the~same in every compactification).
We use these representations to characterize the~complexities attainable by a~specific class of topological spaces.
This provides an alternative proof of the~first result,
and implies the~existence of a~space with non-absolute additive complexity.
We discuss the~method used by Talagrand to construct the~first example of a~space with non-absolute complexity, hopefully providing an explanation which is more accessible than the~original one.
We also discuss the~relation of complexity and local complexity, and show how to construct amalgamation-like compactifications.
\end{abstract}

\end{titlepage}

\setcounter{page}{2}

\section{Introduction}\label{section: intro}


We investigate complexity of $\mc F$-Borel sets, that is, of the sets from the smallest system containing closed sets and stable under taking countable unions and intersections.
This is of particular interest because the~$\mc F\textnormal{-Borel}$ classes are not absolute (unlike the Borel classes - see \cite{talagrand1985choquet} and, for example, \cite{holicky2003perfect}).
In particular, we investigate which values of $\mc F$-Borel complexity may a given space assume.
We further study various representations of $\mc F$-Borel sets.

Let us start by basic definitions and exact formulation of problems.
All topological spaces in this work will be Tychonoff (unless the contrary is explicitly specified -- see Section~\ref{section: amalgamation spaces}). 
For a~family of sets $\mc C$, we will denote by $\mc C_\sigma$ the~collection of all countable unions of elements of $\mc C$ and by $\mc C_\delta$ the~collection of all countable intersections of elements of $\mc C$.
If $\alpha$ is an ordinal and $\mc C_\beta$, $\beta \in [0,\alpha)$, are collections of sets, we denote $\mc C_{<\alpha} := \bigcup_{\beta < \alpha} \mc C_\beta$.

In any topological space $Y$, we have the~family $\mc F$ of closed sets, $\mc F_\sigma$ sets, $\mc F_{\sigma\delta}$ sets and so on. Since this notation quickly gets impractical, we use the~following definition:

\begin{definition}[$\mc F$-Borel\footnote{The $\mc F$-Borel sets are also sometimes called $\mc F$-Borelian, to denote the~fact that we only take the~system generated by countable unions and intersections of closed sets, rather than the~whole \emph{$\sigma$-algebra} generated by closed sets. Since there is little space for confusion in this paper, we stick to the~terminology of Definition~\ref{definition: F Borel sets}.} hierarchy]\label{definition: F Borel sets}
We define the~hierarchy of $\mc F$-Borel sets on a~topological space $Y$ as
\begin{itemize}
\item $\mc F_0(Y) := \mc F(Y) := $ closed subsets of $Y$,
\item $\mc F_\alpha(Y) := \left( \mc F_{<\alpha} (Y) \right)_\sigma$ for $0 < \alpha <\omega_1$ odd,
\item $\mc F_\alpha(Y) := \left( \mc F_{<\alpha} (Y) \right)_\delta$ for $0 < \alpha <\omega_1$ even.\footnote{Recall that every ordinal $\alpha$ can be uniquely written as $\alpha = \lambda + m$, where $\lambda$ is a~limit ordinal or 0 and $m\in\omega$. An ordinal is said to be even, resp. odd, when the~corresponding $m$ is even, resp. odd.}
\end{itemize}
\end{definition}

\noindent The classes $\Fa$ for odd (resp. even) $\alpha$ are called additive (resp. multiplicative).
Next, we introduce the~class of $\mc K$-analytic spaces, and a~related class of Suslin-$\mc F$ sets, which contains all $\mc F$-Borel sets (\cite[Part\,I, Cor.\,2.3.3]{rogers1980analytic}):

\begin{definition}[Suslin-$\mc F$ sets and $\mc K$-analytic spaces]
Let $Y$ be a~topological space and $S\subset Y$. $S$ is said to be \emph{Suslin-$\mc F$} in $Y$ if it is of the~form
\[ S = \bigcup_{n_0,n_1,\dots} \bigcap_{k\in\omega} F_{n_0,\dots,n_k} \]
for closed sets $F_{n_0,\dots,n_k} \subset Y$, $n_i\in\omega$, $i,k\in\omega$. By $\mc F_{\omega_1} (Y)$ we denote the~collection of Suslin-$\mc F$ subsets of $Y$.

A topological space is \emph{$\mc K$-analytic} if it is a~Suslin-$\mc F$ subset of some compact space.
\end{definition}

\noindent (Typically, $\mc K$-analytic spaces are defined differently, but our definition equivalent by \cite[Part\,I,\,Thm.\,2.5.2]{rogers1980analytic}.)

Note that if $Y$ is compact, then $\mc F_0$-subsets of $Y$ are compact, $\mc F_1$-subsets are $\sigma$-compact, $\mc F_2$-subsets are $\mc K_{\sigma\delta}$, $\Fa$-subsets for $\alpha<\omega_1$ are ``$\mc K_\alpha$'', and Suslin-$\mc F$ subsets of $Y$ are $\mc K$-analytic.

By complexity of a~topological space $X$ in some $Y$ we mean the~following:

\begin{definition}[Complexity of $X$ in $Y$]
Suppose that $X$ is Suslin-$\mc F$ subset of a~topological space $Y$. By \emph{complexity} of $X$ in $Y$ we mean
\begin{align*}
\textnormal{Compl}\,(X,Y):=
	& \text{ the~smallest } \alpha\leq\omega_1 \text{ for which } X\in \Fa(Y), \text{ that is} \\
	& \text{ the~unique } \alpha\leq\omega_1 \text{ for which } X\in \Fa(Y)\setminus \mc F_{<\alpha}(Y) .
\end{align*}
\end{definition}

\noindent Recall that a~$\mc K$-analytic space is Suslin-$\mc F$ in every space which contains it (\cite[Theorem 3.1]{hansell1992descriptive}).
Consequently, $\mc K$-analytic spaces are precisely those $X$ for which $\textnormal{Compl}\,(X,Y)$ is defined for every $Y\supset X$, and this is further equivalent to $\textnormal{Compl}\,(X,Y)$ being defined for \emph{some} space $Y$ which is compact.

Central to this work is the~following notion of complexities attainable in different spaces $Y$ containing a~given space. To avoid pathological cases such as $\Compl{X}{X}=0$ being true for arbitrarily ``ugly'' $X$, we have to restrict our attention to ``sufficiently nice'' spaces $Y$.

\begin{definition}[Attainable complexities]\label{definition:att_compl}
The \emph{set of complexities attained by a~space $X$} is defined as
\begin{align*}
\textnormal{Compl}(X) := & \left\{ \alpha \leq \omega_1 \, | \ \alpha = \textnormal{Compl}\,(X,cX) \text{ for some compactification $cX$ of } X \right\} \\
 = & \left\{ \alpha \leq \omega_1 \, | \ \alpha = \textnormal{Compl}\,(X,Y) \text{ for some $Y\supset X$ s.t. } \overline{X}^Y \text{ is compact} \right\} .
\end{align*}
\end{definition}

(The~second identity in Definition \ref{definition:att_compl} holds because $X\subset K \subset Y$ implies $\Compl{X}{K} \leq \Compl{X}{Y}$.)
If $X$ satisfies $\textnormal{Compl}(X) \subset [0,\alpha]$ for some $\alpha \leq \omega_1$, we say that $X$ is an \emph{absolute $\Fa$ space}. The smallest such ordinal $\alpha$ is called the \emph{absolute complexity of $X$} -- we clearly have $\alpha = \sup \textnormal{Compl}(X)$.
Sometimes, $X$ for which $\textnormal{Compl}(X)$ contains $\alpha$ is said to be an \emph{$\Fa$ space}.
Finally, if $\textnormal{Compl}(X)$ is empty or a~singleton, the~complexity of $X$ said to be \emph{absolute}. Otherwise, the~complexity of $X$ is \emph{non-absolute}.

The goal of this paper is to investigate the~following two problems:

\begin{problem}[$X$ of a~given complexity]\label{problem: X with Compl = C}
Let $C\subset [0,\omega_1]$. Is there a~topological space $X$ with $\textnormal{Compl}(X)=C$?
\end{problem}

\begin{problem}[Complexity of an arbitrary $X$]\label{problem: Compl of X}
For any topological space $X$, describe $\textnormal{Compl}(X)$.
\end{problem}

Inspired by behavior of complexity in Polish spaces, one might suspect that any $\Fa$ space is automatically absolutely $\Fa$. However, this quite reasonably sounding conjecture does \emph{not} hold -- there exists an $\fsd$ space $\mathbf{T}$ which is not absolutely $\Fa$ for any $\alpha<\omega_1$ (\cite{talagrand1985choquet}; $\mathbf{T}$ is also defined here, in Definition~\ref{definition: AD topology}).
The following proposition summarizes the~results related to Problem~\ref{problem: X with Compl = C} known so far.

\begin{proposition}[Attainable complexities: state of the~art]\label{proposition: basic attainable complexities}\mbox{}
\begin{enumerate}[(1)]
\item A topological space $X$ is $\mc K$-analytic if and only if $\textnormal{Compl}(X)$ is non-empty.\label{case:K-analytic}
\item For a~topological space $X$, the~following propositions are equivalent:
	\begin{enumerate}[(i)]
		\item $X$ is compact (resp. $\sigma$-compact);
		\item $\textnormal{Compl}(X)=\{0\}$ (resp. $\{1\}$);
		\item $\textnormal{Compl}(X)$ contains $0$ (resp. $1$).
	\end{enumerate} \label{case: compact and sigma compact}
\item The complexity of any separable metrizable space is absolute. In particular, for every $\alpha\leq \omega_1$, there exists a~space satisfying $\textnormal{Compl}(X)=\{\alpha\}$. \label{case: polish spaces and absolute complexity}
\item For every $2 \leq \alpha \le \beta \le \omega_1$, $\beta$ even, there exists a~topological space $X=X_\alpha^\beta$ satisfying \label{case: X alpha}
\[ \left\{ \alpha, \beta \right\} \subset \textnormal{Compl}\left(X \right) \subset [\alpha, \beta] .\]
\end{enumerate}
\end{proposition}

\begin{proof}
\eqref{case: compact and sigma compact} is trivial, since continuous images of compact sets are compact.
Regarding \eqref{case:K-analytic}, we have already mentioned that any Suslin-$\mc F$ subset of a compact space is $\mc K$-analytic, and any $\mc K$-analytic space is Suslin-$\mc F$ in every space which contains it. Since $(\textnormal{Suslin-}\mc F)(Y) = \mc F_{\omega_1}(Y) \supset \bigcup_{\alpha < \omega_1} \Fa(Y)$ holds for any $Y$, \eqref{case:K-analytic} follows.

The first part of \eqref{case: polish spaces and absolute complexity} is by no means obvious -- we are dealing with the~complexity of $X$ in \emph{all} compactifications, not only those which are metrizable. Nonetheless, it's proof is fairly elementary; see for example \cite[Theorem 2.3]{kovarik2018brooms}.
The ``in particular'' part follows from the~fact that the~Borel hierarchy in Polish spaces is non-trivial (see, for example, the~existence of universal sets in \cite{srivastava2008course}).

\eqref{case: X alpha}:
By \eqref{case: polish spaces and absolute complexity} of this proposition, there is some space $X_\alpha^\alpha$ satisfying
\[ \textnormal{Compl}(X^\alpha_\alpha) = \{ \alpha \} .\]
By \cite{talagrand1985choquet}, there exists a~space $\mathbf{T}$ which is $\fsd$ in $\beta \mathbf{T}$, but we have $\mathbf{T}\notin \mc F_{<\omega_1} (c\mathbf{T})$ for some compactification $c\mathbf{T}$ of $\mathbf{T}$. Since such a~space is $\mc K$-analytic, it satisfies 
\[ \left\{ 2, \omega_1 \right\} \subset \textnormal{Compl}\left(\mathbf{T} \right) \subset [2, \omega_1] .\]
For $\beta = \omega_1$, the~topological sum $X := X^\alpha_\alpha \oplus \mathbf{T}$ clearly has the~desired properties.

For even $\beta<\omega_1$, take the~topological sum $X := X^\alpha_\alpha \oplus X^\beta_2 $, where $X^\beta_2$ is some space satisfying
\begin{equation}\label{equation: X 2 beta}
\left\{ 2, \beta \right\} \subset \textnormal{Compl}\left(\mathbf{T} \right) \subset [2, \beta ] .
\end{equation}
The existence of such $X^\beta_2$ follows from \cite[Theorem 5.14]{kovarik2018brooms}.
\end{proof}

\section{Main Results and Open Problems}\label{section:overview}

The first contribution of this paper is showing that \eqref{case: X alpha} from Proposition~\ref{proposition: basic attainable complexities} holds not only for even ordinals, but for general $\beta \in [2,\omega_1]$:

\begin{theorem}[Non-absolute space of additive complexity]
 \label{theorem:X_2_beta}
For every $2 \leq \alpha \le \beta \le \omega_1$, there exists a~topological space $X=X_\alpha^\beta$ satisfying
\[ \left\{ \alpha, \beta \right\} \subset \textnormal{Compl}\left(X \right) \subset [\alpha, \beta] .\]
\end{theorem}

\begin{proof}
As in the~proof of Proposition~\ref{proposition: basic attainable complexities}\,\eqref{case: X alpha}, we can assume that $\alpha=2$.
The existence of $X_2^\beta$ follows from \cite[Theorem 5.14]{kovarik2018brooms} for even $\beta$, resp. from Corollary~\ref{corollary: T alpha for odd} for odd $\beta$.
\end{proof}

The second contribution to Problem~\ref{problem: X with Compl = C} made by this paper is the~following generalization of Theorem~\ref{theorem:X_2_beta}, which resolves the~uncertainty about the~set $(\alpha,\beta) \cap \textnormal{Compl}(X_\alpha^\beta)$ by showing that there is such a~space $X_{[\alpha,\beta]}$ for which $\textnormal{Compl}(X_{[\alpha,\beta]})$ is the~whole interval $[\alpha,\beta]$.

\begin{theorem}[Attainable complexities]\label{theorem: summary}
For every closed interval $I\subset [2,\omega_1]$, there is a~space $X$ with
\[ \textnormal{Compl}\left(X\right) = I .\]
\end{theorem}

\noindent The proof of Theorem \ref{theorem: summary} is presented at the end of Section~\ref{section: zoom spaces}.

Regarding Problem~\ref{problem: Compl of X}, we (weakly) conjecture that the~set of attainable complexities always has the~following properties:

\begin{conjecture}[$\textnormal{Compl}(\cdot)$ is always a~closed interval]\label{conjecture: complexities form closed interval}
For $\mc K$-analytic space $X$,
\begin{enumerate}[(i)]
	\item the~set $\textnormal{Compl}(X)$ is always an interval,
	\item the~set $\textnormal{Compl}(X)$ is always closed in $[0,\omega_1]$.
\end{enumerate}
\end{conjecture}

\noindent If Conjecture~\ref{conjecture: complexities form closed interval} holds, Theorem~\ref{theorem: summary} would actually be a~\emph{complete} solution of Problem~\ref{problem: X with Compl = C}:

\begin{conjecture}[Solution of Problem~\ref{problem: X with Compl = C}]
For any space $X$, exactly one of the~following options holds:
\begin{enumerate}[(1)]
\item $X$ is not $\mc K$-analytic and $\textnormal{Compl}(X)=\emptyset$.
\item $X$ is compact  and $\textnormal{Compl}(X) = \{0\}$.
\item $X$ is $\sigma$-compact and $\textnormal{Compl}(X) = \{1\}$.
\item $\textnormal{Compl}(X) = [\alpha,\beta]$ holds for some $2\leq \alpha \leq \beta \leq \omega_1$.
\end{enumerate}
Moreover, any of the~possibilities above is true for some $X$.
\end{conjecture}

Regarding Problem~\ref{problem: Compl of X}, we prove the~following stronger version of \eqref{case: polish spaces and absolute complexity} from Proposition~\ref{proposition: basic attainable complexities}:

\begin{proposition}
The complexity of any hereditarily Lindelöf space is absolute.
\end{proposition}

\noindent (This result was previously unpublished, but the~core observation behind it is due to J. Spurný and P. Holický -- our contribution is enabling a~simple formal proof.)

We also mention an open problem related to Problem~\ref{problem: Compl of X} which is not further discussed in this paper:
So far, the~only known examples of spaces with non-absolute complexity are based on Talagrand's broom space $\mathbf{T}$. An interesting question is therefore whether having absolute complexity is the~``typical case'' for a~topological space (and $\mathbf{T}$ is an anomaly), or whether there in fact exist many spaces with non-absolute complexity (but somehow this is difficult to prove).
As a~specific example, recall that for some Banach spaces, the~unit ball $B_X$ with weak topology is $\fsd$ the~bi-dual unit ball $B_{X^{\star\star}}$ with $w^\star$ topology. Is $(B_X,w)$ absolutely $\fsd$?

We now give a~brief overview of the~contents of the~paper.
Section~\ref{section: background} contains some preliminary concepts -- compactifications, the~(mostly standard) terminology describing the~trees on $\omega$, and several elementary results regarding derivatives on such trees.
Section~\ref{section: topological sums} is devoted to providing a~simple proof of Theorem~\ref{theorem: summary} by treating the~spaces $X_\alpha^\beta$ from Theorem~\ref{theorem:X_2_beta} as atomic and taking their generalized topological sums (``zoom spaces'', introduced in Section~\ref{section: zoom spaces}).

In Section~\ref{section: simple representations}, we introduce the~concept of a~\emph{simple $\Fa$-representation}, and show how this concept can be used to give an elegant proof of Proposition~\ref{proposition: hereditarily lindelof spaces are absolute}.
We also investigate the~concept of local complexity and its connection with the~standard complexity. As a~side-product, we prove that a~``typical'' topological space $X$ cannot have a~``universal $\Fa$-representation'', even when its complexity is absolute.

In Section~\ref{section: regular representations}, we introduce and investigate the~concept of a~\emph{regular $\Fa$-representation}. As the~name suggests, this is a~(formally) stronger notion than that of a~simple $\Fa$-representation. In particular, it allows us to take any two spaces $X\subset Y$, and construct ``$\Fa$-envelopes'' of $X$ in $Y$, looking for an ordinal $\alpha$ for which the~$\Fa$-envelope of $X$ will be equal to $X$.
We also justify the~concept of a~regular $\Fa$-representation by showing that if $X$ is an $\Fa$ subset of $Y$, then it does have a~regular $\Fa$-representation in $Y$.

In Section~\ref{section: complexity of brooms}, we study the~class of the~so-called ``Talagrand's broom spaces'' -- spaces based on Talagrand's example $\mathbf{T}$.
We investigate the~class of ``amalgamation-like'' compactifications of broom spaces in an abstract setting.
As an application of regular $\Fa$-representations, we compute which complexities are attainable by these spaces. In particular, this gives an alternative proof of Theorem~\ref{theorem: summary}.
What makes this method valuable is that, unlike the~simple approach from Section~\ref{section: topological sums}, it holds a~promise of being applicable not only to broom spaces, but to some other topological spaces as well.

It should be noted that many parts of this article can be read independently of each other, as explained by the~following remark.

\begin{remark}[How to use this paper]\label{remark:how_to_use}
It is \emph{not} necessary to read Sections~\ref{section: topological sums},~\ref{section: simple representations} and~\ref{section: regular representations} in order to read the~subsequent material.
\end{remark}

\noindent Sections~\ref{section: simple representations},~\ref{section: regular representations} and~\ref{section: complexity of brooms} all rely on the notation introduced in Section~\ref{section: sequences}. This notation is, however, not required in Section~\ref{section: topological sums}.
Section~\ref{section: topological sums} deals with topological sums, and can be read independently of the~content of any of the~following sections.
Moreover, Sections~\ref{section: simple representations},~\ref{section: regular representations} and~\ref{section: complexity of brooms} are only loosely related, and can be read independently of each other with one exception: Section~\ref{section: broom space properties} relies on one result from Section~\ref{section: Suslin scheme rank}. But if the~reader is willing to use this result as a~``black box'', this dependency can be ignored.
\section{Preliminaries}\label{section: background}

This section reviews some preliminary results.
The article assumes familiarity with the~concept of compactifications, whose basic overview can be found in Section~\ref{section: compactifications}.
Section~\ref{section: sequences} introduces the~notation used to deal with sequences, trees, and derivatives on trees.


\subsection{Compactifications and Their Ordering}\label{section: compactifications}

By a~\emph{compactification} of a~topological space $X$ we understand a~pair $(cX,\varphi)$, where $cX$ is a~compact space and $\varphi$ is a~homeomorphic embedding of $X$ onto a~dense subspace  of $cX$. Symbols $cX$, $dX$ and so on will always denote compactifications of $X$.

Compactification $(cX,\varphi)$ is said to be \emph{larger} than $(dX,\psi)$, if there exists a~continuous mapping $f : cX\rightarrow dX$, such that $\psi = f \circ \varphi$. We denote this as $ cX \succeq dX $.
Recall that for a~given $T_{3\slantfrac{1}{2}}$ topological space $X$, its compactifications are partially ordered by $\succeq$ and Stone-Čech compactification $\beta X$ is the~largest one.
Sometimes, there also exists the~smallest compactification $\alpha X$, called \emph{one-point compactification} or \emph{Alexandroff compactification}, which only consists of a~single additional point. 

In this paper, we will always assume that $cX \supset X$ and that the~corresponding embedding is identity. In particular, we will simply write $cX$ instead of $(cX,\textrm{id}|_X)$.

Much more about this topic can be found in many books, see for example \cite{freiwald2014introduction}. The basic relation between the~complexity of a~space $X$ and the~ordering of compactifications is the~following observation:

\begin{remark}[Larger compactification means smaller complexity]\label{remark: absolute complexity}
For any $\alpha\leq \omega_1$, we have
 \[ X\in\Fa (dX), \ cX\succeq dX \implies X\in\Fa (cX). \]
\end{remark}

\subsection{Trees and Derivatives on Trees}\label{section: sequences}

We now introduce the notation needed by Sections~\ref{section: simple representations},~\ref{section: regular representations} and~\ref{section: complexity of brooms}.
We start with sequences in $\omega$:

\begin{notation}[Finite and infinite sequences in $\omega$]\label{notation: sequences}
We denote
\begin{itemize}
\item $\omega^\omega:=$ infinite sequences of non-negative integers $:=\left\{\sigma : \omega \rightarrow \omega\right\}$,
\item $\seq:=$ finite sequences of non-neg. integers $:=\left\{s : n \rightarrow \omega |\ n\in\omega \right\}$.
\end{itemize}
\end{notation}

Suppose that $s\in\seq$ and $\sigma\in\baire$. We can represent $\sigma$ as $(\sigma(0),\sigma(1),\dots)$ and $s$ as $(s(0),s(1),\dots,s(n-1))$ for some $n\in\omega$.
We denote the~\emph{length} of $s$ as $|s|=\textrm{dom}(s)=n$, and set $|\sigma|=\omega$.
If for some $t\in \seq \cup \baire$ we have $|t|\geq |s|$ and $t|_{|s|}=s$, we say that $t$ \emph{extends} $s$, denoted as $t\sqsubset s$.
We say that  $u,v\in\seq$ are non-comparable, denoting as $u\perp v$, when neither $u\sqsubset v$ nor $u\sqsupset v$ holds.


Unless we say otherwise, $\omega^\omega$ will be endowed with the~standard product topology $\tau_p$, whose basis consists of sets $\mc N (s):=\left\{ \sigma \sqsupset s | \ \sigma\in\baire \right\}$, $s \in \seq$.

For $n\in\omega$, we denote by $(n)$ the~corresponding sequence of length $1$.
By $s\ext t$ we denote the~\emph{concatenation} of a~sequences $s\in\seq$ and $t\in\seq\cup \baire$.
We will also use this notation to extend finite sequences by integers and sets, using the~convention $t \hat{\ } k:=t \ext (k)$ for $k\in\omega$ and $t \hat \ S:=\left\{t\ext s\, | \ s\in S\right\}$ for $S\subset \seq \cup \baire$.

\bigskip
Next, we introduce the~notation related to trees.

\begin{definition}[Trees on $\omega$]
A \emph{tree} (on $\omega$) is a~set $T\subset \seq$ which satisfies
\[ (\forall s,t\in\seq): s\sqsubset t \ \& \ t\in T \implies s\in T. \]
\end{definition}

\noindent By $\textrm{Tr}$ we denote the~space of all trees on $\omega$.
For $S\subset \seq \cup \baire$ we denote by
\[ \cltr{S} := \left\{u\in\seq | \ \exists s\in S: s\sqsupset u \right\} \]
the smallest tree ``corresponding'' to $S$ (for $S\subset\seq$, $\mathrm{cl}_\mathrm{Tr}(S)$ is the~smallest tree containing $S$).
Recall that the~empty sequence $\emptyset$ can be thought of as the~`root' of each tree, since it is contained in any nonempty tree.
For any $t \in T$, we denote the~set of \emph{immediate successors of $t$ in $T$} as
\begin{equation}\label{equation: immediate successor in T}
\ims{T}{t} := \{ t\ext n | \ n\in\omega , \, t\ext n \in T \} .
\end{equation}
A \emph{leaf} of a~tree $T$ is a~sequence $t\in T$ with no immediate successors in $T$, and the~set of all \emph{leaves of $T$} is denoted as $l(T)$. Finally, we denote as $T^t$ the~tree corresponding to $t$ in $T$:
\begin{equation}\label{equation: T of t}
T^t := \{ t' \in \seq | \ t\ext t' \in T \} .
\end{equation}
For each non-empty tree $T$, we clearly have
\[ T = \{\emptyset\} \cup \bigcup \{ m \ext T^{(m)} \, | \ (m) \in \ims{T}{\emptyset} \} .\]

If each of the~initial segments $\sigma|n$, $n\in\omega$, of some $\sigma\in\baire$ belongs to $T$, we say that $\sigma$ is an infinite branch of $T$.
By $\textrm{WF}$ we denote the~space of all trees which have no infinite branches (the \emph{well-founded} trees). By $\textnormal{IF}$, we denote the~complement of $\textnormal{WF}$ in $\Tr$ (the \emph{ill-founded} trees).

\bigskip
Useful notion for studying trees is the~concept of a~derivative:

\begin{definition}[Derivative on trees]
A mapping $D:\Tr\rightarrow\Tr$ is a~\emph{derivative} on trees if it satisfies $D(T)\subset T$ for every $T\in \Tr$.
\end{definition}

Any derivative on trees admits a~natural extension to subsets of $\seq \cup \baire$ by the~formula $D(S) := D( \cltr{S} )$.
We define \emph{iterated derivatives} $D^\alpha$ in the~standard way:
\begin{align*}
D^0(S) & := \cltr{S}, \\
D^{\alpha +1} \left( S \right) & :=  D\left(D^\alpha\left(S\right)\right) \textrm{ for successor ordinals},\\
D^\lambda \left( S \right) & := \underset {\alpha < \lambda } \bigcap D^\alpha \left( S \right) \textrm{ for limit ordinals}.
\end{align*}

Clearly, the~iterated derivatives of any $S$ either keep getting smaller and smaller, eventually reaching $\emptyset$, or there is some $\alpha$ for which the~iterated derivatives no longer change anything, giving $D^\alpha(T)=D^{\alpha+1}(T)\neq \emptyset$.
Since each $T\in\Tr$ is countable, it suffices to consider $\alpha<\omega_1$.
This allows us to define a~\emph{rank} corresponding $D$ (on subsets of $\seq\cup \baire$):
\[ r(S) := \begin{cases}
-1 & \text{for } S=\emptyset \\
\min \{ \alpha<\omega_1| \ D^{\alpha+1}\left( S \right) = \emptyset \}
	& \text{for $S\neq \emptyset$, if there is some $\alpha<\omega_1$ s.t. } D^{\alpha+1}\left( S \right) = \emptyset \\
\omega_1 & \text{if no $\alpha$ as above exists.}
\end{cases} \]

Note that when $S$ is non-empty and $r(S)<\omega_1$, then $r(S)$ is the~highest ordinal for which $D^{r(S)}(S)$ is non-empty (or equivalently, for which $D^{r(S)}(S)$ contains the~empty sequence).
For more details and examples regarding ranks, see for example \cite[ch. 34 D,E]{kechris2012classical}.
We will be particularly interested in the~following three derivatives on trees:

\begin{definition}[Examples of derivatives]\label{definition: derivatives}
For $T\in \textrm{Tr}$, we denote
\begin{align*}
D_l (T) \ := \big\{ t\in T| \ & T \text{ contains some extension $s\neq t$ of $t$} \big\} ,\\
D_i (T) \ := \big\{ t\in T| \ & T \text{ contains some infinitely many extensions of $t$} \big\} ,\\
\D (T)  := \big\{ t\in T| \ & T \text{ contains infinitely many incomparable} \\
								& \text{ extensions of $s$ of different length} \big\} .
\end{align*}
We use the~appropriate subscripts to denote the~corresponding iterated derivatives and ranks.
\end{definition}
Using transfinite induction, we obtain the~following recursive formula for the~leaf-rank 
\begin{equation}\label{equation: leaf rank formula}
\left( \forall t\in T \right) :
r_l(T^t) = \sup \{ r_l(T^s)+1 \,| \ s\in \ims{T}{t} \} 
\end{equation}
(with the~convention that a~supremum over an empty set is 0).
It is straightforward to check that the~leaf-rank $r_l(T^t)$ of $T^t$ is the~highest ordinal for which $t$ belongs to the~iterated leaf-derivative $D^{r_l(T)}_l(T)$ of $T$.
In particular, leaves of $T$ are precisely those $t\in T$ which satisfy $r_l(T^t)=0$.

If each non-leaf $t\in T$ has immediately many successors in $T$, we clearly have $D^\alpha_i(T) = D^\alpha_l(T)$ for each $\alpha < \omega$. In particular, we have $r_i(T) = r_l(T)$ for any such $T$.
For a~general $T\in \Tr$, $r_l(T)$ is countable if and only if $r_i(T)$ is countable, and this happens if and only if $T$ is well founded.
Note that on well founded trees, $D_i$ behaves the~same way as the~derivative from \cite[Exercise 21.24]{kechris2012classical}, but it leaves any infinite branches untouched.

The following trees serve as examples of trees of rank $\alpha$ for both $r_l$ and $r_i$:

\begin{notation}[``Maximal'' trees of height $\alpha$] \label{notation: trees of height alpha}
For each limit $\alpha<\omega_1$, fix a~bijection $\pi_\alpha : \omega \rightarrow \alpha$.
We set
\begin{align*}
T_0 			& := \{\emptyset\} ,\\
T_{\alpha}	& := \{\emptyset\} \cup \underset {n\in\omega} \bigcup n\hat{\ }T_{\alpha-1} & \text{for countable successor ordinals,} \\
T_\alpha		& := \{\emptyset\} \cup \underset {n\in\omega} \bigcup n\hat{\ }T_{\pi_\alpha (n)}	& \text{for countable limit ordinals,} \\
T_{\omega_1}	& := \seq .
\end{align*}
\end{notation}

In particular, we have $T_k = \omega^{\leq k}$ for $k\in\omega$ and $T_\omega = \{\emptyset\} \, \cup \, (0) \, \cup \, 1\ext \omega^{\leq 1} \, \cup \, 2\ext \omega^{\leq 2} \, \cup \, \dots $ .
Denoting $\pi_\alpha(n):=\alpha-1$ for successor $\alpha$, we can write $T_\alpha$ as $T_\alpha = \{\emptyset\} \cup \underset {n\in\omega} \bigcup n\hat{\ }T_{\pi_\alpha (n)}$ for both limit and successor ordinals.
A straightforward induction over $\alpha$ yields $r_l(T_\alpha) = r_i(T_\alpha) = \alpha$ for each $\alpha \leq\omega_1$.

Every ordinal $\alpha$ can be uniquely written as $\alpha = \lambda + 2n + i$, where $\lambda$ is $0$ or a~limit ordinal, $n\in\omega$ and $i\in\{0,1\}$. We denote $\alpha' := \lambda + n$.
In Section~\ref{section: R T sets}, we will need trees which serve as intermediate steps between $T_\alpha$ and $T_{\alpha+1}$. These will consist of trees $\{\emptyset\} \cup n\ext T_\alpha$, $n\in \omega$. While we could call these trees something like `$T_{\alpha+\frac 1 2}$', we instead re-enumerate them as
\begin{align} \label{equation: canonical trees}
T^c_\alpha 		& := T_{\alpha'} & \text{for even $\alpha\leq\omega_1$} \nonumber ,\\
T^c_{\alpha,n} 	& := \{\emptyset\} \cup n\ext T_{\alpha'}
				& \text{for $i\in\omega$ and odd $\alpha<\omega_1$} ,\\
T^c_{\alpha} 	& := \{\emptyset\} \cup 1\ext T_{\alpha'} = T^c_{\alpha,1}
					& \text{for odd $\alpha<\omega_1$} . \nonumber
\end{align}
\section{Attaining Complexities via Topological Sums} \label{section: topological sums}

Suppose we have topological spaces $X$ and $Y$ satisfying
\[ \{\alpha,\beta\}\subset\textnormal{Compl}(X) \subset [\alpha,\beta] \ \ \ \text{ and } \ \ \ \{\alpha,\gamma\} \subset \textnormal{Compl}(Y) \subset [\alpha,\gamma ] \]
for some $\alpha \leq \beta \leq \gamma$.
It is straightforward to verify that the~topological sum $X\oplus Y$ satisfies
\[ \{ \alpha,\beta,\gamma\} \subset \textnormal{Compl}(X\oplus Y) = [\alpha,\gamma] .\]
In this section, we extend this observation to topological sums of infinitely many spaces.
We shall do so with the~help of the~concept of ``zoom'' spaces, of which topological sums are a~special case.
Moreover, we will be able to put together uncountably many spaces -- not exactly as a~topological sum, but in a~very similar manner -- and show that that this gives the~existence of a~space with $\textnormal{Compl}(X)=[2,\omega_1]$.

In Section~\ref{section: zoom spaces}, we introduce and investigate zoom spaces. In Section~\ref{section:top_sums_as_zoom}, we apply this theory to topological sums, thus proving Theorem~\ref{theorem: summary}. Apart from that, the~results of Section~\ref{section: topological sums} are independent on the~other parts of the~paper.

\subsection{Zoom spaces} \label{section: zoom spaces}
Let $Y$ be a~topological space. Throughout this section, we shall denote by $I_Y$ the~set of all isolated points of $Y$ and by $Y' = Y \setminus I_Y$ the~set of all non-isolated points of $Y$.
\begin{definition}[Zoom space]\label{definition: zoom space}
Let $Y$ be a~topological space and $\mc X=(X_i)_{i\in I_Y}$ a~collection of non-empty topological spaces.
We define the~\emph{zoom space $Z(Y,\mc X)$ of $Y$ with respect to $\mc X$} as the~disjoint union $Y' \cup \bigcup \mc X$. The basis of topology of $Z(Y,\mc X)$ consists of open subsets of $X_i$, $i\in I_Y$, and of all sets of the~form
\[ V_U := (U \setminus I_Y ) \cup \bigcup \{ X_i  | \ i \in U \cap I_Y \}
		, \ \ \ \ U\subset Y \text{ open in }Y .\]
\end{definition}

\noindent Note that the~definition above works even if the~indexing set $I_Y$ of the~collection $\mc X$ is replaced by some $I\subset I_Y$. Moreover, any collection $(X_i)_{i\in I}$ indexed by such a~subset can be extended to $\mc X \cup (\{x\})_{x\in I_Y\setminus I}$ and the~corresponding zoom space is identical to $Z(Y,\mc X)$.

The following notation for collections of sets shall be used throughout the~whole paper:

\begin{notation}[Collection of compactifications]\label{notation: c A}
Let $\mc X=(X_i)_{i\in I}$ be a~collection of topological spaces and suppose that for every $i\in I$, $c X_i$ is a~topological space containing $X_i$. We denote
\[ c\mc X:= \{cX_i | \ i \in I \} .\]
\end{notation}

The basic properties of zoom spaces are as follows:

\begin{proposition}[Basic properties of zoom spaces]\label{proposition: properties of Z}
Let $Z(Y,\mc X)$ be a~zoom space.
\begin{enumerate}[(i)]
\item $Z(Y,\mc X)$ is a~Tychonoff space and $Y$ is its quotient. Each $X_i$ is homeomorphic to the~clopen subset $X_i \subset Z(Y,\mc X)$.  For any selector $s$ of $\mc X$, $Y$ is homeomorphic to the~closed set $Y_s := Y' \cup s(I_Y) \subset Z(Y,\mc X)$. 	\label{case: Y as subspace of Z(Y,Z)}
\item If  $Y$ is a~dense subspace of $\widetilde Y$ and each $X_i$ is a~(dense) subspace of $\widetilde X_i$, then $Z(Y,\mc X)$ is a~(dense) subspace of $Z(\widetilde Y,\widetilde{\mc X})$. \label{case: subspaces}
\item If $Y$ and all $X_i$ are compact (Lindelöf), then $Z(Y,\mc X)$ is compact (Lindelöf). \label{case: Z is compact}
\item If $cY$ and $cX_i$, $i\in I_Y$, are compactifications, then $Z(cY, c\mc X)$ is a~compactification of $Z(Y,\mc X)$. \label{case: compactifications of Z(Y,Z)}
\end{enumerate}
\end{proposition}

\begin{proof}
$(i)$: Clearly, Definition~\ref{definition: zoom space} correctly defines a~topology on $Z(Y,\mc X)$ by giving its basis.
By definition of this topology, a~set $U\subset X_i$ is open in $Z(Y,\mc X)$ if and only if it is open in $X_i$, so the~topology on each $X_i$ is preserved. For every $i\in I_Y$, $Y\setminus \{i\}$ is open in $Y$. It follows that each $Z(Y,\mc X)\setminus X_i = V_{Y\setminus \{ i \}}$ is open and hence each $X_i$ is clopen in $Z(Y,\mc X)$.

To see that $Y$ is a~quotient of $Z(Y,\mc X)$, it is enough to observe that the~following mapping $q : Z(Y,\mc X) \rightarrow Y$  is continuous, surjective and open:
\[ q(x) :=
\begin{cases}
x, 	& \textrm{for } x\in Y'	\\
i, 	& \textrm{for } x\in X_i,\ i\in I_Y .
\end{cases} \]
Indeed, $q$ is clearly continuous (since $q^{-1}(U)=V_U$) and surjective.  Moreover, it maps basic open sets in $Z(Y,\mc X)$ onto open sets: we have $q(V)=\{i\} \in I_Y$ for non-empty $V\subset X_i$ and $q(V_U)=U$ for $U\subset Y$.

Let $s:I_Y\rightarrow \bigcup \mc X$ be a~selector of $\mc X$ and denote by $f_s$ the~following restriction of $q$:
\[ f_s := q|_{Y'\cup s(I_Y)} : Y' \cup s(I_Y) \rightarrow Y .\]
Clearly, this restriction of $q$ is injective and continuous. The restriction is an open mapping, because the~image of $V\subset X_i$ is either empty (when $s(i)\notin V$) or equal to $\{i\}$ and for $U\subset Y$ we have $f_s(V_U) = q(V_U) = U$. In particular, the~range of $f_s$ is $Y$, so $Y_s := Y'\cup s(I_Y)$ and $Y$ are homeomorphic.

The topology on $Z(Y,\mc X)$ is easily seen to be Hausdorff -- indeed, if $x,y\in Z(Y,\mc X)$ are distinct, then either both of them belong to some $X_i$ (which is Hausdorff and open in $Z(Y,\mc X)$), or one of them belongs to $X_i$ and the~other to $Z(Y,\mc X)\setminus X_i$ (so we separate $x$ from $y$ by $X_i$ and its complement), or both belong to $Y'$. In the~last case, we use the~fact that $Y$ is Hausdorff to separate $x$ and $y$ in $Y$ by open subsets $U, U'$ of $Y$ and note that $V_U$ and $V_{U'}$ are open sets separating $x$ and $y$ in $Z(Y,\mc X)$.

To see that $Z(Y,\mc X)$ is Tychonoff, let $F\subset Z(Y,\mc X)$ be closed and $x\in Z(Y,\mc X) \setminus F$. If $x\in Y'$, we find and open subset $U$ of $Y$ such that $x\in V_U \subset Z(Y,\mc X) \setminus F$. We separate $x$ from $Y\setminus U$ in $Y$ by a~continuous function $f:Y\rightarrow [0,1]$ and note that $f\circ q$ is a~continuous function separating $x$ from $F$ in $Z(Y,\mc X)$.

If $x\in X_i$, we find continuous $f:X_i\rightarrow [0,1]$ which separates $x$ from $F\cap X_i$. Since $X_i$ is clopen in $Z(Y,\mc X)$, $f$ can be extended into a~function $\tilde f$ which separates $x$ from $F$ in $Z(Y,\mc X)$.

$(ii)$ immediately follows from the~definition of $Z(\cdot,\cdot)$ and its topology.

$(iii)$: Assume that $Y$ and all $X_i$ are Lindelöf and let $\mc V$ be an open cover of $Z(Y,\mc X)$. Without loss of generality, we can assume that $\mc V$ consists only of basic open sets, that is, we have $\mc V = \mc V_0 \cup \bigcup_{i\in I} \mc V_i$ where each $\mc V_i$, $i\in I$, only contains open subsets of $X_i$ and $\mc V_0$ only contains sets of the~form $V_U$ for $U\subset Y$ open in $Y$.
Denote
\[ \mc U_0:=\{ U \subset Y | \ V_U \in \mc V_0 \} . \]
By definition of topology on $Z(Y,\mc X)$, $\mc U_0$ is a~cover $Y'$.
Since $Y$ is Lindelöf and $Y'$ is closed, there is a~countable  $\mc U^{\star}_0 \subset \mc U_0$ such that $\bigcup \mc U^{\star}_0 \supset Y'$. The set
\[ I^{\star}:= Y\setminus \bigcup \mc U^{\star}_0 = I_Y \setminus \bigcup \mc U^{\star}_0 \]
is closed and discrete in $Y$, and hence countable. Since each $X_i$ is Lindelöf, each $\mc V_i$ has a~countable subcover $\mc V^{\star}_i$. The system $\mc V^{\star} := \mc V_0 \cup \bigcup_{i\in I^{\star}} \mc V^{\star}_i$ is then a~countable subcover of $\mc V$.

The proof of the~compact case is the~same.

$(iv)$: This follows from $(ii)$ and $(iii)$.
\end{proof}

Recall that Definition~\ref{definition: zoom space} introduced the~sets of the~form $V_F$ for $F\subset W$:
\[ V_F := F \cup \bigcup \{ X_i | \ i\in I_W, \, F \ni i \} \subset Z(W,\mc Z) .\]
The following lemma investigates the~complexity of these sets.

\begin{lemma}[Complexity of basic sets]\label{lemma: complexity of V F sets}
For a~zoom space $Z(W,\mc Z)$, $\mc Z=(Z_i)_{i}$, the~sets $V_F$ satisfy
\begin{equation}\label{equation: complexity of V F sets}
\left( \forall F \subset W \right) : F\in \Fa (W) \implies V_F \in \Fa(Z(W,\mc Z)) .
\end{equation}
\end{lemma}

\begin{proof}
Indeed, if $F$ is closed, then $Z(W,\mc Z)\setminus V_F = V_{W\setminus F}$ is open by definition of topology on $Z(W,\mc Z)$, so $V_F$ is closed. Moreover, $\bigcup_{E\in \mc E} V_E = V_{\bigcup {\mc E}}$ and $\bigcap_{E\in \mc E} V_E = V_{\bigcap {\mc E}}$ holds for any $\mc E\subset \mc P(W)$. This implies the~result for higher $\mc F$-Borel classes (by transfinite induction) and for $F\in\mc F_{\omega_1}$ (by showing that if $F$ is Suslin in $W$, then $V_F$ is Suslin in $Z(W,\mc Z)$).
\end{proof}

Before computing the~complexity of a~general zoom space, we need to following technical lemma:

\begin{lemma}[Upper bound on zoom space complexity]\label{lemma: complexity in zoom spaces}
Let $Z(W,\mc Z)$, $\mc Z=(Z_i)_{i\in I}$, be a~zoom space and $\mc C = (C_i)_{i\in I}$ a~system satisfying $C_i \subset Z_i$ for each $i\in I$. Then for any $\alpha \leq \omega_1$, we have
\begin{equation}\label{equation: complexity in zoom spaces}
\left( \left( \forall i \in I \right) : C_i \in \Fa (Z_i) \right) \implies Z(W,\mc C) \in \Fa(Z(W,\mc Z)).
\end{equation}
\end{lemma}

\begin{proof}
For $\alpha=0$, $Z(W,\mc C)$ is closed in $Z(W,\mc Z)$ since $Z_i \setminus C_i$ are basic open sets and
\[ Z(W,\mc C) = Z(W,\mc Z) \setminus \bigcup_{i\in I} (Z_i \setminus C_i) .\]
For $0<\alpha<\omega_1$, \eqref{equation: complexity in zoom spaces} follows by transfinite induction, because we have
\[ \left( \left( \forall i\in I \right) : C_i=\bigcup_n C_i^n \right) \implies
Z\left(W,\mc C\right) = Z(W, ( \bigcup_n C_i^n )_i ) =
 \bigcup_n Z\left(W, (C_i^n )_i \right) \]
and the~analogous formula holds when each $C_i$ satisfies $C_i=\bigcap_n C_i^n$.

For $\alpha=\omega_1$, the~following formula shows that $Z(W,\mc C)$ is Suslin in $Z(W, \mc Z)$:
\begin{align*}
& \left( \left( \forall i\in I \right) : C_i = \bigcup_{\sigma\in\baire} \bigcap_n C_i^{\sigma|n}
 \text{, where each } C_i^{\sigma|n} \text{ is closed in } Z_i  \right) \implies \\
& Z\left(W,\mc C\right) = Z\left(W, \left( \bigcup_{\sigma\in\baire} \bigcap_n C_i^{\sigma|n} \right)_i \right) = \bigcup_{\sigma\in\baire} \bigcap_n Z\left(W, (C_i^{\sigma |n} )_i \right) \\
& \in \mc F_{\omega_1} \left( Z(W,\mc Z)\right) ,
\end{align*}
where at the~last line we have used \eqref{equation: complexity in zoom spaces} with $\alpha=0$.
\end{proof}

The following result then states that the~complexity of zoom spaces can be retrieved from the~complexity of its parts.

\begin{proposition}[Complexities attained by a~zoom space]\label{proposition: complexity of Z(Y,Z)}
Let $Z(Y,\mc X)$ be a~zoom space. If $Y$ is a~dense subspace of $\widetilde Y$ and each $X_i$ is a~subspace of $\widetilde X_i$, then the~spaces $Z(Y, {\mc X}) \subset Z(\widetilde Y, \widetilde {\mc X})$ satisfy
\[ \textnormal{Compl} \left( Z(Y,\mc X), Z(\widetilde Y, \widetilde {\mc X}) \right) = 
	\max\left\{ \textnormal{Compl} \left( Y, \widetilde Y \right), \ 
	\sup_{i\in I_Y} \textnormal{Compl} \left( X_i, \widetilde X_i \right) \right\} ,\]
whenever at least one of the~sides is defined.
\end{proposition}

\begin{proof}
"$\geq$": Let $Y$, $\widetilde Y$, $\mc X$ and $\widetilde {\mc X}$ be as in the~statement. Assume that the~LHS is defined and $Z(Y,\mc X) \in \Fa ( Z(\widetilde Y, \widetilde {\mc X}) )$ holds for some $\alpha \leq \omega_1$.
Let $s$ be a~selector of $\mc X$. Since $\widetilde Y$ and $Y$ have the~same isolated points, $s$ is also a~selector of $\widetilde{\mc X}$. Using the~notation from Proposition~\ref{proposition: properties of Z}\,\eqref{case: Y as subspace of Z(Y,Z)}, we have
\begin{align*}
Y_s \subset \widetilde Y_s \subset Z(\widetilde Y, \widetilde {\mc X}) && \ \& \ && Y_s   & = \widetilde Y_s \cap Z(Y,\mc X) \subset Z(\widetilde Y, \widetilde {\mc X}) \\
X_i \subset \widetilde X_i \subset Z(\widetilde Y, \widetilde {\mc X}) && \ \& \ && X_i & = \widetilde X_i \cap Z(Y,\mc X) \subset Z(\widetilde Y, \widetilde {\mc X}) ,
\end{align*}
where $\widetilde Y_s$ and $\widetilde X_i$ are closed subsets of $Z(\widetilde Y,\widetilde{\mc X})$.
In particular we get (by definition of the~subspace topology) that each $X_i$ satisfies $X_i \in \Fa ( \widetilde X_i)$ and $Y_s$ satisfies $Y_s \in \Fa ( \widetilde Y_s)$. Since $\widetilde Y_s$ is homeomorphic to $\widetilde Y$, we get $Y \in \Fa (\widetilde Y)$. It follows that the~RHS is defined and at most equal to the~LHS.

"$\leq$":
Let $Y$, $\widetilde Y$, $\mc X$ and $\widetilde {\mc X}$ be as in the~proposition. Assume that the~RHS is defined and we have $Y \in \Fa (\widetilde Y)$ and $X_i \in \Fa ( \widetilde X_i)$, $i\in I_Y$, for some $\alpha \leq \omega_1$.
Rewriting the~space $Z(Y,\mc X)$ as
\begin{equation} \label{equation: Z as intersection in tilde Z}
Z := Z\left(Y,\mc X\right) = Z\left(Y,\widetilde {\mc X}\right) \cap Z\left(\widetilde Y,\mc X\right)
 \subset Z\left(\widetilde Y,\widetilde {\mc X} \right) =: \widetilde Z ,
\end{equation}
we obtain an upper bound on its complexity:
\[ \textnormal{Compl}\left( Z, \widetilde Z \right) \le \max \left\{
	\textnormal{Compl}\left( Z\left(Y, \widetilde{\mc X}\right), \widetilde Z \right) ,
	\textnormal{Compl}\left( Z\left(\widetilde Y,\mc X\right), \widetilde Z \right)
 \right\} .\]
Working in the~zoom space $Z(\widetilde Y,\widetilde{ \mc X})$, we have $Z(Y,\widetilde{\mc X})=V_Y$. Applying \eqref{equation: complexity of V F sets} (with $W:=\widetilde Y$, $\mc Z := \widetilde {\mc X}$ and $\alpha$), we get
\[ \textnormal{Compl}\left( Z\left(Y, \widetilde{\mc X}\right), \widetilde Z \right) =
 \textnormal{Compl}\left( V_Y, Z(\widetilde Y, \widetilde{\mc X}) \right)
 \overset{\eqref{equation: complexity of V F sets}}{\le}
 \textnormal{Compl}\left( Y, \widetilde Y \right) \leq \alpha .\]
Application of \eqref{equation: complexity in zoom spaces} (with $W:= \widetilde Y$, $\mc Z := \widetilde{\mc X}$, $\mc C := \mc X$ and $\alpha$) implies that $\textnormal{Compl}\left( Z\left(\widetilde Y, {\mc X}\right), \widetilde Z \right) \leq \alpha $. It follows that the~LHS is defined and no greater than the~RHS.
\end{proof}

As a~corollary of the~proof of Proposition~\ref{proposition: properties of Z}, we obtain the~following result:

\begin{corollary}[Stability under the~zoom space operation]\label{corollary: zoom spaces are K analytic}
A zoom space $Z(Y,\mc X)$ is is compact ($\sigma$-compact, $\mc K$-analytic) if and only if $Y$ each $X_i$ is compact ($\sigma$-compact, $\mc K$-analytic).
\end{corollary}
\begin{proof}
$\Longrightarrow$: By Proposition~\ref{proposition: properties of Z} \eqref{case: Y as subspace of Z(Y,Z)}, $Z(Y,\mc X)$ contains $Y$ and each $X_i$ as a~closed subspace -- this implies the~result.

$\Longleftarrow$: The compact case is the~same as Proposition~\ref{proposition: properties of Z} \eqref{case: Z is compact}. If $Y$ and each $X_i$ is $\sigma$-compact, then each of these spaces is $F_\sigma$ in its Čech-stone compactification. By Proposition~\ref{proposition: complexity of Z(Y,Z)}, $Z(Y,\mc X)$ is an $F_\sigma$ subset of $Z(\beta Y,\beta \mc X)$. The later space is compact by the~compact case of this proposition, which implies that $Z(Y,\mc X)$ is $\sigma$-compact.

Consider the~case where $Y$ and each $X_i$ is $\mc K$-analytic. In the~proof of Proposition~\ref{proposition: complexity of Z(Y,Z)}, we have shown that both $Z\left(Y,\beta {\mc X}\right)$ and $Z\left(\beta Y,\mc X\right)$ are Suslin in $Z(\beta Y, \beta \mc X)$. It follows that both these sets are $\mc K$-analytic. By \eqref{equation: Z as intersection in tilde Z}, $Z(Y,\mc X)$ is an intersection of two $\mc K$-analytic sets, which implies that it it itself $\mc K$-analytic.
\end{proof}

\subsection{Topological sums as zoom spaces}\label{section:top_sums_as_zoom}
In this section, we give some natural examples of zoom spaces, and prove the~first part of our main results.

\begin{example}[Topological sums as zoom spaces] \label{example: application to topological sums}
Whenever $\mc X=(X_i)_{i\in I}$ is a~collection of topological spaces, the~topological sum $\bigoplus \mc X$ is homeomorphic to the~zoom space $Z(I,\mc X)$ of the~discrete space $I$.
\end{example}
\begin{proof}
This follows from the~fact that each $X_i$ is clopen in $Z(I,\mc X)$ and that in this particular case, we have $I_I=I$ and hence $Z(I,\mc X)$ is a~disjoint union of $X_i$, $i\in I$.
\end{proof}

In the~particular case of countable sums, we can use Proposition~\ref{proposition: complexity of Z(Y,Z)} to fully describe the~complexity of the~resulting space:
\begin{proposition}[Complexities attainable by a~topological sum]\label{proposition: complexities of topological sums}\mbox{}
\begin{enumerate}[(i)]
\item If $X_k$, $k\leq n$, are compact, then $\overset{n}{\underset{k=0}{\bigoplus}} X_k$ is compact.
\item If $X_k$, $k\in \omega$, are $\sigma$-compact, then $\underset{k\in \omega}{\bigoplus} X_k$ is $\sigma$-compact.
\item If $X_k$, $k\in\omega$, are $\mc K$-analytic and at least one $X_k$ is not compact, then
\[ \textnormal{Compl}\left( \bigoplus_{k\in\omega} X_k \right) = \left\{ \sup f | \ f\in \prod_{k\in\omega} \textnormal{Compl}(X_k) \right \} . \]
\end{enumerate}
\end{proposition}
\begin{proof}
$(i)$ and $(ii)$ are obvious. To get $(iii)$, note that by Example~\ref{example: application to topological sums}, the~topological sum $\bigoplus_{k} X_k$ can be rewritten as
\[ X:=\bigoplus_{k\in\omega} X_k = Z(\omega,(X_k)_{k}) .\]
We will prove each of the~two inclusions between $\textnormal{Compl}(X)$ and $\{ \sup f | \ f\in \Pi_k \textnormal{Compl}(X_k) \}$.

"$\supset$": Let $f$ be a~selector for $(\textnormal{Compl}(X_k))_{k}$. By definition of $\textnormal{Compl}(\cdot)$, there exist compactifications  $cX_k$, $k\in\omega$, such that $\Compl{X_k}{cX_k}=f(k)$. Applying Proposition~\ref{proposition: properties of Z}\,\eqref{case: compactifications of Z(Y,Z)}, we obtain a~compactification of $X$:
\[ cX:=Z(\omega+1,(cX_k)_k) .\]
Since $\omega$ is $\sigma$-compact, the~complexity of $\omega$ in $\omega+1$ is 1. We have $\sup_{k\in\omega} \Compl{X_k}{cX_k} \geq 1$ (since not every $X_k$ is compact). By Proposition $\ref{proposition: complexity of Z(Y,Z)}$, the~complexity of $X$ in $cX$ is the~maximum of these two numbers, so we get
\[ \sup f = \sup_{k\in\omega} \ \Compl{X_k}{cX_k} = \Compl{X}{cX} \in \textnormal{Compl}(X) . \]

"$\subset$": Let $cX$ be a~compactification of $X$ and denote $Y:= \bigcup_{k\in\omega} \overline{X_k}^{cX}$. Define $f:\omega \rightarrow [1,\omega_1]$ by $f(k) := \Compl{X_k}{\overline{X_k}^{cX}}$ and set $\alpha:=\sup f$. Note that each $X_k$ is an $\Fa$-subset of $Y$. Since $Y$ is $\sigma$-compact and $\alpha\geq 1$, it is enough to show that $X$ is an $\Fa$-subset of $Y$.

If $\Fa$ is an additive class, $X=\bigcup X_k$ is a~countable union of $\Fa$ sets, so we have $X\in \Fa (Y)$. If $\Fa$ is a~multiplicative class, we have to proceed more carefully. Let $k,l\in \omega$ be distinct. We set $K^l_k := \overline{X_k}^Y\cap \overline{X_l}^Y$ and note that by definition of topological sum, $K^l_k$ is a~compact disjoint with $X_l$. Since $X_l$ is $\mc K$-analytic, and hence Lindelöf, we can apply \cite[Lemma 14]{spurny2006solution} to obtain a~set $E^l_k\in F_\sigma(Y)$ satisfying
\begin{equation} \label{equation: separation of E from K}
X_l \subset E^l_k \subset \overline{X_l}^Y \setminus K^l_k  = \overline{X_l}^Y \setminus \overline{X_k}^Y .
\end{equation}
Denote $E_k := \bigcup_{l \neq k} E^l_k \in F_\sigma(Y)$. We finish the~proof by starting with \eqref{equation: separation of E from K} and taking the~union over $l\in\omega \setminus \{k\}$ \dots
\begin{align*}
\bigcup_{l \neq k} X_l \ & \subset & \bigcup_{l \neq k} E^l_k = E_k & & \subset &
\ \ \bigcup_{l \neq k} \overline{X_l}^Y \setminus \overline{X_k}^Y = Y \setminus \overline{X_k}^Y \\
&&&&& \dots \text{then adding } X_k \dots \\
X = X_k \cup \bigcup_{l \neq k} X_l \ & \subset & X_k \cup E_k \ \ & & \subset &
\ \ X_k \cup \left( Y \setminus \overline{X_k}^Y \right) = Y \setminus \left( \overline{X_k}^Y \setminus X \right) \\
&&&&& \dots \text{and intersecting over } k\in\omega \\
X\, \ & \subset & \bigcap_{k\in\omega} \Big( X_k \cup E_k & \Big) & \subset & \ \ Y \setminus \bigcup_{k\in\omega} \left( \overline{X_k}^Y \setminus X \right) = X .
 \end{align*}
This proves that $X$ is an intersection of $\Fa$-subsets of $Y$. Since $\Fa$ is multiplicative, this completes the~proof.
\end{proof}

Proposition~\ref{proposition: complexities of topological sums} and its proof imply that for any closed interval $I\subset [2,\omega_1]$, there exists a~space satisfying $\textnormal{Compl}(X)=I$:

\begin{proof}[Proof of Theorem~\ref{theorem: summary}]
Let $I=[\alpha, \beta] \subset [2,\omega_1]$ be a~closed interval.
For $\gamma \in [\alpha,\beta]$, let $X_\alpha^\gamma$ be a~topological space from Theorem~\ref{theorem:X_2_beta} satisfying
\[ \left\{ \alpha,\gamma \right\} \subset \textnormal{Compl}\left( X_\alpha^\gamma \right) \subset [\alpha,\gamma] \]
and denote $\mc X := (X_\alpha^\gamma)_{\gamma\in [\alpha,\beta]}$.

If $\alpha=\beta = \omega_1$, we simply set $X := X^{\omega_1}_{\omega_1}$.
Next, consider the~case when $\beta < \omega_1$.
The family $\mc X$ is countable and it satisfies
\[ \left\{ \sup f | \ f \text{ is a~selector of } \left( \textnormal{Compl}(X) \right)_{X\in \mc X} \right\} = [\alpha, \beta] .\]
By Proposition~\ref{proposition: complexities of topological sums}\, $(iii)$, the~space $X := \bigoplus \mc X$ has the~desired properties.

Finally, suppose that $\beta=\omega_1$.
We can assume that $\alpha=2$ -- otherwise, it would suffice to construct $\widetilde X$ corresponding to $I=[2,\omega_1]$ and set $X := \widetilde X \oplus X^\alpha_\alpha$.
Denote by $K$ the~one-point compactification of the~\emph{discrete} space $[2,\omega_1]$ (instead of the~usual ordinal topology).
Since $I_K = [2,\omega_1]$, it makes sense to consider the~space $X := Z(K,\mc X)$.
By Corollary~\ref{corollary: zoom spaces are K analytic}, $Z(K,\mc X)$ is $\mc K$-analytic but not $\sigma$-compact, which implies that $\textnormal{Compl}(Z(K,\mc X))\subset [2,\omega_1]$. Clearly, we have
\[ [2,\omega_1] = \left\{ \sup f | \ f \text{ is a~selector of } \left( \textnormal{Compl}(X_2^\gamma) \right)_{\gamma \in [2,\omega_1]} \right\}  .\]
To finish the~proof, it is enough to prove the~inclusion
\[ \left\{ \sup f | \ f \text{ is a~selector of } \left( \textnormal{Compl}(X_2^\gamma) \right)_{\gamma \in [2,\omega_1]} \right\} \subset \textnormal{Compl}( Z(K,\mc X) ) .\]
This can be done exactly as in the~``$\supset$''-part of the~proof of Proposition~\ref{proposition: complexity of Z(Y,Z)}.
\end{proof}

\section{Simple Representations}\label{section: simple representations}

We begin with an example of different ways of describing and representing of $\fsd$-sets.
This example will be rather trivial; however, it will serve as a motivation for studying higher classes of the~$\mc F$-Borel hierarchy.

Suppose that $Y$ is a topological space and $X\in \mc F_2(Y)$.

By definition of $\mc F_2(Y)$, $X$ can be described in terms of $\mc F_1$-subsets of $Y$ as $X=\bigcap_m X_m$, where $X_m \in \mc F_1(Y)$.
And by definition of $\mc F_1(Y)$, each $X_m$ can be rewritten as $X_m := \bigcup_n X_{m,n}$, where $X_{m,n}\in \mc F_0(Y)$.
We could say that the~fact that $X$ is an $\mc F_2$-subset of $Y$ is represented by the~collection $\mc H := (X_{m,n})_{(m,n)\in\omega^2}$.

For $(n_0,n_1,\dots,n_k) \in \seq$, we can set $C(n_0,n_1,\dots,n_k) := X_{0,n_0} \cap X_{1,n_1} \cap \dots \cap X_{k,n_k}$.
We then have $s\sqsubset t \implies C(t) \subset C(s)$ and for each $m\in \omega$, $\{ C(s) | \ s\in \omega^m \}$ covers $X$.
The fact that $X$ is an $\mc F_2$-subset of $Y$ is then represented by the~collection $\mc C := (C(s))_{s}$, in the~sense that $X= \bigcap_{m\in\omega} \bigcup_{s\in \omega^m} C(s)$.
We denote this as $X = \mathbf{R}_2(\mc C)$.

In both cases, $X$ is represented in terms of closed subsets of $Y$. We can also represent $X$ by collections of subsets of $X$, by working with $\mc H_X := (X_{m,n} \cap X)_{(m,n)\in\omega^2}$ and $\mc C_X := (C(s)\cap X )_{s}$ instead.
$X$ is then represented by the~collection $\mc H' := \overline{\mc H_X}^Y$, resp. as $X = \mathbf{R}_2 (\overline{\mc C_X}^Y) =: \mathbf{R}_2 (\mc C_X, Y)$ (where for a collection $\mc A$ of subsets of $Y$, $\overline{\mc A}^Y$ denotes the~collection of the~corresponding closures in $Y$).

We call the~former representation a \emph{simple $\mc F_2$-representation of $X$ in $Y$}, and the~later representation a \emph{regular $\mc F_2$-representation of $X$ in $Y$}.

The advantage of representing $X$ by a collection of subsets of $X$ is that this collection might be able to represent $X$ in more than just one space $Y$.
This allows us to define a \emph{universal} representation of $X$ (either simple or regular) -- a collection which represents $X$ in every space which contains it.
It turns out that only very special topological spaces admit universal representations, but in Section~\ref{section: broom space properties}, we will see that it is the~case for the~class of the~so-called broom spaces.

This section is organized as follows: In Section~\ref{section:simple_rep_def}, we introduce the~concept of a simple representation and discuss its main properties. In Section~\ref{section:her_lind}, we give a sample application of this concept by providing a simple proof of the~fact that the~complexity of hereditarily Lindelöf spaces is absolute. In Section~\ref{section: universal representation}, we investigate the~concept of local complexity and its relation to standard complexity and universal representations (resp. their non-existence).
Reading this section might make it easier to understand the~intuition behind Section~\ref{section: regular representations}, where regular representations are introduced. However, the~results of this section are in no way formally required by the~subsequent parts of this paper.

\subsection{Definition of Simple Representations}\label{section:simple_rep_def}

This section will make an extensive use of the~notions introduced in Section~\ref{section: sequences}, in particular those related to the~leaf-derivative on trees.
We will repeatedly use collections of sets which are indexed by leaves of some tree:

\begin{definition}[Leaf-scheme]\label{definition: leaf scheme}
A collection $\mc H = (\mc H_t)_{t\in l(T)}$, where $T \in \textnormal{WF}$, is said to be a~\emph{leaf-scheme}.
\end{definition}

\noindent When we need to specify the~indexing set, we will say that $\mc H$ is an \emph{$l(T)$-scheme}.
Saying that a~$l(T)$-scheme is \emph{closed in $Y$}, for a~topological space $Y$, shall mean that $\mc H(t)$ is closed in $Y$ for each $t\in l(T)$.

Any $l(T)$-scheme $\mc H$ can be viewed as a~mapping with domain $l(T)$. $\mc H$ has a~natural extension to $T$, defined by the~following recursive formula:
\begin{equation} \label{equation: H t extensions}
\mc H(t) := \begin{cases}
\mc H(t) & \text{when } t \text{ is a~leaf of }T \\
\bigcup \{ \mc H(s) | \ s\in \ims{T}{t} \} & \text{when } r_l(T^t) > 0 \text{ is odd} \\
\bigcap \{ \mc H(s) | \ s\in \ims{T}{t} \} & \text{when } r_l(T^t) > 0 \text{ is even}.
\end{cases}
\end{equation}

It should be fairly obvious that closed leaf-schemes in any topological space $Y$ \emph{naturally} correspond to $\mc F$-Borel subsets of $Y$, in the~sense that a~set $X\subset Y$ is $\mc F$-Borel in $Y$ if and only if it is of the~form $X=\mc H(\emptyset)$ for some closed leaf-scheme $\mc H$ in $Y$.
For now, we only prove one direction of this claim (the other will be shown in Proposition~\ref{proposition: existence of simple representations}).

\begin{lemma}[Complexity of sets corresponding to leaf-schemes]\label{lemma: complexity of S T}
Let $\mc H$ be an $l(T)$-scheme in a~topological space $Y$.
If $\mc H$ is closed in $Y$, then $\mc H(\emptyset) \in \mc F_{r_l(T)}(Y)$.

In particular, if $\mc H$ is a~closed $l(T_\alpha)$-scheme in $Y$ for some $\alpha<\omega_1$, then $\mc H(\emptyset) \in \Fa(Y)$.
\end{lemma}

\begin{proof}
The ``in particular'' part is a~special case of the~first part because $r_l(T_\alpha)=\alpha$.
For the~general part, suppose that $Y$ is a~topological space, and $\mc H$ is a~closed $l(T)$-scheme in $Y$.

To prove the~lemma, it suffices to use induction over $\alpha$ to obtain the~following complexity estimate:
\begin{equation}\label{equation: complexity of H t}
\left( \forall t\in T \right) : r_l(T^t)=\alpha \implies \mc H(t) \in \Fa (Y) .
\end{equation}
For $\alpha = 0$, the~$t$-s for which $r_l(T^t)=0$ are precisely the~leaves of $T$. For these we have $\mc H(t) \in \mc F_0(Y)$ because $\mc H$ is closed in $Y$.

$<\alpha \mapsto \alpha$:
Suppose that \eqref{equation: complexity of H t} holds for every $\beta < \alpha$, and $t\in T$ satisfies $r_l(T^t)= \alpha$.
By \eqref{equation: leaf rank formula}, we have $r_l(T^s) < r_l(T^t) = \alpha$ for every $s\in \ims{T}{t}$.
It then follows from the~induction hypothesis that each $s\in \ims{T}{s}$, $\mc H(s)$ belongs to $\mc F_{<\alpha}(Y)$.
When $\alpha$ is odd, we have
\[ \mc H(t) \overset{\alpha \text{ is}}{\underset{\text{odd}}=}
\bigcup \left\{ \mc H(s) \, | \ s\in\ims{T}{s} \right\} \in (\mc F_{<\alpha}(Y))_\sigma
\overset{\alpha \text{ is}}{\underset{\text{odd}}=} \Fa(Y) .\]
When $\alpha$ is even, we have 
\[ \mc H(t) \overset{\alpha \text{ is}}{\underset{\text{even}}=}
\bigcap \left\{ \mc H(s) \, | \ s\in\ims{T}{s} \right\} \in (\mc F_{<\alpha}(Y))_\delta
\overset{\alpha \text{ is}}{\underset{\text{even}}=} \Fa(Y) .\]
\end{proof}

\bigskip
Lemma~\ref{lemma: complexity of S T} motivates the~following definition of (simple) $\Fa$-representation:

\begin{definition}[Simple representations]\label{definition: simple representation}
Let $X\subset Y$ be topological spaces, and $\mc H$ an $l(T)$-scheme in $Y$. $\mc H$ is said to be a
\begin{itemize}
\item \emph{simple representation} of $X$ in $Y$ if it satisfies $\mc H(\emptyset)=X$;
\item \emph{simple $\mc F$-Borel-representation} of $X$ in $Y$ if it is both closed in $Y$ and a~simple representation of $X$ in $Y$;
\item \emph{simple $\Fa$-representation} of $X$ in $Y$, for some $\alpha<\omega_1$, if it is a~simple $\mc F$-Borel-representation of $X$ in $Y$ and we have $r_l(T)\leq \alpha$.
\end{itemize}
\end{definition}

\noindent Unless we need to emphasize that we are not talking about the~regular representations from the~upcoming Section~\ref{section: regular representations}, we omit the~word `simple'.
We note that the~following stronger version of Lemma~\ref{lemma: complexity of S T} also holds, which allows for a~different, but equivalent, definition of a~simple $\Fa$-representation:

\begin{remark}[Alternative definition of $\Fa$-representations]
In Lemma~\ref{lemma: complexity of S T}, we even get $\mc H(\emptyset) \in \mc F_{r_i(T)}(Y)$.
\end{remark}

\noindent Consider the~derivative $D_i$ from Definition~\ref{definition: derivatives}  which cuts away sequences which only have finitely many extensions, and the~corresponding rank $r_i$.
Using the~fact that the~class of closed sets (and of $\Fa$-sets as well) is stable under finite unions and intersections, we can prove that when a~tree $T$ is finite and we have $\mc H(t)\in \Fa(Y)$ for each $t\in l(T)$, then $\mc H(\emptyset) \in \Fa(Y)$.
It follows that the~conclusion of Lemma~\ref{lemma: complexity of S T} also holds when we  replace `$r_l$' by `$r_i$'.
In particular, we could equally well define $\Fa$-representations as those $\mc F$-Borel representations which are indexed by trees which satisfy $r_i(T)\leq \alpha$.
We will mostly work with trees on which the~ranks $r_l$ and $r_i$ coincide, so the~distinction will be unimportant.

\bigskip

Before proceeding further, we remark that as long as the~intersection with $X$ is preserved, cutting away some pieces of a~representing leaf-scheme still yields a~representation of $X$:

\begin{lemma}[Equivalent representations]\label{lemma: leaf scheme making smaller}
Let $X\subset Y$ be topological spaces, $\mc H$ and $\mc H'$ two $l(T)$-schemes in $Y$, and suppose that $\mc H$ is a~representation of $X$ in $Y$.

If $\mc H'$ satisfies $\mc H(t) \cap X \subset \mc H'(t) \subset \mc H(t)$ for every $t\in l(T)$, then $\mc H'$ is also a~representation of $X$ in $Y$.
\end{lemma}

\begin{proof}
Let $X$, $Y$, $T$, $\mc H$ and $\mc H'$ be as above.
By the~assumptions of the~lemma, we have the~following formula for every $t\in l(T)$:
\begin{equation}\label{equation: H and H'}
\mc H(t) \cap X \subset \mc H'(t) \ \ \ \  \& \ \ \ \ \mc H'(t) \subset \mc H(t).
\end{equation}
Using \eqref{equation: H t extensions} and induction over $r_l(T^t)$, we get \eqref{equation: H and H'} for every $t\in T$.
Applying \eqref{equation: H and H'} to $t=\emptyset$, we get the~desired conclusion of the~lemma:
\[ X = \mc H(\emptyset) = \mc H(\emptyset) \cap X \subset \mc H' (\emptyset) \subset
\mc H(\emptyset) = X .\]
\end{proof}

The existence of simple representations is guaranteed by the~following result:

\begin{proposition}[Existence of $\mc F$-Borel representations]\label{proposition: existence of simple representations}
For topological spaces $X\subset Y$ and $\alpha<\omega_1$, the~following conditions are equivalent.
\begin{enumerate}[(i)]
\item $X \in \Fa(Y)$
\item $X$ has a~simple $\Fa$-representation in $Y$; that is, $\mc H(\emptyset) = X$ holds for some closed $l(T)$-scheme $\mc H$ in $Y$, where $r_l(T)\leq \alpha$.
\item $\mc H(\emptyset) = X$ holds for some closed $l(T_\alpha)$-scheme $\mc H$ in $Y$.
	\newcounter{nameOfYourChoice}
	\setcounter{nameOfYourChoice}{\value{enumi}}
\end{enumerate}
Denoting $\overline{\mc H}^Y\!\! := \left( \overline{\mc H(t) \cap Y}^Y\!\! \right)_{t\in l(T)}$ for any $l(T)$-scheme $\mc H$, this is further equivalent to:
\begin{enumerate}[(i)]
	\setcounter{enumi}{\value{nameOfYourChoice}}
\item $\overline{\mc H}^Y\!\!(\emptyset) = X$ holds for some $l(T)$-scheme $\mc H$ in $X$, where $r_l(T)\leq \alpha$.
\item $\overline{\mc H}^Y\!\!(\emptyset) = X$ holds for some $l(T_\alpha)$-scheme $\mc H$ in $Y$.
\end{enumerate}
\end{proposition}

\begin{proof}
Since the~leaf-schemes $\overline{\mc H}^Y$ in (iv) and (v) are closed in $Y$, it follows from Lemma~\ref{lemma: complexity of S T} that any of the~conditions (ii)-(v) implies (i).
The implications (iii)$\implies$(ii) and (v)$\implies$(iv) are trivial.
By Lemma~\ref{lemma: leaf scheme making smaller}, we have (ii)$\implies$(iv) and (iii)$\implies$(v) (set $\mc H_X(t) := \mc H(t) \cap X$ and $\mc H'(t) := \overline{\mc H_X}^Y\!\!$ in Lemma~\ref{lemma: leaf scheme making smaller}).
Therefore, it remains to prove (i)$\implies$(iii).
We will first prove (i)$\implies$(ii), and then note how to modify the~construction such that it gives (i)$\implies$(iii).

(i)$\implies$(ii):
We proceed by induction over $\alpha$.
For $\alpha=0$, let $X$ be a~closed subset of $Y$. We set $\mc H(\emptyset) := X$, and observe that such $\mc H$ is a~closed $l(S)$-scheme in $Y$, where $S = \{\emptyset\}$.

$\alpha \mapsto \alpha+1$ for odd $\alpha+1$:
Suppose that $X = \bigcup_{n\in \omega} X_n$, where $X_n \in \Fa(Y)$.
For each $n\in \omega$, we apply the~induction hypothesis to obtain a~closed $l(S_n)$-scheme $\mc H_n$ such that $\mc H_n(\emptyset) = X_n$ and $r_l(S_n) = \alpha$.
For $n\in\omega$ and $s\in l(S_n)$, we set $\mc H (n\ext s) := \mc H_n(s)$. This defines a~leaf-scheme indexed by the~leaves of the~tree $S := \{\emptyset\} \cup \bigcup_n n\ext S_n$, and the~formula $\mc H(n\ext s) = \mc H_n(s)$ obviously holds for every $n\ext s \in S$.
By \eqref{equation: leaf rank formula}, we have $r_l(S) = \sup_n r_l(S_n) +1 = \alpha+1$.
Since $\alpha+1$ is odd, it follows that
\[ \mc H(\emptyset) = \bigcup_n \mc H(\emptyset \ext n) = \bigcup_n \mc H(n \ext \emptyset ) = \bigcup_n \mc H_n(\emptyset) = \bigcup_n X_n = X .\]

$<\alpha \mapsto \alpha$ for even $\alpha$:
Suppose that $X = \bigcap_{n\in \omega} X_n$, where $X_n \in \mc F_{\alpha_n}(Y)$ holds for some $\alpha_n < \alpha$.
Since $\mc F_\beta (Y) \subset \mc F_\gamma(Y)$ holds whenever $\beta$ is smaller than $\gamma$, we can assume that either $\alpha$ is a~successor and we have $\alpha_n = \alpha-1$ for each $n$, or $\alpha$ is limit and we have $\sup_n \alpha_n = \alpha$. In either case, we have $\sup_n ( \alpha_n +1 ) = \alpha$.

The remainder of the~proof proceeds as in the~previous case -- applying the~induction hypothesis, defining an $l(S)$-scheme as $\mc H(n\ext s) := \mc H_n(s)$, and observing that $\mc H(\emptyset) = \bigcap_n X_n$ (because $r_l(S) = \alpha$ is even).

(i)$\implies$(iii):
We prove by induction that in each induction-step from the~proof of '(i)$\implies$(ii)', the~tree $S$ can be of the~form $S=T_\alpha$.
Recall here the~Notation~\ref{notation: trees of height alpha} (which introduces these trees).

In the~initial step of the~induction, we have $S=\{\emptyset\}=T_0$ for free.
Similarly for the~odd successor step, we have $S=T_{\alpha+1}$ provided that each $S_n$ is equal to $T_\alpha$.

Consider the~``$<\alpha \mapsto \alpha$ for even $\alpha$'' step.
By the~remark directly below Notation~\ref{notation: trees of height alpha}, we have $T = \{\emptyset\} \cup \bigcup_n n\ext T_{\pi_\alpha(n)}$. Therefore, it suffices to show that we can assume $\alpha_n \leq \pi_\alpha(n)$ for each $n\in\omega$.
By replacing the~formula $X = \bigcap_n X_n$ by
\[ X = X_0 \cap \overline{X}^Y \cap X_1 \cap \overline{X}^Y \cap X_2 \cap \dots ,\]
we can guarantee that $\alpha_n = 0$ holds for infinitely many $n\in\omega$.
Re-enumerating the~sets $X_n$, we can ensure that $\alpha_n \leq \pi_\alpha(n)$ holds for each $n\in\omega$.
 \end{proof}

The conditions (iv) and (v) motivate the~definition of a~universal simple representation:

\begin{definition}[Universal representation]\label{definition: universal simple representation}
A leaf-scheme $\mc H$ in $X$ is said to be a~\emph{universal simple $\mc F$-Borel-representation of $X$} if for every $Y\supset X$, $\overline{\mc H}^Y\!\!$ is a~simple $\mc F$-Borel-representation of $X$ in $Y$.
\end{definition}

Clearly, compact and $\sigma$-compact spaces have universal representations.
In Section~\ref{section: broom space properties}, we show an example of more complex spaces with universal representation.
In Section~\ref{section: universal representation}, we show that these spaces are rather exceptional, because sufficiently topologically complicated spaces do not admit universal  representations.

\subsection{The Complexity of Hereditarily Lindelöf Spaces}\label{section:her_lind}

We note that simple representations can be used not only with the $\mc F$-Borel class, but also be with other descriptive classes:

\begin{remark}\label{remark: simple C representations}
Simple representations based on $\mc C$-sets.
\end{remark}

\noindent Let $Y$ be a topological space and $\mc C(Y)$ a family of subsets of $Y$.
Analogously to Definition \ref{definition: F Borel sets}, we define
\begin{itemize}
\item $\mc C_0(Y) := \mc C(Y)$,
\item $\mc C_\alpha(Y) := \left( \mc C_{<\alpha} (Y) \right)_\sigma$ for $0 < \alpha <\omega_1$ odd,
\item $\mc C_\alpha(Y) := \left( \mc C_{<\alpha} (Y) \right)_\delta$ for $0 < \alpha <\omega_1$ even.
\end{itemize}

Analogously to \ref{definition: simple representation}, an $l(T)$-scheme $\mc H$ which represents $X$ in $Y$ is said to be a \emph{simple $\mc C_\alpha$-representation of $X$ in $Y$} when the tree $T$ satisfies $r_l(T)\leq \alpha$ and we have $\mc H(t) \in \mc C(Y)$ for each $t\in l(T)$.

Clearly, the equivalence (i)$\iff$(ii)$\iff$(iii) in Proposition~\ref{proposition: existence of simple representations} holds for $\mc C_\alpha$-sets and $\mc C_\alpha$-representations as well. (The proof literally consists of replacing `$\mc F$' by `$\mc C$' in Lemma \ref{lemma: complexity of S T} and in the proof of Proposition~\ref{proposition: existence of simple representations}.)

\bigskip

To prove the main result of this subsection, we need the following separation lemma, which is an immediate result of \cite[Lemma\,14]{spurny2006solution}.

\begin{lemma}[$F_\sigma$-separation for Lindelöf spaces]\label{lemma: fsd spaces are Fs separated}
Let $L$ be a Lindelöf subspace of a compact space $C$. Then for every compact set $K\subset C \setminus L$, there exists $H\in \mc F_\sigma\left(C\right)$, such that $L\subset H\subset C\setminus K$.
\end{lemma}

Recall that a topological space is said to be \emph{hereditarily Lindelöf} if its every subspace is Lindelöf. For example, separable metrizable spaces (or more generally, spaces with countable weight) are hereditarily Lindelöf. The following proposition shows that the $\mc F$-Borel complexity of such spaces is absolute:

\begin{proposition}[J. Spurný and P. Holický]\label{proposition: hereditarily lindelof spaces are absolute}
For a hereditarily Lindelöf space and $\alpha \leq \omega_1$,
the following statements are equivalent:
\begin{enumerate}[(i)]
\item $X \in \Fa(Y)$ holds for every $Y\supset X$;
\item $X \in \Fa(cX)$ holds for some compactification $cX$.
\end{enumerate}
\end{proposition}

In \cite{kalenda2018absolute} the author of the present paper, together with his supervisor, proved Proposition~\ref{proposition: hereditarily lindelof spaces are absolute} for the class of $\fsd$ sets ($\alpha=2$). J. Spurný later remarked that it should be possible to give a much simpler proof using the fact that the classes originating from the \emph{algebra} generated by closed sets are absolute (\cite{holicky2003perfect}). Indeed, this turned out to be true, and straightforward to generalize for $\mc F_n$, $n\in\omega$. However, the proof in the general case of $\Fa$, $\alpha \in \omega_1$, would would not be very elegant, and can be made much easier by the use of simple representations.

\begin{proof}
For $\alpha=0$, $\alpha=1$, and $\alpha=\omega_1$, this follows from the fact that compact, $\sigma$-compact and $\mc K$-analytic spaces are absolute.

To prove (ii)$\implies$(i), it suffices to show that (ii) implies that $X$ is $\Fa$ in every compactification (by, for example, \cite[Remark\,1.5]{kovarik2018brooms}).

Let $\alpha \in [2,\omega_1)$, suppose that $X\in \Fa(cX)$ holds for some compactification $cX$, and let $dX$ be another compactification of $X$.
Since we have
\[ \mc F(cX) \subset \{ F\cap G | \ F\subset cX \text{ is closed, $G\subset cX$ is open} \}  =: (\mc F\land \mc G)(cX) ,\]
it follows that $X \in (\mc F\land \mc G)_\alpha (cX)$.
The classes $(\mc F\land \mc G)_\alpha$ are absolute by \cite[Corollary 14]{holicky2003perfect}, so we have $X \in (\mc F\land \mc G)_\alpha(dX)$.
By Remark \ref{remark: simple C representations}, $X$ has a simple $(\mc F\land \mc G)_\alpha$-representation in $dX$ -- that is, there exists an $l(T_\alpha)$-scheme $\mc H$ in $dX$, such that $X = \mc H(\emptyset)$, and we have $\mc H(t)\in (\mc F \land \mc G)(dX)$ for each $t\in l(T_\alpha)$.

Let $t\in l(T_\alpha)$ and denote $\mc H(t)$ as $\mc H(t) = F_t \cap G_t$, where $F_t$ is closed in $dX$ and $G_t$ is open in $dX$.
Since $X$ is hereditarily Lindelöf, the intersection $X\cap G_t$ is Lindelöf.
Applying Lemma \ref{lemma: fsd spaces are Fs separated} to ``$L$''$=X\cap G_t$ and ``$K$''$=cX\setminus G_t$, we get a set $H_t \in \mc F_\sigma(dX)$ s.t. $X \cap G_t \subset H_t \subset G_t$.
Denote $\mc H'(t) := F_t \cap H_t$.
$\mc H'$ is an $l(T_\alpha)$-scheme in $dX$, such for each $t\in l(T_\alpha)$, we have $\mc H'(t) \in \mc F_\sigma(dX)$ and $X\cap \mc H(t) \subset \mc H'(t) \subset \mc H(t)$.

It follows that $\mc H'(\emptyset) \in (\mc F_\sigma)_\alpha(dX)$ (Remark \ref{remark: simple C representations}) and $\mc H'(\emptyset) = \mc H(\emptyset) = X$ (Lemma \ref{lemma: leaf scheme making smaller}).
Since $(\mc F_\sigma)_1(\cdot)$ coincides with $\mc F_1(\cdot)$, it follows that $(\mc F_\sigma)_\alpha(dX) = \Fa(dX)$.
This shows that $X \in \Fa(dX)$.
\end{proof}

\subsection{Local Complexity and Spaces with No Universal Representation}
	\label{section: universal representation}

It is evident that only a~space which is absolutely $\Fa$ stands a~chance of having a~universal (simple) $\Fa$-representation. We will show that even among such spaces, having a~universal representation is very rare.
(Obviously, this does not include spaces which are $F_\sigma$ in some compactification -- such spaces are $\sigma$-compact, and the~representation $X = \bigcup_n K_n$ works no matter where $X$ is embedded.)
To obtain this result, we first discuss the~notion of local complexity.

Let $X$ be a~$\mc K$-analytic space, $Y$ a~topological space containing $X$, and $y\in Y$.
Consider open neighborhoods of $y$ in $Y$, and the~corresponding sets of the~form $\overline{U}^Y\!\! \cap X$.
Clearly, we have $\Compl{\overline{U}^Y \!\! \cap X}{Y} \leq \Compl{X}{Y}$.
Moreover, if $U\subset V$ are two neighborhoods of $y$, then $\Compl{\overline{U}^Y \!\! \cap X}{Y}$ will be smaller than $\Compl{\overline{V}^Y \!\! \cap X}{Y}$.
It follows that there exists a~limit of this decreasing sequence (or rather, net) of ordinals.
And since the~ordinal numbers are well-ordered, the~limit is attained for some neighborhood $V$ of $y$ (and thus for all $U\subset V$ as well).
We shall call the~limit the~local complexity of $X$ in $Y$ at $y$:

\begin{definition}[Local complexity] \label{definition: local complexity}
Let $X\in \mc F_{\omega_1}(Y)$ and $y\in Y$. The \emph{local complexity of $X$ in $Y$ at $y$} is the~ordinal
\[ \textnormal{Compl}_y(X,Y) := \min \left\{ \Compl{\overline{U}^Y\!\!\cap X}{Y} | \ 
 U \textnormal{ is a~neighborhood of $y$ in $Y$} \right\}. \]
We define the \emph{local complexity of $X$ in $Y$} as
\begin{equation}\label{equation:local_complexity}
\textnormal{Compl}_\textnormal{loc}(X,Y) := \sup_{x\in X} \textnormal{Compl}_x(X,Y) .
\end{equation}
We also set $\textnormal{Compl}_{\overline{\textnormal{loc}}}\,(X,Y) := \sup_{y\in Y} \textnormal{Compl}_y(X,Y)$.
\end{definition}

\noindent Note that for $x\in X$, we can express the~local complexity of $X$ in $Y$ at $x$ as
\[ \textnormal{Compl}_x(X,Y) = \min \left\{ \Compl{\overline{U}^X\!\!}{Y} | \ 
 U \textnormal{ is a~neighborhood of $x$ in $X$} \right\}. \]
 Also, if the local complexity of $X$ in $Y$ is a non-limit ordinal, then the supremum in \eqref{equation:local_complexity} is in fact a maximum.

We have the~following relation between the~local and ``global'' complexity:

\begin{lemma}[Complexity and local complexity]\label{lemma: local and global complexity}
For any $\mc K$-analytic space $X$ and any $Y\supset X$, we have 
\begin{equation}\label{equation:loc_and_glob_compl}
\textnormal{Compl}_{\textnormal{loc}}\,(X,Y) \leq \textnormal{Compl}_{\overline{\textnormal{loc}}}\,(X,Y) \leq \textnormal{Compl}(X,Y) \leq \textnormal{Compl}_\textnormal{loc}(X,Y) + 1 .
\end{equation}
Moreover, if either $\textnormal{Compl}_\textnormal{loc}(X,Y)$ is odd, or $\textnormal{Compl}(X,Y)$ is even, we get
\[ \textnormal{Compl}_{\textnormal{loc}}\,(X,Y) = \textnormal{Compl}_{\overline{\textnormal{loc}}}\,(X,Y) = \textnormal{Compl}(X,Y) .\]
When $Y$ is compact, we have
\[ \textnormal{Compl}(X,Y) = \textnormal{Compl}_{\overline{\textnormal{loc}}}\,(X,Y) = \max_{y\in Y} \textnormal{Compl}_y(X,Y) .\]
\end{lemma}

\begin{proof}
Let $X$ be a $\mc K$-analytic space and $Y\supset X$, and denote $\gamma := \textnormal{Compl}_{\textnormal{loc}}\,(X,Y)$ and $\alpha := \textnormal{Compl}(X,Y)$.

The first inequality in \eqref{equation:loc_and_glob_compl} is trivial.
The second follows from the~fact that the~complexity of $\overline{U}^Y\!\! \cap X$ in $Y$ is at most the~maximum of the~complexity of $\overline{U}^Y\!\!$ in $Y$ (that is, $0$) and $\Compl{X}{Y}$.
The last one immediately follows from the~following claim:

\begin{claim}
$X$ can be written as a countable union of $\mc F_\gamma$-subsets of $Y$.
\end{claim}

To prove the~claim, assume that for each $x\in X$ there exists an open neighborhood $U_x$ of $x$ in $X$ which satisfies $\overline{U_x}^X \in \mc F_{\gamma}(Y)$.
Since $X$ is $\mc K$-analytic, it is in particular Lindelöf (\cite[Proposition\,3.4]{kkakol2011descriptive}).
Let $\{U_{x_n} | \ n\in\omega \}$ be a~countable subcover of $\{ U_x | \ x\in X \}$.
It follows that $X = \bigcup_{n\in\omega} U_{x_n} = \bigcup_{n\in\omega} \overline{U_{x_n}}^X$ is a~countable union of $\mc F_{\gamma}(Y)$ sets, and hence $X \in (\mc F_\gamma(Y) )_\sigma $.

Next, we prove the ``moreover'' part of the lemma.
When $\gamma$ is odd, we have $(\mc F_\gamma(Y))_\sigma = \mc F_\gamma$, so the claim yields $\textnormal{Compl}(X,Y) \leq \gamma$ (and the conclusion follows from (i)).

Suppose that $\alpha$ is even and assume for contradiction that $\gamma < \alpha$.
The claim implies that $X$ belongs to the additive $\mc F$-Borel class $(\mc F_\gamma(Y) )_\sigma$. Since we have $\gamma < \alpha$ and the class $\Fa(Y)$ is multiplicative, it follows that $(\mc F_\gamma(Y) )_\sigma \subsetneq \Fa(Y)$. This contradicts the definition of $\textnormal{Compl}(X,Y)$ (which states that $\alpha$ is the smallest ordinal satisfying $X\in \Fa(Y)$).

Lastly, suppose that $Y$ is compact and denote $\eta := \textnormal{Compl}_{\overline{\textnormal{loc}}}\,(X,Y)$.
For each $y\in Y$, let $U_y\subset Y$ be an open neighborhood of $y$ in $Y$ s.t. $\overline{U}^Y\cap X$ is an $\mc F_\eta$-subset of $Y$.
Since $Y$ is compact, we have $Y = U_{y_0} \cup \dots \cup U_{y_k}$ for some $y_i \in Y$, $k\in\omega$.
It follows that $X$ can be written as $X = \bigcup_{i=0}^k (\overline{U_{y_i} }^Y \cap X)$ of finitely many $\mc F_\eta$-subsets of $Y$.
This shows that $X\in \mc F_\eta(Y)$, and thus $\textnormal{Compl}(X,Y) \leq \gamma$. Since the supremum in $\textnormal{Compl}_{\overline{\textnormal{loc}}}\,(X,Y)$ is equal to the supremum of finitely many numbers $\textnormal{Compl}_{y_i}\,(X,Y)$, it is in fact a maximum.
\end{proof}

Note that \eqref{equation:loc_and_glob_compl} is optimal in the sense that when $\textnormal{Compl}(X,Y)$ is odd, it could be equal to either of the numbers $\textnormal{Compl}_{\textnormal{loc}}\,(X,Y)$ and $\textnormal{Compl}_\textnormal{loc}(X,Y) + 1$.
Indeed, consider the examples $X_1 := (0,1) \subset [0,1] =: Y$ and $X_2 = \Q \cap [0,1] \subset [0,1] = Y$.

In Section~\ref{section: broom space properties}, we construct absolute $\Fa$ spaces $\mathbf{T}_\alpha$, each of which has a~special point $\infty \in \mathbf{T}_\alpha$. For any $Y\supset \mathbf{T}_\alpha$, the~local complexity of $\mathbf{T}_\alpha$ in $Y$ is 0 at every $x\in \mathbf{T}_\alpha \setminus \infty$ and 1 at every $y \in Y\setminus \mathbf{T}_\alpha$. (This shall immediately follow from the~fact that the only non-isolated point of $X$ is $\infty$ and any closed subset of $\mathbf{T}_\alpha$ which does not contain $\infty$ is at most countable.)
This in particular shows that in general, it is not useful to consider the~supremum of $\sup_y \textnormal{Compl}_y(X,Y)$ over $Y\setminus X$.

We will use the~following local variant of the~notion of descriptive ``hardness'':

\begin{definition}[local $\Fa$-hardness]\label{definition: Fa hardness}
Let $X\subset Y$ be topological spaces s.t. $X\in \mc F_{\omega_1}(Y)$.
For $\alpha\leq \omega_1$, $X$ is said to be
\begin{itemize}
\item \emph{$\Fa$-hard in $Y$ at $y$}, for some $y\in Y$, if we have $\textnormal{Compl}_y\left(X,Y\right) \geq \alpha$;
\item \emph{locally $\Fa$-hard in $Y$} if it is $\Fa$-hard in $Y$ at every $x\in X$.
\end{itemize}
\end{definition}

\noindent Note that $X$ is locally $\Fa$-hard in $Y$ if and only if we have $\overline{U}^X \notin \mc F_{<\alpha}(Y)$ for every open subset $U$ of $X$ (or equivalently, when $F\notin \mc F_{<\alpha}(Y)$ holds for every regularly closed\footnote{Recall that a~set $F$ is regularly closed if it is equal to the~closure of its interior.} subset of $X$).

Note that if $X$ is locally $\Fa$-hard in $Y$, then it is $\Fa$-hard in $Y$ at every point of $\overline{X}^Y\setminus X$.

\begin{lemma}[Density of ``too nice'' sets]\label{lemma: Fa hardness and density}
Let $X\subset Y$ be topological spaces. Suppose $X\subset H \in \mc F_{<\alpha}(Y)$ holds for some $H$ and $\alpha \leq \omega_1$.
\begin{enumerate}[(i)]
\item If $X$ is dense in $Y$ and locally $\Fa$-hard in $Y$, then $H\setminus X$ is dense in $Y$.
\item More generally, if $X$ is $\Fa$-hard in $Y$ at some $y\in Y$, then $y\in \overline{H\setminus X}^Y$.
\end{enumerate}
\end{lemma}

In particular, (i) says that if a~space $X$ is locally hard in $Y$, then any set $H\supset X$ of lower complexity than $X$ must be much bigger than $X$, in the~sense that the~closure of $H\setminus X$ is the~same as that of $X$.

\begin{proof}
(i): By density of $X$ in $Y$ (and the~remark just above this lemma), (i) follows from (ii).

(ii): Suppose we have $X\subset H \subset Y$, $H \in \mc F_{<\alpha}(Y)$, and that $X$ is $\Fa$-hard in $Y$ at some $y\in Y$.
Let $U$ be an open neighborhood of $y$ in $Y$.
Let $V$ be some open neighborhood of $y$ in $Y$ which satisfies $\overline{V}^Y \subset U$.
Since $\textnormal{Compl}_y\left(X,Y\right) \geq \alpha$, we have $\overline{V}^Y \cap X \notin \mc F_{<\alpha}(Y)$.
On the~other hand, the~assumptions imply that
\begin{equation}\label{equation: Fa hardness and density}
\overline{V}^Y \!\! \cap X \subset \overline{V}^Y \!\! \cap H \in \mc F_{<\alpha}(Y) .
\end{equation}
Since the~two sets in \eqref{equation: Fa hardness and density} have different complexities, they must be distinct. In particular, the~intersection $\overline{V}^Y \!\! \cap (H\setminus X) \subset U \cap (H\setminus X)$ is non-empty.
Since $U$ was arbitrary, this shows that $y$ is in the~closure of $H\setminus X$.
\end{proof}

Before proceeding further, make the~following simple observation:

\begin{lemma}[$G_\delta$-separation of $\mc F$-Borel sets]\label{lemma: G delta separation}
Let $X$ be a~$\mc F$-Borel subset of $Y$. Then for every $y\in Y\setminus X$, there is a~$G_\delta$ subset $G$ of $Y$ which satisfies $x\in G \subset Y \setminus X$.
\end{lemma}

\begin{proof}
We proceed by transfinite induction over the~complexity of $X$ in $Y$.
When $X$ is closed in $Y$, the~conclusion holds even with the open set $G:= Y \setminus X$.

Suppose that $X=\bigcup_n X_n$ and $y\in Y \setminus X$. If each $X_n$ satisfies the~conclusion -- that is, if we have $y\in G_n \subset Y \setminus X_n$ for some $G_\delta$ subsets $G_n$ of $Y$ -- then the~set $\bigcap_n G_n$ is $G_\delta$ in $Y$ and we have
\[ y \in \bigcap_n G_n \subset \bigcap_n (Y \setminus X_n) = Y \setminus \bigcup_n X_n = Y \setminus X .\]

Suppose that $X = \bigcap_n X_n$ and let $y\in Y \setminus X$. Let $n_0\in\omega$ be s.t. $y Y \setminus X_n$. By the induction hypothesis, there is some $G_\delta$ set $G$ which satisfies the~conclusion for $X_{n_0}$. Clearly, $G$ satisfies the~conclusion for $X$ as well.
\end{proof}

We are now ready to prove that the~existence of universal representation is rare.
Recall that a point $x$ is a~cluster point of a sequence $(x_n)_n$ when each neighborhood of $x$ contains $x_n$ for infinitely many $n$-s.

\begin{proposition}[Non-existence of universal representations]
	\label{proposition: non-existence of universal repre}
Let $X$ be a~$\mc K$-analytic space which is not $\sigma$-compact.
Suppose that $X$ has a~compactification $cX$, such that one of the~following conditions holds for some \emph{even} $\alpha < \omega_1$:
\begin{enumerate}[(a)]
\item $X$ is locally $\Fa$-hard in $cX$;
\item there is a~sequence $(x_n)_{n\in\omega}$ in $X$ with no cluster point in $X$, such that $X$ is $\Fa$-hard in $cX$ at each $x_n$;
\item $X$ is $\Fa$-hard in $cX$ at some point of $cX\setminus X$.
\end{enumerate}
Then $X$ does not have a~universal simple $\Fa$-representation.
\end{proposition}

\begin{proof}
Note that both (a) and (b) imply (c).
Indeed, by the~remark between Definition~\ref{definition: Fa hardness} and Lemma~\ref{lemma: Fa hardness and density}, (a) implies that $X$ is $\Fa$-hard at each point of the~(non-empty) set $cX\setminus X$.
If (b) holds, then $(x_n)_n$ must have some cluster point $y$ in $cX\setminus X$. And since each neighborhood of $y$ is a~neighborhood of some $x_n$, $X$ must be $\Fa$-hard in $cX$ at $y$:
\[ \textnormal{Compl}_y(X,cX) \geq \inf_{n\in\omega} \textnormal{Compl}_{x_n}(X,cX) \geq \alpha .\]

It remains to show that (c) implies the~non-existence of universal $\Fa$-representation of $X$.
Suppose that $X$ is $\Fa$-hard in $cX$ at some $y\in cX\setminus X$.
By Lemma~\ref{lemma: G delta separation}, there is some $G_\delta$ subset $G$ of $Y$, such that $y\in G \subset cX \setminus X$.
Since $cX$ is regular, we can assume that $G = \bigcap_n G_n$, where $G_n$ are open subsets of $cX$ which satisfy $\overline{G_{n+1}} \subset G_n$.

Let $\mc H$ be a~``candidate for a~universal simple representation of $X$'' -- that is, an $l(T_\alpha)$-scheme in $X$ satisfying $X = \overline{\mc H}^{cX}\!\!(\emptyset)$.
We shall prove that $\mc H$ is not a~universal representation by constructing a~compactification $dX$ in which $\overline{\mc H}^{dX}\!\!$ is not a representation of $X$ (that is, we have $\overline{\mc H}^{dX}\!\!(\emptyset) \neq X$).

Since $\alpha$ is even, we get $\overline{\mc H}^{cX}\!\!(\emptyset) = \bigcap_n \overline{\mc H}^{cX}\!\!(n)$.
By Lemma~\ref{lemma: complexity of S T} (or more precisely, by \eqref{equation: complexity of H t} from the~proof of Lemma~\ref{lemma: complexity of S T}), the~sets $\overline{\mc H}^{cX}\!\!(n)$ satisfy $X \subset \overline{\mc H}^{cX}\!\!(n) \in \mc F_{<\alpha}(Y)$.
In this setting, we can use Lemma~\ref{lemma: Fa hardness and density}\,(ii) to get $y\in  \overline{\overline{\mc H}^{cX}\!\!(n) \setminus X}^{cX}$.
In particular, there exists some $x_n \in \left( \overline{\mc H}^{cX}\!\!(n) \setminus X \right) \cap G_n$.

Set $K := \overline{\{ x_n | \ n\in\omega\} }^{cX}$.
Since we have $\bigcap_n \overline{G_n}^{cX} = \bigcap_n G_n \subset cX \setminus X$, it follows that $K$ is a~compact subset of $cX$ which is disjoint with $X$.
Let $dX := cX/_K$ be the~compact space obtained by ``gluing together'' the~points of $K$ (formally, we define an equivalence $\sim$ on $cX$ as
\[ x \sim y \iff ( x=y ) \lor ( x,y \in K ) ,\]
and define $dX$ as the~corresponding quotient of $cX$).
We can identify $dX$ with the~set $\{[K]\} \cup cX \setminus X$, where a~set containing $[K]$ is open in $dX$ if and only if it is of the~form $\{[K]\} \cup U \setminus K$ for some open subset of $U\supset K$ of $cX$.
Since $K$ is disjoint from $X$, $dX$ is a~compactification of $X$.

We finish the~proof by showing $[K] \in \overline{\mc H}^{dX}\!(\emptyset) \setminus X$.
Let $n\in\omega$.
The topology of $dX$ is such that for every $A\subset X$, we have $\overline{A}^{cX} \ni x_n \implies \overline{A}^{dX} \ni [K]$. In particular, we obtain the~following implication for any leaf $t\in l(T_\alpha)$:
\begin{equation*}
\overline{\mc H}^{cX}\!(t) = \overline{\mc H(t)}^{cX} \ni x_n \implies
\overline{\mc H}^{dX}\!(t) = \overline{\mc H(t)}^{dX} \ni [K] .
\end{equation*}
It follows (by \eqref{equation: H t extensions} and induction over $r_l(T^t_\alpha)$) that the~following implication holds for every $t\in T_\alpha$
\begin{equation*}
\overline{\mc H}^{cX}\!(t) \ni x_n \implies \overline{\mc H}^{dX}\!(t) \ni [K] .
\end{equation*}
And since $\overline{\mc H}^{cX}\!\!(n)$ contains $x_n$, we get $[K] \in \overline{\mc H}^{dX}\!(n)$.
Because $n$ was arbitrary, this concludes the~proof:
\[ \overline{\mc H}^{dX}\!(\emptyset)
\overset{\alpha \text{ is}}{\underset{\text{even}}=}
\bigcap_n \overline{\mc H}^{dX}\!(n) \ni [K] .\]
\end{proof}

In particular, we obtain the~following corollary for Banach spaces:

\begin{corollary}[No universal representations for Banach spaces]\label{corollary: Banach spaces have no universal representation}
A Banach space $X$ is either reflexive (and therefore weakly $\sigma$-compact), or $(X,w)$ has no universal $\fsd$ representation.
\end{corollary}

\begin{proof}
Let $X$ be a~non-reflexive Banach space.
By Proposition~\ref{proposition: existence of simple representations}\,(a), it is enough to note that $X$ is locally $\fsd$-hard in some compactification.
Since $(X^{\star\star},w^\star)$ is $\sigma$-compact, being locally $\fsd$-hard in $(X^{\star\star},w^\star)$ implies being locally $\fsd$-hard in, for example, $\beta (X^{\star\star},w^\star)$.
By the~remark between Definition~\ref{definition: Fa hardness} and Lemma~\ref{lemma: Fa hardness and density}, it suffices to show that no regularly closed subset of $(X,w)$ is $F_\sigma$ in $(X^{\star\star},w^\star)$.
In this particular setting, this means showing that weakly regularly closed sets in $X$ are not weakly $\sigma$-compact.

For contradiction, assume that there is a~sequence of weakly compact sets $K_n \subset X$, such that $\bigcup_n K_n$ has got a~non-empty interior in weak topology.
In particular, the~interior contains a~norm-open set.
Consequently, Baire category theorem yields $n_0\in \omega$ such that $K_{n_0}$ contains an open ball $U_X(x,\epsilon)$.
However, this would imply that the~closed ball $B_X(x,\frac{\epsilon}{2})$ -- a~weakly closed subset of $K_{n_0}$ -- is weakly compact.
This is only true for reflexive spaces.
\end{proof}
\section{Regular Representations}\label{section: regular representations}

In this section, we introduce representations which are more structured than the~simple representations from Section~\ref{section: simple representations}.
In Section~\ref{section:Suslin_schemes}, we recall the~notion of a Suslin scheme and some related results. Section~\ref{section: R T sets} introduces the~concept of a regular representation and investigates its basic properties. In Section~\ref{section: Suslin scheme rank}, we give an alternative description of regular representations. This yields a criterion for estimating $\Compl{X}{Y}$, which will in particular be useful in Section~\ref{section: complexity of brooms}. Section~\ref{section: existence of regular representation} is optional, and justifies the~concept of regular representation by proving their existence.

\subsection{Suslin Schemes}\label{section:Suslin_schemes}

A tool relevant to $\mc K$-analytic sets and $\mc F$-Borel complexities is the~notion of complete sequence of covers.

\begin{definition}[Complete sequence of covers]\label{definition: complete seq. of covers}
Let $X$ be a topological space.
\emph{Filter} on $X$ is a family of subsets of $X$, which is closed with respect to supersets and finite intersections and does not contain the~empty set.
A point $x\in X$ is said to be an \emph{accumulation point} of a filter $\mc F$ on $X$, if each neighborhood of $x$ intersects each element of $\mc F$.

A sequence $\left( \mc C_n \right)_{n\in\mathbb N}$ of covers of $X$ is said to be \emph{complete}, if every filter which intersects each $\mc C_n$ has an accumulation point in $X$.
\end{definition}

A notion related to Suslin sets is that of a Suslin scheme:

\begin{definition}[Suslin schemes and sets]\label{definition: Suslin schemes}
By a \emph{Suslin scheme in $Y$} we will understand a family $\mc C =\{ C(s) | \ s\in\seq \} \subset \mc P(Y)$ which satisfies the~following monotonicity condition:
$$\left(\forall s,t\in\seq\right) : t\sqsupset s \implies C(t)\subset C(s) .$$
A \emph{Suslin operation} is the~mapping
\[ \mc A : \mc C \mapsto \mc A(\mc C) := \bigcup_{\sigma\in\baire} \bigcap_{n\in\omega} C(\sigma|n) .\]

Let $\mc C$ be a Suslin scheme in $Y$ and $X\subset Y$. We say that $\mc C$
\begin{itemize}
\item is \emph{closed} in $Y$ if $\mc C \subset \mc F_0 (Y)$;
\item \emph{covers $X$} if it satisfies $\mc A(\mc C) \supset X$;
\item is a Suslin scheme \emph{on $X$} if it is actually a Suslin scheme in $X$ which covers $X$;
\item is \emph{complete on $X$} if it is Suslin scheme on $X$ and $\left( \mc C_n \right)$ is a complete sequence of covers of $X$, where
\[ \mc C_{n} := \left\{ C(s) | \ s\in\omega^n \right\} .\]
\end{itemize}
\end{definition}

Note that a subset of $Y$ is Suslin-$\mc F$ in $Y$ if it is the~image under Suslin operation of some closed Suslin scheme in $Y$.
The existence of complete Suslin schemes is guaranteed by the~following result.

\begin{proposition}[Existence of complete Suslin schemes]
\label{proposition: K analytic spaces have complete Suslin schemes}
On any $\mc K$-analytic space, there is a complete Suslin scheme.
\end{proposition}

\begin{proof}
By \cite[Theorem 9.3]{frolik1970survey}, any $\mc K$-analytic space has a complete sequence of countable covers. Enumerate the~$n$-th cover as $\mc C_n = \{ C_n^k | \ k\in\omega \}$. Denoting $C(\emptyset) := X$ and $C(s) := C_0^{s(0)} \cap C_1^{s(1)} \cap \dots \cap C_k^{s(k)}$, we obtain a Suslin scheme (in $X$) covering $X$. Moreover, $\mc C$ is easily seen to be complete on $X$.
\end{proof}

The main reason for our interest in complete Suslin schemes is their the~following property:

\begin{lemma}\label{lemma: complete suslin schemes in super spaces}
Let $\mc C$ be a Suslin scheme on $X$. If $\mc C$ is complete on $X$, then $\mc A(\overline{\mc C}^Y\!\!) = X$ holds for every $Y\supset X$.
\end{lemma}

\begin{proof}
This holds, for example, by \cite[Lemma 4.7]{kalenda2018absolute}.
\end{proof}

\subsection{Definition of a Regular Representation} \label{section: R T sets}

The basic tool for the~construction of regular representations will be the~following mappings from trees to sequences:

\begin{definition}[Admissible mappings]\label{definition: admissible mapping}
Let $T\in\Tr$ be a tree and $\varphi : T\rightarrow \seq$ a mapping from $T$ to the~space of finite sequences on $\omega$. We will say that $\varphi$ is \emph{admissible}, if it satisfies
\begin{enumerate}[(i)]
\item $(\forall s,t\in T): s \sqsubset t \implies \varphi (s) \sqsubset \varphi(t)$
\item $(\forall t = (t_0,\dots,t_k)\in T): \left| \varphi (t) \right| = t(0)+\dots+t(k)$.
\end{enumerate}
\end{definition}

The next lemma completely describes how these mappings look like.
(Recall from Section~\ref{section: sequences} that each non-empty tree $T$ can be rewritten as $T = \{ \emptyset \} \cup \bigcup \{ m\ext T^{(m)} \, | \ m\in \ims{T}{\emptyset} \}$, where $T^{(m)}$ is defined as $\{ t' \in \seq \, | \ m\ext t' \in T \}$.)

\begin{lemma}[Construction of admissible mappings] \label{lemma: construction of admissible mappings}$\ $ 
\begin{enumerate}[(i)]
\item The only admissible mapping with range $\{\emptyset\}$ is the~mapping $\varphi : \emptyset \mapsto \emptyset$.
\item A mapping $\varphi : T \rightarrow \seq$ is admissible if and only if it is a restriction of some admissible mapping $\psi : \seq \rightarrow \seq $.
\item For $\{\emptyset\}\neq T\in \Tr$, a mapping $\varphi : T\rightarrow \seq$ is admissible if and only if it is defined by the~formula
\[ \varphi(t) =
\begin{cases}
\emptyset, 					& \textrm{for } t = \emptyset \\
s_m\ext \varphi_m(t'), 	& \textrm{for } t=m\ext t' \ \ \ (\text{where } (m)\in \ims{T}{\emptyset}, t'\in T^{(m)})
\end{cases} \]
for some sequences $s_m\in\omega^m$ and admissible mappings $\varphi_m : T^{(m)}\rightarrow \seq$.
\end{enumerate}
\end{lemma}
\begin{proof}
$(i)$: By definition of an admissible mapping, we have $|\varphi(\emptyset)|=0$, so $\varphi(\emptyset)$ must be equal to the~empty sequence $\emptyset$.

$(ii)$: Let $T\in \Tr$. From the~definition of an admissible mapping, we see that if $\psi: \seq \rightarrow \seq$ is admissible, then $\psi|_T : T \rightarrow \seq$ is also admissible. Moreover, any admissible mapping $\varphi: T \rightarrow \seq$ is of this form. To see this, fix any $\sigma\in \baire$.
For every $s\in \seq $, there exists some $t\in T$, such that $s=t\ext h$ holds for some $h\in\seq$.
Denoting by $h=(h(0),\dots,h(k))\in\seq$ the~sequence corresponding to the~longest such $t\in T$, we set $\psi(s):=\varphi(t)\ext \sigma|{h(0)+\dots+h(k)}$. Clearly, $\psi$ is admissible and coincides with $\varphi$ on $T$.

$(iii)$:
Denote $M := \ims{T}{\emptyset}$.
Firstly, note that if $\varphi:T\rightarrow \seq$ is admissible, then $\varphi(\emptyset)=\emptyset$, for every $m\in M$ we have $s_m := \varphi(m) \in \omega^m$ and for other elements of the~tree, we have $\varphi(t)=\varphi(m\ext t')\sqsupset s_m$ for some $m\in M$ and $t'\in T^{(m)}$. Denoting $\varphi(m\ext t') =: s_m \ext \varphi_m(t')$, we get mappings $\varphi_m : T^{(m)}\rightarrow \seq$ which are easily seen to be admissible.

On the~other hand, it is not hard to see that if mappings $\varphi_m : T^{(m)} \rightarrow \seq$ are admissible and the~sequences $s_m$ for $m\in M$ are of the~correct length, then the~formula in $(iii)$ defines an admissible mapping.
\end{proof}

Before defining the~desired regular representations, we introduce the~following $\mathbf{R}_T$-sets (of which regular representations will be a special case, for particular choices of $\mc C$ and $T$):

\begin{definition}[Regular representation] \label{definition: R T sets}
Let $T$ be a tree and $\mc C$ a Suslin scheme. We define
\[ \mathbf{R}_T(\mc C):=\left\{ x \in C(\emptyset) \ | \ \left(\exists \varphi : T\rightarrow \seq \textrm{ admissible} \right)\left( \forall t\in T \right) : x\in C \left(\varphi (t)\right) \right\} . \]
Let $X\subset Y$ be topological spaces. The pair $(\mc C,T)$ is said to be a \emph{regular representation of $X$ in $Y$} if $\mc C$ is a Suslin scheme in $Y$ and we have $X = \mathbf{R}_T(\mc C)$.
\end{definition}

\noindent When $x$ satisfies $x \in \bigcap_T C(\varphi(t))$ for some admissible $\varphi:T\rightarrow \seq$, we shall say that \emph{$\varphi$ witnesses that $x$ belongs to $\mathbf{R}_T(\mc C)$}.
We also need a technical version of the~notation above. For $h\in\seq$, we set
\[ \mathbf{R}^h_T(\mc C):=\left\{ x \in C(\emptyset) \ | \ \left(\exists \varphi : T\rightarrow \seq \textrm{ admissible} \right)\left( \forall t\in T \right) : x\in C \left(h\ext\varphi (t)\right) \right\} .\]

The technical properties of $\mathbf{R}_T$-sets are summarized by the~following lemma:

\begin{lemma}[Basic properties of $\mathbf{R}_T$-sets] \label{lemma: basic properties of R T sets}
Let $\mc C$ be a Suslin scheme, $S,T\in\Tr$ and $h\in\seq$.
\begin{enumerate}[(i)]
\item Suppose that there is a mapping $f:T \rightarrow S$ satisfying \label{case: R T and embeddings}
	\begin{itemize}
	\item $\left( \forall t_0, t_1 \in T \right): t_0 \sqsubset t_1 \implies f(t_0) \sqsubset f(t_1)$,
	\item $\left(\forall t = (t(0), \dots, t(k)) \in T \right) : f(t)(0)+\dots + f(t)(k) \geq t(0) + \dots + t(k)$.
	\end{itemize}
Then $\mathbf{R}_S(\mc C) \subset \mathbf{R}_T(\mc C)$.
\item In particular, we have $\mathbf{R}_S(\mc C) \subset \mathbf{R}_T(\mc C)$ whenever $S\supset T$. \label{case: R T and smaller trees}
\item We have $\mathbf{R}^h_{\{\emptyset\}}(\mc C) = C(h)$ and the~following recursive formula: \label{case: R T formula}
\begin{equation*}
 \mathbf{R}_T^h(\mc C) \ = \ \bigcap_{(m) \in \ims{T}{\emptyset}} \bigcup_{s_m\in\omega^m} \mathbf{R}_{T^{(m)}}^{h\ext s_m}(\mc C) .
\end{equation*}
\item If $T\in\textrm{IF}$ has a branch with infinitely many non-zero elements, then $\mathbf{R}_T(\cdot)$ coincides with the~Suslin operation $\mc A(\cdot)$. \label{case: R T and IF}
\end{enumerate}
\end{lemma}

We can view the~mapping $f:T\rightarrow S$ from Lemma~\ref{lemma: basic properties of R T sets}\,\eqref{case: R T and embeddings} as an embedding of the~tree $T$ into $S$.
In this light, \eqref{case: R T and embeddings} implies that ``larger tree means smaller $\mathbf{R}_{(\cdot)}(\mc C)$'' and for ``equivalent'' trees, the~corresponding sets $\mathbf{R}_{(\cdot)}(\mc C)$ coincide.

\begin{proof}
Let $\mc C$, $T$ and $h$ be as above. For simplicity of notation, we will assume that $h=\emptyset$ and therefore we will only work with the~sets $\mathbf{R}_T(\mc C)$. However, the~proof in the~general case is exactly the~same as for $h=\emptyset$.

\eqref{case: R T and embeddings}: Let $f:T\rightarrow S$ be as in the~statement and suppose that $\varphi : S \rightarrow \seq$ witnesses that $x\in \mathbf{R}_S(\mc C)$.
Using the~properties of $f$, we get that the~following formula defines an admissible mapping:
\[ \psi : t\in T \mapsto \varphi(f(t)) |_{t(0)+\dots + t(|t|-1)} . \]
Using monotonicity of $\mc C$, we prove that $\psi$ witnesses that $x$ belongs to $\mathbf{R}_T(\mc C)$:
\begin{equation*} \begin{split}
x \in 		& \bigcap_{s\in S} C(\varphi(s)) \overset{f(T)\subset S}\subset
				\bigcap_{t\in T} C \left(\varphi(f(t))\right)
				\overset{\textnormal{mon.}}{\underset{\text{of }\mc C}\subset}
				\bigcap_{t\in T} C \left(\varphi(f(t)) |_{\sum t(k)}\right)
				\overset{\textnormal{def.}}{\underset{\text{of }\psi}\subset}				
				\bigcap_{t\in T} C(\psi(t)) .
\end{split} \end{equation*}
(iii): The identity $\mathbf{R}_{\{\emptyset\}}(\mc C) = C(\emptyset)$ follows from Lemma~\ref{lemma: construction of admissible mappings}~$(i)$.
For $\{\emptyset\} \subsetneq T$, we can rewrite $T$ as
\[  T = \{\emptyset\}\cup \bigcup \{ m\ext T^{(m)} | \ m\in M\} ,\]
where $M := \ims{T}{\emptyset}$.

Let $x\in C(\emptyset)$.
By definition of $\mathbf{R}_T$, we have $x \in \mathbf{R}_T(\mc C)$ if and only if there exists a witnessing mapping -- that is, an admissible mapping $\varphi : T \rightarrow \seq$ which satisfies $x \in \bigcap_T C(\varphi(t))$.
Since $\mc C$ is monotone, this is further equivalent to the~existence of admissible $\varphi : T \rightarrow \seq$ s.t.
\[ x \in \bigcap_T C(\varphi(t)) = \bigcap_M \bigcap_{t' \in T^{(m)}} C(\varphi(m\ext t)) .\]
By Lemma~\ref{lemma: construction of admissible mappings}\,(iii), this is true if and only if there are sequences $s_m \in \omega^m$, $m\in M$, and admissible mappings $\varphi_m : T^{(m)} \rightarrow \seq$, $m\in M$, for which $x$ belongs to $\bigcap_{T^m} C(s_m \ext \varphi_m (t'))$.
This is precisely when, for each $m\in M$, there exist some mapping witnessing that $x \in \mathbf{R}^{s_m}_{T^{(m)}}(\mc C)$ holds for some $s_m \in \omega^m$.

(iv): This part immediately follows from (iii).

(v): Let $T$ be as in the~statement and $\mc C$ be a Suslin scheme. To show that $\mc A(\mc C) \subset \mathbf{R}_T(\mc C)$, it suffices to show that $\mc A(\mc C) \subset \mathbf{R}_{\seq}(\mc C)$ (by \eqref{case: R T and smaller trees}).
Let $x \in \bigcap_{m\in\omega} C(\sigma|m) \subset \mc A(\mc C)$ and define $\varphi : \seq \rightarrow \seq $ as
\[ \varphi(\, (t(0),\dots,t(k)) \,):= \sigma |_{t(0)+\dots+t(k)} .\]
Since $\varphi$ is admissible, it witnesses that $x$ belongs to $\mathbf{R}_{\seq}(\mc C)$.

For the~reverse inclusion, let $\nu$ be a branch of $T$ with $\sum_k \nu(k) = \infty$ and $\varphi:T\rightarrow \seq$ be an admissible mapping. The sequence $\varphi(\nu|_0) \sqsubset \varphi(\nu|_1) \sqsubset \varphi(\nu|_2) \sqsubset \dots$ eventually grows to an arbitrary length, which means that there exists $\sigma\in\baire$ such that for each $m\in\omega$ there is $k\in\omega$ such that $\sigma|m \sqsubset \varphi(\nu|_k)$. In particular, we have
\[ \bigcap_{t\in T} C(\varphi(t)) \subset \bigcap_{k\in\omega} C(\varphi(\nu|_k)) \subset \bigcap_{m\in\omega} C(\sigma|m) \subset \mc A (\mc C) . \]
\end{proof}

Since our goal is to study $\mc F$-Borel sets, we will focus on $\mathbf{R}_T$-sets corresponding to closed Suslin schemes. The following notation is handy for that purpose.

\begin{definition}[$\mathbf{R}_T(\mc C,Y)$-sets]
For any Suslin scheme $\mc C$, tree $T$, and a topological space $Y$, we define
\[ \mathbf{R}_T(\mc C,Y):=\mathbf{R}_T(\bar{\mc C}^Y) \textrm{, where } \bar{\mc C}^Y
 := \left\{ \overline{C(s)\cap Y}^Y | \ s\in\seq \right\} . \]
\end{definition}

We have the~following complexity estimate for $\mathbf{R}_T$-sets:

\begin{lemma}[Complexity of $\mathbf{R}_T$-sets] \label{lemma: complexity of R T}
Let $\mc C$ be a closed Suslin scheme in a topological space $Y$, and $T\in \textnormal{WF}$.
Denoting the~leaf-rank of $T$ as $r_l(T) = \lambda+n$ (where $\lambda$ is a limit ordinal or $0$ and $n\in\omega$), we have:
\begin{enumerate}[(i)]
\item $\mathbf{R}_T(\mc C) \in \mc F_{\lambda+2n}(Y)$;
\item If $\ims{T}{\emptyset}$ is finite, then $\mathbf{R}_T(\mc C) \in \mc F_{\lambda+2n-1}(Y)$.
\end{enumerate}
\end{lemma}

\noindent (The expression in (ii) is well-defined because if $\ims{T}{\emptyset}$ is finite, then $r_l(T)$ cannot be a limit ordinal.)

\begin{proof}
We proceed by induction over $r_l(T)$.
For $r_l(T) = 0$, we necessarily have $T= \{\emptyset\}$, which yields
\[ \mathbf{R}_T(\mc C) = \mathbf{R}_{\{\emptyset\}} (\mc C)
\overset{L\ref{lemma: construction of admissible mappings}(i)}= C(\emptyset) \in \mc F_0(Y) .\]

Let $\mc C$ be a closed Suslin scheme in $Y$, $T\in \textnormal{WF}$ a tree s.t. $r_i(T) = \lambda + n$.
Suppose, as an induction hypothesis, that for every closed Suslin scheme $\mc D$ in $Y$ and every $S\in \Tr$ with $r_l(S) < \lambda + n$,
we have $\mathbf{R}_S(\mc D) \in \mc F_\beta (Y)$ for some \emph{even} ordinal $\beta_S<\lambda + 2n$.
Denoting $M:=\ims{T}{\emptyset}$, Lemma~\ref{lemma: basic properties of R T sets}\,\eqref{case: R T formula} yields
\begin{equation}\label{equation: R T formula}
\mathbf{R}_T(\mc C) = \bigcap_{m\in M} \bigcup_{s\in \omega^m} \mathbf{R}^s_{T^{(m)}}(\mc C) .
\end{equation}
Let $m\in M$.
We have $r_l(T^{(m)}) < \lambda + n$, and $\mathbf{R}^s_{T^{(m)}}(\mc C)$ can be rewritten as $\mathbf{R}_{T^{(m)}}(\mc D)$, where $\mc D := \left( C(s\ext u) \right)_{u\in \seq}$ is a closed Suslin scheme in $Y$.
Using the~induction hypothesis, we get $\mathbf{R}^s_{T^{(m)}}(\mc C) \in \mc F_{\beta_m}(Y) = \mc F_{<\beta_m+1}(Y)$ for some even $\beta_m < \lambda +2n$.
Since $\beta_m +1$ is odd, we get
\[ \bigcup_{s\in\omega^m} \mathbf{R}^s_{T^{(m)}}(\mc C) \in (\mc F_{<\beta_m +1}(Y))_\sigma = \mc F_{\beta_m+1}(Y) \subset \mc F_{<\lambda +2n}(Y) .\]

(i): It follows that $\mathbf{R}_T(\mc C) \in (\mc F_{<\lambda +2n}(Y))_\delta = \mc F_{\lambda+2n}(Y)$.

(ii): When the~set $M$ is finite, we necessarily have $n\geq 1$.
It follows that each $\bigcup_{s\in\omega^m} \mathbf{R}^s_{T^{(m)}}(\mc C)$ belongs to $\mc F_{<\lambda +2n}(Y) = \mc F_{\lambda +2n-1}(Y)$.
Since the~intersection in \eqref{equation: R T formula} is finite, it follows that $\mathbf{R}_T(\mc C)$ belongs to $\mc F_{\lambda +2n-1}(Y)$ as well.
\end{proof}

It follows from Lemma~\ref{lemma: complexity of R T} (and Lemma~\ref{lemma: basic properties of R T sets}\,\eqref{case: R T and IF}) that for each $T\in \Tr$, there exists $\alpha_T \leq \omega_1$ such that $\mathbf{R}_T(\mc C,Y) \in \mc F_{\alpha_T}(Y)$ holds for every $\mc C$ and $Y$.
In this sense, each $\mathbf{R}_T(\cdot , Y)$ can be understood as an operator which maps Suslin schemes to $\mc F_{\alpha_T}$-subsets of $Y$.

It follows from (ii) in Lemma~\ref{lemma: basic properties of R T sets} that for the~purposes of studying the~$\mathbf{R}_T$-sets, many trees are in fact equivalent.
We therefore focus our attention on the~``canonical'' trees $T^c_\alpha$, $\alpha \leq \omega_1$.
In Proposition~\ref{proposition: existence of regular F alpha representations}, we show that this is sufficient -- every $\Fa$-subset $X$ of $Y$ can be written as
$X = \mathbf{R}_{T^c_\alpha}(\mc C,Y)$ for some $\mc C$.

\begin{definition}[$\mathbf{R}_\alpha$-sets]\label{definition: R alpha sets}
For a Suslin scheme $\mc C$ and $\alpha\leq \omega_1$, we denote
\[ \mathbf{R}_\alpha(\mc C):=\mathbf{R}_{T^\textrm{c}_\alpha}(\mc C) .\]
We define $\mathbf{R}_\alpha(\mc C,Y)$, $\mathbf{R}^h_\alpha(\mc C)$ and $\mathbf{R}^h_\alpha(\mc C,Y)$ analogously.
\end{definition}

As a special case of Lemma~\ref{lemma: complexity of R T}, we get that the~sets $\mathbf{R}_\alpha(\cdot,\cdot)$ are always $\Fa$:

\begin{proposition}[$\mathbf{R}_\alpha$-sets are $\Fa$]\label{proposition: complexity of R alpha sets}
For any Suslin scheme $\mc C$ and $\alpha \leq \omega_1$, we have $\mathbf{R}_\alpha (\mc C, Y) \in \Fa(Y)$ for any topological space $Y$.
\end{proposition}

\begin{proof}
For $\alpha<\omega_1$, this follows from Lemma~\ref{lemma: complexity of R T}. For $\alpha=\omega_1$, this holds by \eqref{case: R T and IF} from Lemma~\ref{lemma: basic properties of R T sets}.
\end{proof}

Let $X\subset Y$ be topological space and $\mc C$ a Suslin scheme in $X$, and suppose that $\mc C$ covers $X$.
These assumptions in particular imply $\mc A(\mc C) = X$ and $C(\emptyset) = X$.
By definition of $\mathbf{R}_0(\cdot)$ (and Lemma~\ref{lemma: construction of admissible mappings}\,(i)), we have $\mathbf{R}_0(\mc C) = C(\emptyset)$.
Consequently, the~corresponding closed Suslin scheme $\overline{\mc C}^Y$ in $Y$ satisfies $\mc A(\overline{\mc C}^Y\!) \supset X$ and $\mathbf{R}_0(\mc C,Y) = \overline{C(\emptyset)}^Y = \overline{X}^Y$.
As a particular case of \eqref{case: R T and embeddings} from Lemma~\ref{lemma: basic properties of R T sets}, we get
\begin{equation} \label{equation: R alpha as approximations of X}
\overline{X}^Y = \overline{C(\emptyset)}^Y = \mathbf{R}_0(\mc C,Y)
\supset \mathbf{R}_1(\mc C,Y) \supset \dots \mathbf{R}_\alpha(\mc C,Y) \supset \dots 
\supset \mathbf{R}_{\omega_1}(\mc C,Y) = \mc A(\overline{\mc C}^Y\!) \supset X .
\end{equation}

In the~sense of \eqref{equation: R alpha as approximations of X}, the~sets $\mathbf{R}_\alpha(\mc C,Y)$, $\alpha\leq \omega_1$, can be viewed as approximations of $X$ in $Y$, or as its $\Fa$-envelopes in $Y$.
The set $\mathbf{R}_0(\mc C,Y)$ is the~nicest (in other words, closed) approximation, but also the~biggest. We can get a more accurate approximation by increasing $\alpha$, but this comes at a cost of increased complexity.
We might get lucky with the~choice of $\mc C$, and get $\mathbf{R}_\alpha(\mc C,Y) = X$ at some point (and thus also $\mathbf{R}_\beta (\mc C,Y)$ for every $\beta \in [\alpha,\omega_1]$). In that case, we would say that $\mc C$ is a regular $\Fa$-representation of $X$ in $Y$.
This is the~case when $\mc C$ is complete, because then $X = \mc A(\overline{\mc C}^Y\!)$ (Lemma~\ref{lemma: complete suslin schemes in super spaces}). But even then, there is no reason to expect that $\alpha = \Compl{X}{Y}$.

We have seen that a regular representation of $X$ in $Y$ can also be specified by picking a Suslin scheme $\mc C$ in $X$ and providing an ordinal $\alpha$ for which this $\mc C$ ``works''. This motivates the~following definition:

\begin{definition}[Regular $\Fa$-representations]\label{definition: regular representations}
A Suslin scheme $\mc C$ in $X$ is said to be a
\begin{itemize}
\item \emph{regular $\Fa$-representation of $X$ in $Y$}, for some $Y\supset X$ and $\alpha\leq\omega_1$, if it satisfies $X = \mathbf{R}_\alpha (\mc C,Y)$;
\item \emph{universal regular $\Fa$-representation of $X$} if it is a regular $\Fa$-representation of $X$ in every $Y\supset X$.
\end{itemize}
\end{definition}

The existence of regular $\Fa$-representations is studied in Section~\ref{section: existence of regular representation}.
However, as this work is not a detective novel, we can go ahead and spoil the~surprise right away: any $\Fa$-subset $X$ of $Y$ has some regular $\Fa$-representation (Proposition~\ref{proposition: existence of regular F alpha representations}). And when either $X$ or $Y$ is $\mc K$-analytic, this representation can be made complete ``without loss of generality'' (Theorem~\ref{theorem: existence of regular representation}).

Clearly, the~condition of having a universal regular $\Fa$-representation is even stronger than being an absolute $\Fa$ space (at least formally).
In Section~\ref{section: broom space properties} we study a class of spaces (the so-called ``broom spaces'') which do admit universal regular representations. However, these spaces are rather simple from the~topological point of view -- they only have a single point which is not isolated.
As we have seen is Section~\ref{section: universal representation}, this is rather an exception.

\bigskip

Next, we describe $\mathbf{R}_\alpha$-sets in terms of a ``Suslin scheme rank'' on trees.

\subsection{Regular Representations and Suslin Scheme Ranks}\label{section: Suslin scheme rank}

For a Suslin scheme $\mc C$ in $X$ and a point $y\in Y \supset X$ we denote
\begin{equation} \label{equation: S C of y}
S_{\mc C}(y)  := \{ s\in \seq | \ \overline{C(s)}^Y \ni y  \} .
\end{equation}
Note that $S_{\mc C}(y)$ is always a tree (by the~monotonicity property of Suslin schemes).
The \emph{rank of $y$ corresponding to $\mc C$} is defined as $r_{\mc C}(y) := \rank (S_{\mc C}(y) )$.
The following lemma ensures that as long as we investigate only points from $Y\setminus \mc A (\overline{\mc C}^Y\!\!)$, we can assume that $r_{\mc C}$ is countable.

\begin{lemma}[$S_{\mc C}(\cdot)$ and $\textnormal{WF}$]\label{lemma: IF trees}
Let $X\subset Y$ be topological spaces, $\alpha\leq \omega_1$, and $\mc C$ a Suslin scheme in $Y$.
\begin{enumerate}[(i)]
\item For every $y\in Y$, we have $y \in \mc A( \overline{\mc C}^Y\!) \iff
	S_{\mc C}(y) \in \textnormal{IF} \iff r_i (S_{\mc C}(y) ) = \omega_1$. \label{case: Suslin operation and S of y}
\item If $\mc A (\overline{\mc C}^Y\!\!) = X$, we have $S_{\mc C}(y) \in \textnormal{WF}$ for every $y\in \mathbf{R}_\alpha(\mc C,Y) \setminus X$. \label{case: complete C and WF}
\end{enumerate}
\end{lemma}

By Lemma~\ref{lemma: complete suslin schemes in super spaces}, the~assumptions of (ii) are in particular satisfied when $\mc C$ is complete on $X$.

\begin{proof}
The first equivalence in \eqref{case: Suslin operation and S of y} is immediate once we rewrite $\mc A( \overline{\mc C}^Y\!)$ as $\bigcup_\sigma \bigcap_n \overline{C(\sigma|n)}^Y$. The second equivalence follows from the~fact the~derivative $D_i$ leaves a branch of a tree ``untouched'' if and only if the~branch is infinite.

\eqref{case: complete C and WF}: We have $\mc A( \overline{\mc C}^Y\!) \subset X$.
By \eqref{case: Suslin operation and S of y}, the~tree $S_{\mc C}(y)$ corresponding to $y \in Y\setminus X$ must not be ill-founded.
\end{proof}

We have seen that $\mathbf{R}_\alpha(\mc C,Y) \supset \mc A(\overline{\mc C}^Y\!) \supset X$ holds for every $\alpha\leq \omega_1$. In other words, $\mathbf{R}_\alpha(\mc C,Y)$ always contains $X$, and then maybe some extra points from $Y\setminus X$.
It is our goal in this subsection to characterize this remainder of $\mathbf{R}_\alpha(\mc C,Y)$ in $Y$.
We claim that for even $\alpha$, this remainder can be written as
\[ \mathbf{R}_\alpha(\mc C,Y) \setminus X = \{ y\in Y \setminus X | \ r_{\mc C}(y) \geq \alpha' \} . \footnote{Recall that for $\alpha=\lambda+2n+i$ we have $\alpha ' := \lambda +n$ (where $\lambda$ is limit or $0$, $n\in\omega$ , and $i\in\{0,1\}$).} \]
Unfortunately, we haven't found any such succinct formulation for odd $\alpha$. We will therefore describe the~remainder in terms of $\D$-derivatives of $S_{\mc C}(y)$. (The proof of this description is presented later in this section.)

\begin{proposition}[Description of regular representations via $\D$]\label{proposition: regular representations and D}
Let $\alpha < \omega_1$ and let $\mc C$ be a Suslin scheme on $X \subset Y$ satisfying $\mc A(\overline{\mc C}^Y\!\!) = X$.
\begin{enumerate}[(i)]
\item For even $\alpha$, we have $\mathbf{R}_\alpha (\mc C ,Y) \setminus X =
	\{ y\in Y\setminus X | \ \D^{\alpha'}(S_{\mc C} (y)) \neq \emptyset \}$
\item For odd $\alpha$, we have $\mathbf{R}_{\alpha} (\mc C ,Y) \setminus X =
	\{ y\in Y \setminus X| \ \D^{\alpha'}(S_{\mc C} (y)) \supsetneq \{\emptyset\} \} $.
 \item Moreover, we have $\mathbf{R}_{T_{\alpha',i}^{\textnormal{c}}} (\mc C ,Y) \setminus X =
 	\{ y\in Y \setminus X| \ \D^{\alpha'}(S_{\mc C} (y)) \cap \omega^i \neq \emptyset \} $ for any $i\in\omega$.
\end{enumerate}
\end{proposition}

\noindent (Recall that for $\alpha=\omega_1$ and complete $\mc C$ on $X$, we have $\mathbf{R}_{\omega_1}(\mc C ,Y) = X$.)

For even $\alpha$ we have $\alpha'=(\alpha+1)'$ and the~set $\D^{\alpha'}(S_{\mc C}(y))$ is always a tree. Therefore, the~case (i), resp. (ii) and (iii), of the~proposition says that the~remainder is equal to those $y\in Y\setminus X$ for which $\D^{\alpha'}(S_{\mc C}(y))$ contains a sequence of length 0 (that is, the~empty sequence $\emptyset$), resp. some sequence of length 1 (equivalently, any $s\neq \emptyset$), resp. some sequence of length $i$.

This gives the~following criterion for bounding the~complexity of $X$ in $Y$ from above (actually, even for showing that $X$ admits a regular $\Fa$-representation in $Y$):

\begin{corollary}[Sufficient condition for being $\Fa$]
	\label{corollary: sufficient condition for Fa}
A space $X\subset Y$ satisfies $X\in \Fa(Y)$ for some $\alpha<\omega_1$, provided that there is a Suslin scheme $\mc C$ on $X$, s.t. $\mc A(\overline{\mc C}^Y\!\!) = X$ and one of the~following holds:
\begin{enumerate}[(i)]
\item $\alpha$ is even and $\D^{\alpha'} (S_{\mc C}(y))$ is empty for every $y \in Y\setminus X$;
\item $\alpha$ is odd and $\D^{\alpha'} (S_{\mc C}(y))$ is either empty or it only contains the~empty sequence for every $y \in Y\setminus X$;
\item $\alpha$ is odd and there is $i\in\omega$ s.t. $\D^{\alpha'} (S_{\mc C}(y))$ only contains sequences of length $\leq i$ (that is, it is a subset of $\omega^{\leq i}$) for every $y \in Y\setminus X$.
\end{enumerate}
\end{corollary}

We now proceed to give the~proof of Proposition~\ref{proposition: regular representations and D}.
For the~``$\supset$'' inclusion, it suffices to construct a suitable admissible mapping.

\begin{proof}[Proof of Proposition~\ref{proposition: regular representations and D}, the~``$\supset$'' part]
Recall that for any derivative $D$ on $\seq$, the~corresponding rank $r$ is defined such that for every $S\subset \seq$, $r(S)$ is the~highest ordinal for which $D^{r(S)}(S)$ is non-empty.

(i): Let $\mc C$ be a Suslin scheme in $X\subset Y$, $\alpha<\omega_1$ even, and $y\in Y$ s.t. $r_{\mc C}(y) = \rank (S_{\mc C}(y)) \geq \alpha'$.
(This direction of the~proof does not require $\mc C$ to satisfy $\mc A(\overline{\mc C}^Y\!\!)$.)
Recall that $\mathbf{R}_\alpha(\mc C,Y) \overset{\text{def.}}= \mathbf{R}_{T_{\alpha'}}(\mc C,Y)$ and, as noted below Notation~\ref{notation: trees of height alpha}, $r_l(T_{\alpha'})=\alpha' \leq r_{\mc C}(y)$.

We will prove that for every $T\in \Tr$ and $y\in Y$ satisfying $r_l(T) \leq r_{\mc C}(y)$, there exists an admissible mapping witnessing that $y$ belongs to $\mathbf{R}_T(\mc C,Y)$.
In particular, we will use induction over $|t|\in \omega$, to construct a mapping $\varphi : T \rightarrow \seq$ which satisfies (a) and (b) for every $t\in T$:
\begin{enumerate}[(a)]
\item If $t=t'\ext m$ for some $t'\in T$ and $m\in\omega$, then we have $\varphi(t) \sqsupset \varphi(t')$ and $|\varphi(t)| = |\varphi(t')| + m$.
\item For every $\gamma<\omega_1$, we have $t\in D^\gamma_l(T) \implies \varphi(t) \in \D^\gamma (S_{\mc C}(y))$.
\end{enumerate}
By (a), the~resulting mapping will be admissible. By (b) with $\gamma=0$, we have $\varphi: T \rightarrow S_{\mc C}(y)$, which proves that $\varphi$ witnesses that $y\in \mathbf{R}_T(\mc C,Y)$.

$|t|=0$: The only sequence of length $0$ is the~empty sequence, so (a) is trivially satisfied by setting $\varphi(\emptyset) := \emptyset$. Any $\gamma$ satisfying $\emptyset\in D^\gamma_l(T)$ is smaller than $r_l(T)$ by definition, and we have $r_l(T) \leq \rank(S_{\mc C}(y)) = r_{\mc C}(y)$. It follows that $\varphi(\emptyset)$ belongs to $\D^\gamma(S_{\mc C}(y))$:
\[ \D^\gamma(S_{\mc C}(y)) \supset \D^{r_l(T)}(S_{\mc C}(y))
\supset \D^{\rank(S_{\mc C}(y))}(S_{\mc C}(y)) \ni \emptyset \]
(where the~last set is non-empty by definition of $\rank$).

Let $t=t'\ext m\in T$ (where $t'\in T$ and $m\in\omega$) and suppose we already have $\varphi(t')$ satisfying (a) and (b).
Let $\gamma_t$ be the~highest ordinal for which $t\in D^{\gamma_t}_l(T)$. By definition of $D_l$, we have $t' \in D^{\gamma_t+1}_l(T)$.
Consequently, the~induction hypothesis gives $\varphi(t') \in \D( \D^{\gamma_t} (S_{\mc C}(y)) )$.
By definition of $\D$, there is some $s_m \in \D^{\gamma_t} (S_{\mc C}(y))$ for which $|s_m| \geq |\varphi(t')| + m$.
We set $\varphi(t) = \varphi(t'\ext m) := s_m |_{|\varphi(t')|+m}$. Clearly, $\varphi(t)$ satisfies (a) and (b).

(ii), (iii): (ii) is a special case of (iii), so it suffices to prove (iii).
Let $\mc C$ be a Suslin scheme on $X\subset Y$, $i\in\omega$, and $\alpha<\omega_1$ an even ordinal (so that $\alpha+1$ is odd -- recall that $(\alpha+1)'=\alpha'$).
Let $y\in Y$ be s.t. $\D^{\alpha'}(S_{\mc C}(y))$ contains some sequence $s_y$ of length $i$.

Since $T^c_{\alpha+1,i} \overset{\text{def.}}= \{\emptyset\} \cup i\ext T^c_\alpha$, an application of Lemma~\ref{lemma: basic properties of R T sets}\,\eqref{case: R T formula} yields the~first identity in the~following formula:
\[ \mathbf{R}_{T^c_{\alpha,i}} (\mc C,Y)
= \bigcup_{s\in\omega^i} \mathbf{R}^s_{T^c_\alpha} (\mc C,Y)
\supset \mathbf{R}^{s_y}_{T^c_\alpha} (\mc C,Y)
= \mathbf{R}^{s_y}_{\alpha} (\mc C,Y)
= \mathbf{R}_{\alpha} \left( \mc D,Y\right) ,\]
where $\mc D := \left( C(s_y \ext t) \right)_{t\in\seq}$.
Clearly, $S_{\mc D}(y)$ contains all sequences $t$ for which $s_y \ext t$ belongs to $S_{\mc C}(y)$.
In particular, $\D^{\alpha'}(S_{\mc D}(y))$ is non-empty, so $y$, $\alpha$ and $\mc D$ satisfy all assumptions used in the~proof of case (i). Therefore, $y$ belongs to $\mathbf{R}_{\alpha} \left( \mc D,Y\right) \subset \mathbf{R}_{T^c_{\alpha,i}} (\mc C,Y)$, which concludes the~proof. 
\end{proof}

For the~``$\subset$'' inclusion, we start with the~following lemma:

\begin{lemma}[Points from the~remainder force big $S_{\mc C}(\cdot)$]\label{lemma: admissible mappings and rank}
Let $\mc C$ be a Suslin scheme on $X$, $T\in\Tr$, and $\varphi$ witnesses that $y\in \mathbf{R}_T(\mc C,Y)$. If $\mc A(\overline{\mc C}^Y\!\!) = X$ and $y\notin X$, then $S_{\mc C}(y) \supset \varphi( D^\gamma_i (T) )$ holds for every $\gamma<\omega_1$.
\end{lemma}

\begin{proof}
Let $\mc C$, $T$, $y$ and $\varphi$ be as in the~lemma.
Since $y\in \bigcap_T \overline{C(\varphi(t))}^Y$, we have $\varphi (T) \subset S_{\mc C}(y)$ by definition of $S_{\mc C}(y)$. Since $S_{\mc C}(y)$ is a tree (by monotonicity of $\mc C$), we even get
\[ \D^\gamma(\varphi(T)) \subset \cltr{\varphi(T)} \subset S_{\mc C}(y) \]
for every $\gamma<\omega_1$.
Consequently, it suffices to prove the~following implication for every $\gamma<\omega_1$
\begin{equation}\label{equation: varphi induction hypothesis}
t\in D^\gamma_i (T) \implies \varphi(t) \in \D^\gamma(\varphi(T)) .
\end{equation}

We proceed by induction over $\gamma$.
For $\gamma=0$, \eqref{equation: varphi induction hypothesis} holds trivially:
\[ t\in D^0_i (T) = \cltr{T} = T \implies \varphi(t) \in \varphi(T) \subset \cltr{\varphi(T)} = \D^0 (\varphi(T)) .\]

$\gamma\mapsto \gamma+1$:
Let $t\in D^{\gamma+1}_i(T) = D_i ( D^\gamma_i(T) )$.
By definition of $D_i$, there exist infinitely many $n\in\omega$ for which $t\ext n \in D^\gamma_i(T)$.
By the~induction hypothesis, we have $\varphi(t\ext n) \in \D^\gamma (\varphi(T))$ for any such $n$.
Consider the~tree
\[ S := \cltr{\{\varphi(t\ext n) | \ n\in \omega, \ \varphi(t\ext n) \in \D^\gamma (\varphi(T)) \} } \subset \D^\gamma (\varphi(T)) \subset S_{\mc C}(y) .\]
Since $\varphi$ is admissible, we have $\varphi(t) \sqsubset \varphi(t\ext n)$ and $|\varphi(t\ext n)| = | \varphi(t)| + n$ for every $n\in \omega$.
It follows that $S$ is infinite -- since $S_{\mc C}(y)\in \textnormal{WF}$ holds by Lemma~\ref{lemma: IF trees} -- well-founded.
By König's lemma, there is some $s\in\seq$ and an infinite set $M\subset \omega$ such that $\{ s\ext m | \ m\in M \} \subset S$.
For every $m\in M$, we denote by $n_m$ the~integer for which $\varphi(t\ext n) \sqsupset s\ext m$.
The set $\{ \varphi(t\ext n_m) | \ m\in M \}$ then witnesses that $S$ contains infinitely many incomparable extensions of $\varphi(t)$. It follows that
\[ \varphi(t) \in\D(S) \subset \D (\D^\gamma(\varphi(T))) = \D^{\gamma+1}(\varphi(T)) .\]

Let $\lambda<\omega_1$ be limit and suppose that \eqref{equation: varphi induction hypothesis} holds for every $\gamma<\lambda$. By definition of $\D^\lambda$ and $D_i^\lambda$, \eqref{equation: varphi induction hypothesis} holds for $\lambda$ as well:
\[ t\in D^\lambda_i (T) = \bigcap_{\gamma<\lambda} D^\gamma_i (T)
\overset{\eqref{equation: varphi induction hypothesis}}{\implies}
\forall \gamma<\lambda : \varphi(t) \in \D^\gamma(\varphi(T)) \implies
\varphi \in \bigcap_{\gamma<\lambda} \D^\gamma (\varphi(T)) = \D^\lambda (\varphi(T)) .\]
\end{proof}

We can now finish the~proof of Proposition~\ref{proposition: regular representations and D}:

\begin{proof}[Proof of Proposition~\ref{proposition: regular representations and D}, the~``$\subset$'' part]
Let $\alpha<\omega_1$ and let $\mc C$ be a Suslin scheme satisfying $\mc A(\overline{\mc C}^Y\!\!) = X$.
Recall that by definition of $\mathbf{R}_\alpha(\cdot)$, we have $\mathbf{R}_\alpha(\cdot) = \mathbf{R}_{T^c_\alpha}(\cdot)$.

(i): Suppose that $\alpha$ is even.
The tree $T^c_\alpha \overset{\text{def.}}= T_{\alpha'}$ is constructed to satisfy $r_i(T_{\alpha'}) = \alpha'$, so we have $D^{\alpha'}_i (T_{\alpha'}) \neq \emptyset$.
In particular, $D^{\alpha'}_i (T_{\alpha'})$ contains the~empty sequence. By Lemma~\ref{lemma: admissible mappings and rank}, $\D^{\alpha'} (S_{\mc C}(y))$ contains $\varphi(\emptyset)$. In particular, $\D^{\alpha'} (S_{\mc C}(y))$ is non-empty, which proves the~inclusion ``$\subset$'' in (i).

(ii): Suppose that $\alpha$ is odd.
The tree $T^c_\alpha$ is defined as $T^c_\alpha \overset{\text{def.}}= \{\emptyset\} \cup 1\ext T_{\alpha'}$.
Since $\emptyset\in D^{\alpha'}_i (T_{\alpha'})$, it follows that $D^{\alpha'}_i (1\ext T_{\alpha'})$ contains the~sequence $1\ext \emptyset = (1)$.
By Lemma~\ref{lemma: admissible mappings and rank}, $\D^{\alpha'} (S_{\mc C}(y))$ contains $\varphi( (1) )$.
Since $\varphi$ is admissible, we have $|\varphi((1))| = 1$, which implies $\emptyset\neq \varphi((1))$.
This shows that $\D^{\alpha'} (S_{\mc C}(y))$ contains some other sequence than $\emptyset$, and therefore proves the~inclusion ``$\subset$'' in (ii).

The proof of (iii) is the~same as the~proof of (ii), except that we get $(i)\in \D^{\alpha'} (S_{\mc C}(y))$ instead of $(1)\in \D^{\alpha'} (S_{\mc C}(y))$, and the~admissibility of $\varphi$ gives $|\varphi((i))| = i$.
\end{proof}
\subsection{Existence of a~Regular Representation for \texorpdfstring{$\mc F$}{F}-Borel Sets} \label{section: existence of regular representation}

The goal of this section is to prove the~following theorem:

\begin{theorem}[Existence of regular $\Fa$-representations]\label{theorem: existence of regular representation}
Let $X$ be a~$\mc K$-analytic space and $\alpha\leq \omega_1$. For $Y \supset X$, the~following conditions are equivalent:
\begin{enumerate}[(1)]
\item $X\in\Fa(Y)$; \label{case: X Fa in Y}
\item $X=\mathbf{R}_\alpha(\mc C)$ for some closed Suslin scheme in $Y$;
	\label{case: R(C) with closed C}
\item $X$ has a~regular $\Fa$-representation in $Y$ -- that is, $X=\mathbf{R}_\alpha(\mc C,Y)$ for some Suslin scheme in $X$;
	\label{case: regular representation}
\item $X=\mathbf{R}_\alpha(\mc C,Y)$ for some complete Suslin scheme on $X$.
	\label{case: regular representation by complete C}
\end{enumerate}
\end{theorem}

First, we take a~look at the~implications in Theorem~\ref{theorem: existence of regular representation} in more detail. The implications $\eqref{case: regular representation by complete C} \implies \eqref{case: regular representation} \implies \eqref{case: R(C) with closed C}$ are trivial,
and the~implication $\eqref{case: R(C) with closed C} \implies \eqref{case: X Fa in Y}$ follow from Proposition~\ref{proposition: complexity of R alpha sets}.
The implication $\eqref{case: X Fa in Y} \implies \eqref{case: R(C) with closed C}$ is the~hardest one, and follows from Proposition~\ref{proposition: existence of regular F alpha representations}.
The implication $\eqref{case: R(C) with closed C} \implies \eqref{case: regular representation}$ is straightforward -- when $X=\mathbf{R}_\alpha(\mc C)$ holds for some closed Suslin scheme $\mc C$ in $Y$, we define a~Suslin scheme on $X$ as $\mc C' := (C(s) \cap X)_s$, and note that $\mathbf{R}_\alpha(\mc C',Y) = \bigcup_\varphi \bigcap_t \overline{C(\varphi(t))\cap X}^Y$ satisfies
\[ \mathbf{R}_\alpha(\mc C) = \mathbf{R}_\alpha(\mc C) \cap X \overset{\text{def.}}=
 \bigcup_\varphi \bigcap_t C(\varphi(t)) \cap X
 \subset  \bigcup_\varphi \bigcap_t \overline{C(\varphi(t))}^Y
 \overset{\mc C \text{ is}}{\underset{\text{closed}}=} \bigcup_\varphi \bigcap_t C(\varphi(t))
 \overset{\text{def.}}= \mathbf{R}_\alpha(\mc C) \]
 (where the~unions are taken over all admissible mappings $\varphi : T^c_{\alpha} \rightarrow \seq$, and the~intersections over $t\in T^c_\alpha$).
Lastly, any regular representation can be made complete by Lemma~\ref{lemma: completing representation} (provided that $X$ is $\mc K$-analytic), which proves the~implication $\eqref{case: regular representation} \implies \eqref{case: regular representation by complete C}$.

Next, we aim to prove the~implication $\eqref{case: X Fa in Y} \implies \eqref{case: R(C) with closed C}$.
The following lemma shows that \eqref{case: X Fa in Y} and \eqref{case: R(C) with closed C} are equivalent for $\alpha=2$. More importantly, it will used in the~induction step when proving the~equivalence for general $\alpha$.

\begin{lemma}[Existence of regular representations for $\fsd$ sets]\label{lemma: indexing representation by Suslin scheme}
For any $\alpha<\omega_1$ and $\mc X = X\in \left(\Fa(Y)\right)_{\sigma\delta}$, there exist a~Suslin scheme $\{X_s| \ s\in\seq\}$ in $Y$ satisfying
\begin{enumerate}[(i)]
\item $X=\bigcap_{m\in\omega} \bigcup_{s\in\omega^m} X_s$,
\item $\left( \forall s\in\seq \right) : X_s \in \Fa(Y)$,
\item $\mc X$ covers $X$.
\end{enumerate}
\end{lemma}
\begin{proof}
Let $\alpha$ and $X$ be as above.
Since $X$ belongs to $\left(\Fa(Y)\right)_{\sigma\delta}$, there exists some countable families $\mc P_m\subset \mc \Fa(Y)$, such that $X=\bigcap_{m\in\omega} \bigcup \mc P_m$. 
Using the~notation $\mc A \land \mc B = \{ A \cap B | \ A\in \mc A, B\in \mc B \}$, we set
\[ \mc R_m:= \mc P_0 \land \mc P_1 \land \dots \land \mc P_m .\]
Clearly, we have $X=\bigcap_{m\in\omega} \bigcup \mc R_m$ and $\mc R_m \subset \Fa(Y)$. Since $\mc P_m$ are all countable, we can enumerate them as $\mc P_m =: \{ P^m_n | \ n\in\omega \}$. We then get
\[ \mc R_m = \{ P^0_{n_0} \cap \dots \cap P^m_{n_m} | \ n_i \in \omega, \ i\leq m \} = \{ P^0_{s(0)} \cap \dots P^m_{s(m)} | \ s\in \omega^{m+1} \} . \]
Denoting $X_s := P^0_{s(0)} \cap \dots \cap P^m_{s(m)}$ for $s=(s(0),\dots,s(m))\in\seq$, we get a~Suslin scheme $\mc X$ in $Y$ which satisfies the~conditions $(i)$ and $(ii)$.
Moreover, any $\mc P_m$ is a~cover of $X$, so $\mc X$ covers $X$.
\end{proof}

The following proposition shows that the~equivalence between \eqref{case: X Fa in Y} and \eqref{case: R(C) with closed C} holds for $\alpha=4$. While the~proposition is not required for the~proof of Theorem~\ref{theorem: existence of regular representation}, we include it to demonstrate the~method of the~proof in a~less abstract setting. The proof also introduces the~notation necessary to get the~general result.

\begin{proposition}[Existence of regular representations for $\mc F_{\sigma\delta\sigma\delta}$ sets] \label{proposition: representation for F 4 sets}
A set $X\subset Y$ is $\mc F_{\sigma\delta\sigma\delta}$ in $Y$ if and only if there exists a~closed Suslin Scheme $\mc C$ in $Y$ s.t.
\[ X = \bigcap_{m\in\omega} \bigcup_{s\in \omega^m} \bigcap_{n\in\omega} \bigcup_{t\in \omega^n} C(s\ext t). \]
\end{proposition}
\begin{proof}
Let $X$ be a~subset of $Y$. Since the~implication from right to left is immediate, we need to prove that if $X$ is an $\mc F_{\sigma\delta\sigma\delta}$ subset of $Y$, it has a~representation with the~desired properties.

By Lemma~\ref{lemma: indexing representation by Suslin scheme}, there exists a~Suslin scheme $\mc X = \{X_s | \ s\in\seq \} \subset F_{\sigma\delta}(Y)$ covering $X$ such that
\[ X = \bigcap_{m\in\omega} \bigcup_{s\in \omega^m} X_s . \]
Using Lemma~\ref{lemma: indexing representation by Suslin scheme} once more on each $X_s$, we obtain Suslin schemes $\mc X_s = \{ X_s^t | \ t\in\seq \}$ such that
	\begin{itemize}
	\item for each $s$, we have $ X_s = \bigcap_{n\in\omega} \bigcup_{t\in \omega^n} X_s^t $,
	\item for each $s$ and $t$, $X_s^t$ is closed in $Y$,
	\item for each $s$,  $\{ X_s^t | \ t\in\seq \}$ covers $X_s$.
	\end{itemize}
It follows that $X$ can be written as
\[ X = \bigcap_{m\in\omega} \bigcup_{s\in \omega^m} \bigcap_{n\in\omega} \bigcup_{t\in \omega^n} X_s^t. \]

The idea behind the~proof is the~following: If it was the~case that $X_s^t=X_u^v$ held whenever $s\ext t = u \ext v$, we could define $C(s\ext t)$ as $X_s^t$, and the~proof would be finished.
Unfortunately, there is no reason these sets should have such a~property. However, by careful refining and re-indexing of the~collection $\{ X_s^t | \ s,t\in\seq \}$, we will be able to construct a~new family $\{ C_s^t | \ s,t\in\seq \}$ for which the~above condition will hold. And since the~condition above holds, we will simply denote the~sets as $C(s\ext t)$ instead of $C_s^t$.

We now introduce the~technical notation required for the~definition of $\mc C$. For a~less formal overview of what the~notation is about, see Figure~\ref{figure: notation for representation theorem}.
\begin{notation}
\begin{itemize}
	\item $S_0:=\{\emptyset\}$, $S_m:=\omega^2 \times \omega^3 \times \dots \times \omega^{m+1}$ for $m\geq 		1$ and $S:=\bigcup_{m\in\omega} S_m$,
	\item Elements of $S$ will be denoted as
		\begin{equation*} \begin{split}
		\vec s_m 	& = \left(s_1,s_2,\dots,s_m\right)  \\
				& = \left( \left(s_1^0, s_1^1\right), \left(s_2^0, s_2^1, s_2^2\right), \dots, \left(s_m^0, s_m^1, \dots, s_m^m\right) \right) ,
		\end{split} \end{equation*}
		where $s_k \in \omega^{k+1}$ and
		$s^l_k\in\omega$. By length $|\vec s_m|$ of $\vec s_m$ we will understand the~number of sequences 			it contains -- in this case `$m$'. If there is no need to specify the~length, we will denote an 			element of $S$ simply as $\vec s$.
	\item $\pi_k:\omega^k\rightarrow \omega$ will be an arbitrary fixed bijection (for $k\geq 2$) and we 			set $\pi:=\bigcup_{k=2}^\infty \pi_k$.
	\item We define $\varrho_0: \emptyset \in S_0 \mapsto \emptyset$, for $m\geq 1$ we set
		$$\varrho_m : \vec s_m \in S_m \mapsto \left( \pi_2 (s_1), \pi_3 (s_2), \dots, \pi_{m+1} (s_m) 			\right) \in \omega^m$$
		and we put these mappings together as $\varrho:=\bigcup_{m\in\omega} \varrho_m : S \rightarrow \seq 		$.
	\item For $k\in\omega$ we define the~mapping $\Delta_k$ as
		$$\Delta_k : \vec s \in \bigcup_{m\geq k} S_m \mapsto \left( s_1^1, s^2_2,  \dots, s^k_k 			\right) \in \omega^k$$
		(where $\Delta_0$ maps any $\vec s \in S$ to the~empty sequence $\emptyset$).
	\item We also define a~mapping $\xi_k$ for $k\in\omega$:
		$$\xi_k : \vec s \in \bigcup_{m\geq k} S_m \mapsto \left( s^k_{k+1}, s^k_{k+2},
		\dots,s_{|\vec s|}^k \right) \in \omega^{|\vec s|-k}.$$
\end{itemize}

\begin{figure}
\definecolor{Gray}{gray}{0.9}
\definecolor{DarkGray}{gray}{0.6}
\definecolor{Cyan}{rgb}{0.88,1,1}
\begin{tabular}{r l l l l l l l l l}
$\omega^0 \ni$ & \!\!\!\!``$s_0$''=	& \ \cellcolor{DarkGray}$\emptyset$ & & &		& & & $\longrightarrow \emptyset$ & \\
$\omega^2 \ni$ & 	$s_1=$	& $(s^0_1$ 	& \cellcolor{DarkGray} $s^1_1)$	& &	& & & $\longrightarrow \pi(s_1)$ & $\in \omega$ \\
$\omega^3 \ni$ & 	$s_2=$	& $(s^0_2$	& $s^1_2$	& \cellcolor{DarkGray} $s^2_2)$ & & & & $\longrightarrow \pi(s_2)$ & $\in \omega$ \\
$\omega^4 \ni$ & 	$s_3=$	& $(s^0_3$	& $s^1_3$	& \cellcolor{Gray}$s^2_3$	& $s^3_3)$	& & & $\longrightarrow \pi(s_3)$ & $\in \omega$ \\
$\omega^5 \ni$ & 	$s_4=$	& $(s^0_4$	& $s^1_4$	& \cellcolor{Gray}$s^2_4$	& $s^3_4$	& $s^4_4)$ & & $\longrightarrow \pi(s_4)$ & $\in \omega$ \\
& \dots \ \ \ \ 	&		&			& \cellcolor{Gray} \dots	&			& & & \ \ \dots & \\
$\omega^{m+1} \ni$ & $s_m=$ 	& $(s^0_m$	& $s^1_m$	& \cellcolor{Gray} $s^2_m$	& $s^3_m$	& ...	& $s^m_m)$ & $\longrightarrow \pi(s_m)$ & $\in \omega$ \\
			\\
$S_m \ni$ & $\vec s_{m}$ \ \ \ &		&			&			&			& & & $\longrightarrow \varrho(\vec s_{m})$ & $\in \omega^m$ 
\end{tabular}
\caption{An illustration of the~notation from the~proof of Proposition~\ref{proposition: representation for F 4 sets}.
Each sequence $s_k\in\omega^{k+1}$ is mapped to an integer $\pi(s_k)$  via a~bijection.
This induces a~mapping of the~``sequence of sequences'' $\vec s_m$ to a~sequence $\varrho(\vec s_m)$ of length $|\vec s_m|=m$.
The diagonal sequence highlighted in dark gray corresponds to $\Delta_2(\vec s_{m})$. Note that the~letter $\Delta$ stands for `diagonal' and the~lower index ($2$ in this case) corresponds to the~length of this sequence. The part of the~column highlighted in light gray corresponds to the~sequence $\xi_2(\vec s_{m})$.
\label{figure: notation for representation theorem}}
\end{figure}
\end{notation}

Suppose that $w\in \omega^m$ satisfies $w=\varrho(\vec s)$ for some $\vec s \in S_m$. Without yet claiming that this correctly defines a~Suslin scheme, we define $C(w)$ as
\begin{equation*} \begin{split} \label{equation: definition of C}
C\left(\varrho\left(\vec s\right)\right) \ 
:= \ & \bigcap_{k=0}^{|\vec s|} X^{\xi_k(\vec s)}_{\Delta_k(\vec s)} \ 
 = \ X^{\xi_0(\vec s)}_{\Delta_0(\vec s)} \cap X^{\xi_1(\vec s)}_{\Delta_1(\vec s)} \cap \dots \cap X^{\xi_{|\vec s|}(\vec s)}_{\Delta_{|\vec s|}(\vec s)} \\
 = \ & X_\emptyset^{s_1^0 s_2^0 \dots s_m^0} \cap X_{s_1^1}^{s_2^1 s_3^1 \dots s_m^1}
 	\cap X_{s_1^1 s_2^2}^{s_3^2 s_4^2 \dots s_m^2} \cap \dots \cap X^\emptyset_{s_1^1 s_2^2 \dots s_m^m} .
\end{split} \end{equation*}

In order to show that $\mc C$ indeed does have the~desired properties, we first note the~following properties of $\varrho$, $\Delta_k$ and $\xi_k$ (all of which immediately follow from the~corresponding definitions).
\begin{lemma} \label{lemma: properties of auxiliary functions}
The functions $\varrho$, $\Delta_k$ and $\xi_k$ have the~following properties:
\begin{enumerate}[(a)]
\item $\left(\forall \vec s \in S \right) \left(\forall k\leq |\vec s|\right) :
	\left|\Delta_k(\vec s)\ext \xi_k(\vec s)\right| = \left| \vec s \right| = \left| \varrho(\vec s) \right|$.
\item $\varrho : S \rightarrow \seq$ is a~bijection.
\item $\left(\forall u,v \in \seq \right): v\sqsupset u \iff \varrho^{-1}(v) \sqsupset \varrho^{-1}(u)$.
\item $\left( \forall \vec s, \vec t \in S \right) \left( \forall k \leq |\vec s| \right) : \vec t \sqsupset \vec s \implies \Delta_k(\vec t) = \Delta_k(\vec s) \ \& \ \xi_k(\vec t) \sqsupset \xi_k(\vec s) $.
\end{enumerate}
\end{lemma}
From $(b)$ it follows that each $C(\varrho(\vec s))$ is well defined and that $C(w)$ is defined for every $w\in\seq$.
To see that $\mc C$ is a~Suslin scheme, we need to verify its monotonicity. Let $v\sqsupset u$ be two elements of $\seq$. We have
\begin{equation*} \begin{split}
C(v) =
	\ & \bigcap_{k=0}^{|\varrho^{-1}(v)|} X^{\xi_k(\varrho^{-1}(v))}_{\Delta_k(\varrho^{-1}(v))} 
	\ \overset{(c)}\subset \ \bigcap_{k=0}^{|\varrho^{-1}(u)|} X^{\xi_k(\varrho^{-1}(v)) }_{\Delta_k(\varrho^{-1}(v))} \\
	\overset{(c),(d)}{=} & \ \bigcap_{k=0}^{|\varrho^{-1}(u)|} X^{\xi_k(\varrho^{-1}(v))}_{\Delta_k(\varrho^{-1}(u))}
	\ \subset \ \bigcap_{k=0}^{|\varrho^{-1}(u)|} X^{\xi_k(\varrho^{-1}(u))}_{\Delta_k(\varrho^{-1}(u))} = C(u) ,
\end{split} \end{equation*}

where the~first and last identities are just the~definition of $C(\cdot)$ and the~last inclusion holds because we have $\xi_k(\varrho^{-1}(v)) \sqsupset \xi_k(\varrho^{-1}(u))$ (by (c) and (d)) and $\mc X_{\Delta_k(\varrho^{-1}(u))} = (X^t_{\Delta_k(\varrho^{-1}(u))})_t$ is a~Suslin scheme (and hence monotone).

Next, we will show that
$$ X \subset \bigcup_{\sigma\in\baire} \bigcap_{m\in\omega} C(\sigma|m) = \mc A(\mc C)
 = \mathbf{R}_{\omega_1}(\mc C) \subset \mathbf{R}_4(\mc C)
 = \bigcap_{m\in\omega} \bigcup_{s\in \omega^m} \bigcap_{n\in\omega} \bigcup_{t\in \omega^n} C(s\ext t)
 .$$
Going from right to left, the~identities (resp. inclusions) above follow from: definition of $\mathbf{R}_4$, Lemma~\ref{lemma: basic properties of R T sets}\,\eqref{case: R T and embeddings}, Lemma~\ref{lemma: basic properties of R T sets}\,\eqref{case: R T and IF}, and definition of $\mc A(\cdot)$; it remains to prove the~first inclusion.
Let $x\in X$.
Our goal is to produce $\mu \in \baire$ such that $x\in C(\mu|m)$ holds for each $m\in\omega$. We shall do this by finding a~sequence $\vec s_0 \sqsubset \vec s_1 \sqsubset \dots $, $\vec s_m \in S_m$, for which which $x\in \bigcap_m C(\varrho(\vec s_m))$. Once we have done this, $(c)$ ensures that $\mu|m := \varrho (\vec s_m)$ correctly defines the~desired sequence $\mu$.

To this end, observe that since $\{X_s | \ s\in\seq \}$ covers $X$, there exists a~sequence $\sigma\in \baire$ such that $x\in \bigcap_k X_{\sigma|k}$. Similarly, $\{X_{\sigma|k}^t | \ t\in\omega \}$ covers $X_{\sigma|k}$ for any $k\in\omega$, so there are sequences $\nu_k \in \baire$, $k\in\omega$, for which $x\in \bigcap_m X_{\sigma|k}^{\nu_k|m}$. For $m\in\omega$ we denote
\[ s_m := \left( \nu_0(m-1), \nu_1(m-2), \dots, \nu_k (m-1-k), \dots, \nu_{m-1}(0), \sigma (m-1) \right) \]
and define $\vec s_m := (s_1, s_2, \dots , s_m )$. For $k\leq m$ we have
\begin{align*}
\Delta_k(\vec s_m) = & ( s_1^1, \dots , s_k^k ) = ( \sigma(0), \sigma(1), \dots, \sigma(k-1) ) = \sigma|k ,\\
\xi_k(\vec s_m) = & ( s_{k+1}^k, s_{k+2}^k, \dots, s_m^k ) \\
 = & \left( \nu_k(k+1-1-k), \nu_k(k+2-1-k), \dots , \nu_k(m-1-k) \right) \\
 = & \left( \nu_k(0), \nu_k(1), \dots , \nu_k(m-k-1) \right) = \nu_k|m-k.
\end{align*}
It follows that $X_{\Delta_k(\vec s_m)}^{\xi_k(\vec s_m)} = X_{\sigma|k}^{\nu_k|m-k} \ni x$ holds for every $k\leq m$. As a~consequence, we get $C(\varrho(\vec s_m)) = \bigcap_{k=0}^m X_{\sigma|k}^{\nu_k|m-k} \ni x$ for every $m\in\omega$, which shows that $\mc C$ covers $X$.

In order to finish the~proof of the~theorem, it remains to show that
\[ \bigcap_{m\in\omega} \bigcup_{u\in \omega^m} \bigcap_{n\in\omega} \bigcup_{v\in \omega^n} C(u\ext v)
\subset \bigcap_{m\in\omega} \bigcup_{s\in \omega^m} \bigcap_{n\in\omega} \bigcup_{t\in \omega^n} X_s^t = X
. \]
To get this inclusion, it is enough to prove that
\[ \left( \forall u \in \seq \right) \left( \exists s_u \in \omega^{|u|} \right)
	\left( \forall v \in \seq \right) \left( \exists t_{uv} \in \omega^{|v|} \right) : C(u\ext v) \subset X_{s_u}^{t_{uv}} . \]
We claim that a~suitable choice is $s_u := \Delta_{|u|}(\varrho^{-1}(u))$ and $t_{uv}:= \xi_{|u|}(\varrho^{-1}(u\ext v))$. Indeed, for any $u\in\omega^m$ and $v\in\omega^n$ we have
\[ C(u\ext v) \overset{\textrm{def.}}=
	\bigcap_{k=0}^{m+n} X^{\xi_k(\varrho^{-1}(u\ext v))}_{\Delta_k(\varrho^{-1}(u\ext v))}
	\subset X^{\xi_m(\varrho^{-1}(u\ext v))}_{\Delta_m(\varrho^{-1}(u\ext v))}
	\overset{(d)}= X^{\xi_m(\varrho^{-1}(u\ext v))}_{\Delta_m(\varrho^{-1}(u))} = X_{s_u}^{t_{uv}} . \]
\end{proof}

The following proposition proves the~implication \eqref{case: X Fa in Y}$\implies$\eqref{case: R(C) with closed C} of Theorem~\ref{theorem: existence of regular representation}.

\begin{proposition}[Existence of regular representations]
	\label{proposition: existence of regular F alpha representations}
Let $X\subset Y$ and $\alpha\leq \omega_1$. Then $X\in\Fa(Y)$ if and only if $X=\mathbf{R}_\alpha(\mc C)$ holds for some Suslin scheme $\mc C$ which is closed in $Y$ and covers $X$.
\end{proposition}
\begin{proof}
The implication ``$\Longleftarrow$'' follows from Proposition~\ref{proposition: complexity of R alpha sets}, so it remains to prove the~implication ``$\implies$''.

For $\alpha=0$, $X$ is closed and any Suslin scheme $\mc C$ satisfies $\mathbf{R}_0(\mc C) = C(\emptyset)$. Consequently, it suffices to set $C(u):=\overline{X}^Y$ for every $s\in \seq$.
We already have the~statement for $\alpha=2$ (Lemma~\ref{lemma: indexing representation by Suslin scheme}), $\alpha=4$ (Proposition~\ref{proposition: representation for F 4 sets}) and $\alpha=\omega_1$ ($(5)$ from Lemma~\ref{lemma: basic properties of R T sets}).

Suppose that the~statement holds for an even ordinal $\alpha<\omega_1$ and let $X\in \mc F_{\alpha+1}(Y)$. By induction hypothesis we have
\[ X= \bigcup_{m\in\omega} X_m = \bigcup_{m\in\omega} \mathbf{R}_\alpha (\mc C_m) \]
for some Suslin schemes $\mc C_m$, $m\in\omega$, which are closed in $Y$ and cover $X_m$.
We define $\mc C$ as $C(\emptyset):= \overline{X}^Y$ and, for $m\in\omega$ and $t\in \seq$, $C(m\ext t):=\overline{X}^Y \cap C_m(t)$.
Clearly, $\mc C$ is a~Suslin scheme which is closed in $Y$ and covers $X$.
Recall that for odd $\alpha+1$ we have $T^\textrm{c}_{\alpha+1}=\{\emptyset\} \cup 1 \ext T_\alpha^\textrm{c}$.
For $x\in Y$ we have
\begin{equation*} \begin{split}
x \in X & \iff \left(\exists m\in \omega \right) : x \in \mathbf{R}_\alpha(\mc C_m) \\
		& \iff \left(\exists m \in \omega \right) \left(\exists \varphi : T^\textrm{c}_\alpha \rightarrow \seq \textrm{ adm.}\right) : x \in \bigcap_{t'\in T^\textrm{c}_\alpha} C_m(\varphi(t')) = \bigcap_{t'\in T^\textrm{c}_\alpha} C(m\ext \varphi(t')) \\
		& \iff \left(\exists \psi : 1\ext T^\textrm{c}_\alpha \rightarrow \seq \textrm{ adm.}\right) : x \in \bigcap_{1\ext t'\in 1\ext T^\textrm{c}_\alpha} C(\psi(m \ext t')) \ \& \ x\in \overline{X}^Y = C(\emptyset) \\
		& \iff \left(\exists \psi : T^\textrm{c}_{\alpha+1} \rightarrow \seq \textrm{ adm.}\right) : x \in \bigcap_{t\in T^\textrm{c}_{\alpha+1}} C(\psi(t)) \\
		& \iff x \in \mathbf{R}_{\alpha+1}(\mc C).
\end{split} \end{equation*}

Finally, suppose that $X\in \Fa(Y)$ holds for $0<\alpha<\omega_1$ even and that the~statement holds for all $\beta<\alpha$.

If $\alpha$ is a~successor ordinal, we can use Lemma~\ref{lemma: indexing representation by Suslin scheme} to obtain a~Suslin scheme $\{ X_s | \ s\in\seq \}\subset \mc F_{\alpha-2}(Y)$ satisfying the~conditions (i)-(iii) from the~lemma.
For each $m\in\omega$, we set $\alpha_m:=\alpha-2$ and define
\begin{equation}\label{equation:T}
T:=\{\emptyset \} \cup \bigcup_{m\in\omega} m\ext T_{\alpha_m}^\textrm{c} .
\end{equation}
Since $T = T^c_\alpha$, we trivially have $\mathbf{R}_T(\cdot) = \mathbf{R}_\alpha(\cdot)$ and the~definition of $\alpha_m$ is chosen to satisfy the~following:
\begin{equation}\label{equation: X s is in F alpha m}
\left( \forall s \in \seq \right) : X_s \in \mc F_{\alpha_{|s|}}(Y) .
\end{equation}

For limit $\alpha$, we can repeat the~proof of Lemma~\ref{lemma: indexing representation by Suslin scheme} to obtain a~sequence $(\alpha_m)_m$ and a~Suslin scheme $\{ X_s | \ s\in\seq \}$ which satisfies (i)-(iii) from Lemma~\ref{lemma: indexing representation by Suslin scheme}, \eqref{equation: X s is in F alpha m} and $\alpha_m < \alpha$ for each $m\in\omega$.
Moreover, we can assume without loss of generality that $\sup_m \alpha_m = \alpha$.
Let $T$ be the~tree defined by \eqref{equation:T}. 
We claim that the~trees $T^\textrm{c}_\alpha$ and $T$ are equivalent in the~sense of \eqref{case: R T and embeddings} from Lemma~\ref{lemma: basic properties of R T sets}, so that we have $\mathbf{R}_T(\cdot ) = \mathbf{R}_\alpha(\cdot)$ even when $\alpha$ is limit.

Indeed, the~sequence $(\pi_\alpha(n))_n$, defined in Notation~\ref{notation: trees of height alpha}, converges to $\alpha$ with increasing $n$, so for every $m\in\omega$ there exists $n_m\geq m$ such that $\pi_\alpha(n_m) \geq \alpha_{|m|}$. Since any canonical tree can be embedded (in the~sense of Lemma~\ref{lemma: basic properties of R T sets}\,\eqref{case: R T and embeddings}) into any canonical tree with a~higher index, $m\ext T_{\alpha_m}^\textrm{c}$ can be embedded by some $f_m$ into $n_m \ext T_{\pi_\alpha (n_m)}^\textrm{c}$. Defining $f_\emptyset : \emptyset \mapsto \emptyset$ and $f:=f_\emptyset \cup \bigcup_{m\in\omega} f_m$, we have obtained the~desired embedding of $T$ into $T^\textrm{c}_\alpha$. The embedding of $T^\textrm{c}_\alpha$ into $T$ can be obtained by the~same method.

Since $\alpha_{|s|} < \alpha$ for every $s\in \seq$, we can apply the~induction hypothesis  and obtain a~Suslin scheme $\mc C_s$ which is closed in $Y$, covers $X_s$, and satisfies $X_s = \mathbf{R}_{\alpha_m}\left(\mc C_s\right)$.
This in particular gives
\[ X = \bigcap_{m\in\omega} \bigcup_{s\in \omega^m} X_s
	= \bigcap_{m\in\omega} \bigcup_{s\in \omega^m} \mathbf{R}_{\alpha_m}\left(\mc C_s\right) . \]
Analogously to the~proof of Proposition~\ref{proposition: representation for F 4 sets}, we define
\[ C\left(\varrho(\vec s)\right) := \bigcap_{k=0}^{|\vec s|} C_{\Delta_k (\vec s)}\left( \xi_k(\vec s) \right) . \]
Using the~exactly same arguments as in the~proof of Proposition~\ref{proposition: representation for F 4 sets}, we get that this correctly defines a~Suslin scheme $\mc C$ which is closed in $Y$ and covers $X$.

To finish the~proof, it remains to show that $\mathbf{R}_\alpha(\mc C)\subset X$, which reduces to proving the~inclusion in the~following formula:
\[ \mathbf{R}_\alpha(\mc C) = \mathbf{R}_T(\mc C) \subset \bigcap_{m\in\omega} \bigcup_{s\in\omega^m} \mathbf{R}_{\alpha_m}(\mc C_s) = X . \]
Using the~definition of $\mathbf{R}_T(\cdot)$ and the~fact that $T^m := \{ t' \in \seq | \ m\ext t' \in T \} = T^\textrm{c}_{\alpha_m}$ for any $m\in\omega$, we get
\begin{equation*} \begin{split}
x \in \mathbf{R}_T(\mc C) \iff & \left( \exists \varphi : T \rightarrow \seq \text{ adm.}\right) : x \in \bigcap_{t\in T} C(\varphi(t)) \\
\underset{(iii)}{\overset{L\ref{lemma: construction of admissible mappings}}\iff} & \left( \forall m\in \omega \right) \left( \exists s\in\omega^m \right) \left( \exists \varphi_m : T^m \rightarrow \seq \text{ adm.} \right) : 
		x \in \bigcap_{t'\in T^m} C(s\ext \varphi_m(t')) \\
\iff & \left( \forall m\in\omega \right) \left( \exists s\in \omega^m \right) \left( \exists \varphi_m : T^\textrm{c}_{\alpha_m} \rightarrow \seq \text{ adm.} \right) :
		x \in \bigcap_{t'\in T^\textrm{c}_{\alpha_m}} C(s\ext \varphi_m(t'))		\\
\iff & \left( \forall m\in\omega \right) \left( \exists s\in\omega^m \right) : 
		x \in \mathbf{R}^{s}_{\alpha_m} (\mc C) \\
\iff & x \in \bigcap_{m\in\omega} \bigcup_{s\in\omega^m} \mathbf{R}^{s}_{\alpha_m}(\mc C) .
\end{split} \end{equation*}
We claim that for each $s\in\omega^m$ there exists $\tilde{s}\in \omega^m$ such that $\mathbf{R}^s_{\alpha_m}(\mc C) \subset \mathbf{R}_{\alpha_m}(\mc C_{\tilde{s}})$. Once we prove this claim, we finish the~proof by observing that
\[ \mathbf{R}_T(\mc C) = \bigcap_{m\in\omega} \bigcup_{s\in\omega^m} \mathbf{R}^{s}_{\alpha_m}(\mc C)
	\subset \bigcap_{m\in\omega} \bigcup_{s\in\omega^m} \mathbf{R}_{\alpha_m}(\mc C_{\tilde{s}})
	\subset \bigcap_{m\in\omega} \bigcup_{\tilde s\in\omega^m} \mathbf{R}_{\alpha_m}(\mc C_{\tilde s}) = X . \]
Let $s\in\omega^m$. To prove the~claim, observe first that for any $v\in\seq$ we have
\begin{equation} \label{equation: final inequality in F alpha representation proposition} \begin{split}
 C(s\ext v) & =
 	\bigcap_{k=0}^{m+|v|}C_{\Delta_k(\rho^{-1}(s\ext v))}\left(\xi_k(\rho^{-1}(s\ext v))\right)
	\subset C_{\Delta_m(\rho^{-1}(s\ext v))}\left(\xi_m(\rho^{-1}(s\ext v))\right) \\
 & \overset{L\ref{lemma: properties of auxiliary functions}}{\underset{(d)}=} C_{\Delta_m(\rho^{-1}(s))}\left(\xi_m(\rho^{-1}(s\ext v))\right) 
	= C_{\tilde{s}}\left(\xi_m(\rho^{-1}(s\ext v))\right) 
\end{split} \end{equation}
(where we denoted $\tilde{s}:=\Delta_m(\rho^{-1}(s))\in\omega^m$). By definition, any $x\in \mathbf{R}^s_{\alpha_m}(\mc C)$ satisfies $x\in \bigcap_{t\in T^\textrm{c}_{\alpha_m}} C(s\ext \varphi(t))$ for some admissible $\varphi : T^\textrm{c}_{\alpha_m} \rightarrow \seq$. Taking $v=\varphi(t)$ in the~computation above, we get
\[ x\in \bigcap_{t\in T^\textrm{c}_{\alpha_m}} C(s\ext \varphi(t)) \implies
	x \in \bigcap_{t\in T^\textrm{c}_{\alpha_m}} C_{\tilde{s}}\left(\xi_m(\rho^{-1}(s\ext \varphi(t)))\right) . \]
Since $s\in\omega^m$, it follows from Lemma~\ref{lemma: properties of auxiliary functions} that the~mapping $t\in T^\textrm{c}_{\alpha_m} \mapsto \xi_m\left((\rho^{-1}(s\ext \varphi(t))\right)$ is admissible. It follows that the~intersection on the~right side of \eqref{equation: final inequality in F alpha representation proposition} is contained in $\mathbf{R}_{\alpha_m}(\mc C_{\tilde{s}})$, so $\mathbf{R}^s_{\alpha_m}(\mc C) = \mathbf{R}_{\alpha_m}(\mc C_{\tilde{s}})$ and the~proof is complete.
\end{proof}

To prove Theorem~\ref{theorem: existence of regular representation}, it remains the~prove that every regular representation can be ``made complete without loss of generality'', in sense of the~following proposition.
When we say that a~regular $\Fa$-representation $\mc C$ of $X$ in $Y$ is \emph{complete}, we mean that $\mc C$ is complete on $X$.
A regular $\Fa$-representation $\mc C$ is a~\emph{refinement} of a~regular $\Fa$-representation $\mc D$ when for every $s\in\seq$, there is some $t\in \seq$ with $|t|=|s|$ s.t. $C(s) \subset D(s)$.

\begin{lemma}[Completing regular representations]
	\label{lemma: completing representation}
Let $X$ be $\mc K$-analytic space and $Y\supset X$. For any regular $\Fa$-representation $\mc D$ of $X$ in $Y$, there is a~complete regular $\Fa$-representation $\mc C$ of $X$ in $Y$ which refines $\mc D$.
\end{lemma}

\begin{proof}
Suppose that a~$\mc K$-analytic space $X$ satisfies $X\in\Fa(Y)$ for some $\alpha\leq \omega_1$.
By Proposition~\ref{proposition: K analytic spaces have complete Suslin schemes}, there exists some Suslin scheme $\mc R$ which is complete on $X$.
Let $\mc D$ be a~regular $\Fa$-representation of $X$ in $Y$. We fix an arbitrary bijection $\pi : \omega^2 \rightarrow \omega$ and denote by
\[ \varrho : (s,t) = \Big( \big(s(0),t(0)\big), \dots, \big(s(k),t(k)\big) \Big)
 \in \bigcup_{m\in\omega} \omega^m \times \omega^m \mapsto
 \Big( \pi\big(s(0),t(0)\big),\dots,\pi\big(s(k),t(k)\big) \Big) \in \seq \]
the bijection between pairs of sequences and sequences induced by $\pi$.
Finally, we define a~complete Suslin scheme $\mc C$ as
\[ C(u) := D(\varrho^{-1}_1(u)) \cap R(\varrho^{-1}_2(u)) \ \ \text{ (or equivalently, }
 C(\varrho(s,t)) := D(s) \cap R(t) ).\]
Denoting $\mc C_m:=\{ C_s | \ s\in\omega^m \}$, we obtain a~sequence of covers $(\mc C_m)_m$ which is a~refinement of the~complete sequence $(\mc R_m)_m$.
By a~straightforward application of the~definition of completeness, $(\mc C_m)_m$ is complete on $X$ as well.
In particular, Lemma~\ref{lemma: complete suslin schemes in super spaces} yields
\[ \mathbf{R}_\alpha(\mc C,Y) = \mathbf{R}_\alpha(\overline{\mc C}^Y\!\!)
\supset \mathbf{R}_{\omega_1}(\overline{\mc C}^Y\!\!)
= \mc A(\overline{\mc C}^Y\!\!) = X .\]
For the~converse inclusion, let $x\in \mathbf{R}_\alpha (\mc C,Y)$ and suppose that $\varphi$ is an admissible mapping witnessing\footnote{Recall that an admissible mapping $\varphi$ witnesses that $x$ belongs to $\mathbf{R}_\alpha(\mc C,Y)$ if and only if its domain is the~canonical tree $T^c_\alpha$ and we have $x \in \bigcap_{T^c_\alpha} \overline{C(\varphi(t))}^Y$.} that $x$ belongs to $\mathbf{R}_\alpha (\mc C,Y)$.
The mapping $\psi := \varrho^{-1}_1 \circ \varphi$ is clearly admissible as well, and for any $u\in\seq$, we have
\[ C( \varphi(u)) = D( \varrho^{-1}_1 (\varphi(u)))
 \cap R( \varrho^{-1}_2 ( \varphi(u)))
 \subset D( \varrho^{-1}_1 (\varphi(u)))
 = D( \psi (u) ) .\]
In particular, $\psi$ witnesses that $x$ belongs to $\mathbf{R}_\alpha (\mc D,Y) = X$.
\end{proof}
\section{Complexity of Talagrand's Broom Spaces} \label{section: complexity of brooms}

In this section, we study the so-called ``broom spaces'', based on the non-absolute $\fsd$ space $\mathbf{T}$ from \cite{talagrand1985choquet} (also defined in Definition \ref{definition: AD topology} here).
In \cite{kovarik2018brooms} the author shows that for each \emph{even} $\alpha$, there is a broom space $T$ with $\{2,\alpha\} \subset \textnormal{Compl}(T) \subset [2,\alpha]$.
We improve this results in two ways, by providing conditions under which the set $\textnormal{Compl}(T)$ of complexities attainable by a broom space $T$ is \emph{equal to the whole interval} $[2,\alpha]$ for \emph{any} $\alpha \in [2,\omega_1]$.

The existence of spaces with this property already follows from Theorem \ref{theorem: summary}. The interesting part will instead be the methods used to prove this result.
First of all, to show that a space $T$ attains many different complexities, we need to find many different compactifications of $T$.
In each such $cT$, we need to bound $\Compl{T}{cT}$ from below and from above. To obtain the lower estimate, we describe a refinement of a method from \cite{talagrand1985choquet}. For the upper bound, we use a criterion from Section \ref{section: Suslin scheme rank} (Corollary~\ref{corollary: sufficient condition for Fa}).

This section is organized as follows: In Section \ref{section: broom sets}, we describe the hierarchy of broom \emph{sets}, which are then used in Section \ref{section: broom space properties} to define corresponding hierarchy of topological spaces, called broom spaces. Section \ref{section: broom space properties} also explores the basic complexity results related to broom spaces.
Section \ref{section: Talagrand lemma} reformulates and refines some tools from \cite{talagrand1985choquet}, useful for obtaining complexity lower bounds.
In Section \ref{section: amalgamation spaces}, we study compactifications of broom spaces in the abstract setting of amalgamation spaces. We conclude with Sections \ref{section: cT} and \ref{section: dT}, where all of these results are combined in order to compute the complexities attainable by broom spaces.

We note that while reading the subsections in the presented order should make some of the notions more intuitive, it is not always strictly required; Sections \ref{section: Talagrand lemma} and \ref{section: amalgamation spaces} are completely independent, and their only relation to Section \ref{section: broom space properties} is that broom spaces constitute an example which one might wish to use to get an intuitive understanding of the presented abstract results.

\subsection{Broom sets}\label{section: broom sets}

This section introduces a special sets of finite sequences on $\omega$, called \emph{finite broom sets}, and related sets of infinite sequences on $\omega$, called \emph{infinite broom sets}.
We will heavily rely on the notation introduced in Section~\ref{section: sequences}.

\begin{definition}[Finite broom sets]\label{definition: finite brooms}
A \emph{forking sequence} is a sequence $( f_n )_{n\in\omega}$ of elements of $\seq$ such that for distinct $m,n\in\omega$ we have $f_n(0)\neq f_m(0)$.

We denote $\mc B_{0} := \left\{ \, \{ \emptyset \} \, \right\}$. For $\alpha\in [1,\omega_1]$ we inductively define the hierarchy of (finite) broom sets as
\begin{equation*}
\mc B_{\alpha} := 
\begin{cases}
\ \mc B_{<\alpha} \cup \{ h\ext B | \ B \in \mc B_{\alpha-1} , \ h\in \seq \} & \text{ for } \alpha \text{ odd} \\
\ \mc B_{<\alpha} \cup \{ \bigcup_n f_n \ext B_n | \ B_n \in \mc B_{<\alpha}, \, (f_n)_n \text{ is a forking seq.} \} & \text{ for } \alpha \text{ even.}
\end{cases}
\end{equation*}
By $h_B$ we denote the \emph{handle} of a finite broom set $B$, that is, the longest sequence common to all $s\in B$.
\end{definition}

Note that $h_B$ is either the empty sequence or the sequence $h$ from Definition \ref{definition: finite brooms}, depending on whether $B$ belongs to $\mc B_\alpha \setminus \mc B_{<\alpha}$ for even or odd $\alpha$.

For odd $\alpha$, every $B\in \mc B_\alpha$ can be written as $B = h_B \ext B'$, where $B' \in \mc B_{\alpha-1}$.
Moreover, we have $h_0 \ext B' \in \mc B_\alpha$ for any $h_0 \in \seq$.

We claim that the hierarchy of finite broom sets is strictly increasing and stabilizes at the first uncountable step (the `strictly' part being the only non-trivial one):
\begin{equation} \label{equation: broom hierarchy}
\mc B_{0} \subsetneq \mc B_{1} \subsetneq \mc B_{2} \subsetneq \dots \subsetneq \mc B_{\omega_1} = \mc B_{<\omega_1} .
\end{equation}
Indeed, for odd $\alpha$, if $h \in \seq$ is not the empty sequence and $\alpha$ is the smallest ordinal for which $B\in \mc B_{\alpha-1}$, then it follows from the above definition that $h\ext B \in \mc B_\alpha \setminus \mc B_{<\alpha}$.
For $\alpha$, if $(f_n)_n$ is a forking sequence and $\alpha$ is the smallest ordinal satisfying $\{ B_n | \ n\in \omega \} \subset \mc B_{<\alpha}$, then $\bigcup_n f_n\ext B_n$ belongs to $\mc B_\alpha \setminus \mc B_{<\alpha}$.

By extending each sequence from a finite broom set into countably many infinite sequences (in a particular way), we obtain an \emph{infinite broom set} of the corresponding rank:

\begin{definition}[Infinite broom sets]\label{definition: infinite brooms}
A countable subset $A$ of $\baire$ is \emph{an infinite broom set} if there exists some $B\in \mc B_{\omega_1}$ such that the following formula holds:
\begin{align}\label{equation: A extension}
\left( \exists \{ \nu^s_n | \, s\in B,n\in \omega \} \subset \baire \right)
\left( \exists \text{ forking sequences } (f^s_n)_n, \ s\in B \right) : 
A = \{ s\ext f^s_n \ext \nu^s_n | \ s\in B, n\in\omega \} .
\end{align}
If $A$ and $B$ satisfy \eqref{equation: A extension}, we say that $A$ is a \emph{broom-extension} of $B$.
For $\alpha\in[0,\omega_1]$, the family of $\mc A_\alpha$ consists of all broom-extensions of elements of $\mc B_\alpha$.
\end{definition}

When a broom extension $A$ of $B$ belongs to some $\mc A \subset \mc A_{\omega_1}$, we will simply say that $A$ is an $\mc A$-extension of $B$.
Note that the set $B$ of which $A$ is a broom extension is uniquely determined, so $A\in \mc A_\alpha$ holds if and only if $B\in \mc B_\alpha$.
Using the definition of broom-extension and \eqref{equation: broom hierarchy}, we get
\[ \mc A_{0} \subsetneq \mc A_{1} \subsetneq \mc A_{2} \subsetneq \dots \subsetneq \mc A_{\omega_1} = \mc A_{<\omega_1} .\]

We claim that in Definition \ref{definition: infinite brooms}, we have
\begin{equation} \label{equation: B tilde is B 2 + alpha}
\widetilde B := \{ s\ext f^s_n | \ s\in B, n\in\omega \} \in \mc B_{2+\alpha} .
\end{equation}
(Recall that for infinite $\alpha$, we have $2+\alpha = \alpha$.)
It follows from \eqref{equation: B tilde is B 2 + alpha} that each $\mc A_\alpha$ is actually the family of ``normal infinite extensions'' of certain $\mc B_{2+\alpha}$-brooms.
We define $\mc A_\alpha$ as in Definition \ref{definition: infinite brooms} to simplify notation later (the previous paper of the author, \cite{kovarik2018brooms}, uses the corresponding ``alternative'' numbering).

We omit the proof of \eqref{equation: B tilde is B 2 + alpha}, since it is a simple application of transfinite induction, and the above remark is not needed anywhere in the remainder of this paper.


The following result estimates the rank of broom sets.

\begin{lemma}[Rank of broom sets]\label{lemma: rank of broom sets}
Let $\alpha\in [0,\omega_1]$.
\begin{enumerate}[1)]
\item Let $B \in \mc B_\alpha$. Then
	\begin{enumerate}[(i)]
	\item $\D^{\alpha'}(B)$ is finite and
	\item if $\alpha$ is even, then $\D^{\alpha'}(B)$ is either empty or equal to $\{\emptyset\}$.
	\end{enumerate}
\item Let $A \in \mc A_\alpha$. Then
	\begin{enumerate}[(i)]
	\item $\D^{\alpha'}( \D(A) )$ is finite and
	\item if $\alpha$ is even, then $\D^{\alpha'}( \D(A) )$ is either empty or equal to $\{\emptyset\}$.
	\end{enumerate}
\end{enumerate}
\end{lemma}

\begin{proof}
$1)$: The proof of this part is essentially the same as the proof of Proposition 4.11(i) from \cite{kovarik2018brooms}, where an analogous result is proven for derivative
\[ D_{i}(B) := \{ s\in \seq | \ \textnormal{cl}_{\Tr}(B) \text{ contains infinitely many extensions of } s \} .\]
For any $B\in \mc B_{\omega_1}$, the initial segment $\textnormal{cl}_{\Tr}(h_B)$ is finite. Moreover, for $B \in \mc B_\alpha \setminus \mc B_{<\alpha}$, $\alpha$ even, we have $\textnormal{cl}_{\Tr}(h_B) = \{ \emptyset \}$.
Therefore -- since we clearly have $\D^{\alpha'}(B) \subset D_{i}^{\alpha'}(B)$ -- it suffices to prove 
\begin{equation} \label{equation: B derivative induction hypothesis}
\left( \forall B \in \mc B_{\omega_1} \right) : B \in \mc B_\alpha \implies D^{\alpha'}_i (B) \subset \textnormal{cl}_{\Tr}(h_B) .
\end{equation}

To show that \eqref{equation: B derivative induction hypothesis} holds for $\alpha=0$ and $\alpha=1$, note that every $B\in \mc B_1$ is of the form $B = \{ h \} = \{ h_B \}$ for some $h\in \seq$. Since $1'=0'=0$, such $B$ satisfies
\[ D_i^{1'} ( B ) = D_i^{0'} ( B ) = D_i^{0} ( B ) = \cltr{B} = \cltr{ \{h_B\} } .\]

Let $\alpha = \lambda + 2n$, where $n\geq 1$, and suppose that \eqref{equation: B derivative induction hypothesis} (and hence also $1)$ from the statement) holds for every $\beta < \alpha$. We will prove that \eqref{equation: B derivative induction hypothesis} holds for $\alpha$ and $\alpha+1$ as well.
Let $B\in \mc B_\alpha \setminus \mc B_{<\alpha}$. By definition of $\mc B_\alpha$, we have $B = \bigcup_n h_B \ext f_n \ext B_n$ for some forking sequence $(f_n)_n$, broom sets $B_n \in \mc B_{<\alpha}$ and $h_B = \emptyset$.
Clearly,
\[ \alpha-1= \lambda + 2n = \lambda+2(n-1)+1\]
is an odd ordinal, so we have $B \in \mc B_{\lambda+2(n-1)+1}$. Since $\lambda+2(n-1)+1$ is strictly smaller than $\alpha$ and we have
\[ (\lambda+2(n-1)+1)' = \lambda + n -1 = \alpha'-1 ,\]
the induction hypothesis implies that $D_i^{\alpha'-1} (f_n\ext B_n )$ is finite.
Consequently, $D_i^{\alpha'} ( f_n \ext B_n )$ is empty. It follows that the longest sequence that might be contained in
\[ D_i^{\alpha'} ( B ) = D_i^{\alpha'} ( \bigcup_n h_B \ext f_n \ext B_n ) \]
is $h_B$, which gives \eqref{equation: B derivative induction hypothesis}.

For $B\in \mc B_{\alpha+1} \setminus \mc B_\alpha$, the proof is the same, except that $h_B \neq \emptyset$.

When $\alpha=\lambda$ is a limit ordinal, the proof is analogous.

$2)$: If $A\in \mc A_\alpha$ is a broom-extension of $B$, we observe that $D(A) = \textnormal{cl}_{\Tr}(B)$:
\begin{align}
\D(A) & =  & \{ t\in \seq | \ & \textnormal{cl}_{\Tr}(A) \text{ contains infinitely many incomparable} \nonumber \\
							&&& \text{ extensions of $t$ of different length} \} \nonumber \\
	& = & \{ t\in \seq | \ & A \text{ contains infinitely many extensions of $t$} \}
	\label{equation: rewriting derivative of A} \\
	& = & \{ t\in \seq | \ & \text{the set }  \{ s\ext f^s_n \ext \nu^s_n | \ s\in B, n\in \omega \}
		\text{ contains} \nonumber \\
	&	&	& \text{infinitely many extensions of $t$} \} \nonumber \\	
	& = & \{ t \in \seq | \ & t \sqsubset s \text{ holds for some } s\in B \}
	= \textnormal{cl}_{\Tr}(B) , \nonumber 
\end{align}
The results then follows from $1)$, because the derivatives of $\textnormal{cl}_{\Tr}(B)$ are the same as those of $B$, and $B$ is an element of $\mc B_\alpha$.
\end{proof}


Finally, we describe a particular family of brooms sets, used in \cite{talagrand1985choquet}, which has some additional properties.

Let $\varphi_n : \left( \mc B_{\omega_1} \right)^n \rightarrow \mc S$, for $n\in\omega$, be certain functions (to be specified later), such that for each $(B_n)_{n\in\omega}$, $\left( \varphi_n ( (B_i)_{i<n} \right))_{n\in \omega}$ is a forking sequence.
The collection $\mc B^T$ of \emph{finite Talagrand's brooms} is defined as the closure of $\mc B_0 = \{ \, \{ \emptyset \} \, \}$ in $\mc B_{\omega_1}$ with respect to the operations $B \mapsto h \ext B$, for $h \in \seq$, and
\begin{equation}\label{equation: construction of B T}
(B_n)_{n\in \omega} \mapsto \bigcup_{n\in\omega} f_n \ext B_n, \ \text{ where } f_n = \varphi_n ( (B_i)_{i<n} ) .
\end{equation}
For $\alpha \leq \omega_1$, we set $\mc B_\alpha^T := \mc B^T \cap \mc B_\alpha$.
Clearly, finite Talagrand's brooms satisfy an analogy of \eqref{equation: broom hierarchy}.

The family $\mc A^T$ of \emph{infinite Talagrand's brooms} consists of broom-extensions of elements of $\mc B^T$, where for each $B$, only certain combinations (to be specified later) of $(f^s_n)_n$, $s\in B$, and $\{ \nu^s_n | \ s\in B, n\in \omega \}$ are allowed.
For $\alpha \leq \omega_1$, we set $\mc A_\alpha^T := \mc A^T \cap \mc A_\alpha$.

The precise form of the as-of-yet-unspecified parameters above can be found in \cite{talagrand1985choquet}. However, our only concern is that the following lemma holds. (Recall that a family $\mc A$ is almost-disjoint if the intersection of any two distinct elements of $\mc A$ is finite.)

\begin{lemma}[The key properties of $\mc A^T$ and $\mc B^T$] \label{lemma: properties of Talagrands brooms}
The functions $\varphi_n$ and the ``allowed combinations'' above can be chosen in such a way that the following two properties hold:
\begin{enumerate}[(i)]
\item $\mc A^T$ is almost-disjoint; \label{case: A T is AD}
\item Let $B\in \mc B^T$ and let $L_s$, $s\in B$, be some sets. If each $L_s \cap \mc N(s)$ is $\tau_p$-dense in $\mc N(s)$, then there is some $\mc A^T$-extension $A$ of $B$ s.t. each $A\cap L_s$ is infinite. \label{case: A T is rich}
\end{enumerate}
\end{lemma}

\begin{proof}
$(i)$ follows from \cite[Lemma 4]{talagrand1985choquet}. $(ii)$ follows from \cite[Lemma 3]{talagrand1985choquet}.
\end{proof}

A third important property of Talagrand's broom sets is the following:

\begin{remark}
For Talagrand's brooms, the conclusion of Lemma \ref{lemma: rank of broom sets} is optimal.
\end{remark}

\noindent By ``conclusion being optimal'' we mean that for every $B\in \mc B_\alpha^T \setminus \mc B^T_{<\alpha}$ the derivative $\D^{\alpha'}(B)$ is non-empty, and similarly $\D^{\alpha'}(\D(A))$ is non-empty for every $A\in \mc A_\alpha^T \setminus \mc A^T_{<\alpha}$).
The ``importance'' of this property only appears implicitly -- if this property, Talagrand's brooms would not be useful for our purposes. From this reason, we do not include the proof here, as we will not need the result in the following text. However, it can be proven analogously to Proposition 4.11~(ii) from \cite{kovarik2018brooms}.

\subsection{Broom spaces} \label{section: broom space properties}

Next, we introduce the class of broom spaces and state the main result of Section \ref{section: complexity of brooms}.
A useful related concept is that of a space with a single non-isolated point:

\begin{definition}[Space with a single non-isolated point] \label{definition: 1 non isolated point}
Let $\Gamma$ be a set, $\mc A\subset \mc P(\Gamma)$ and $\infty$ a point not in $\Gamma$. We define the \emph{space with a single non-isolated point} corresponding to $\Gamma$ and $\mc A$ as $\left(\Gamma\cup\{\infty\},\tau(\mc A)\right)$, where $\tau (A)$ is the topology in which
\begin{itemize}
\item each $\gamma\in \Gamma$ is isolated,
\item the neighborhood subbasis of $\infty$ consists of all sets of the form
\begin{align*}
& \{ \infty\} \cup \Gamma \setminus \{\gamma\}	 \ \ \ \text{ for } \gamma \in \Gamma \text{ and} \\
& \{ \infty\} \cup \Gamma \setminus A 			 \ \ \ \ \ \, \text{ for } A \in \mc A .
\end{align*}
\end{itemize}
\end{definition}

Broom spaces are a particular type of spaces with a single non-isolated point:

\begin{definition}[Broom spaces] \label{definition: AD topology}
A \emph{broom space} $T_{\mc A}$ corresponding to a non-empty family $\mc A \subset \mc A_{\omega_1}$ of infinite broom sets is defined as
\[ T_{\mc A} := \left( \baire \cup \{\infty\}, \tau(\mc A) \right) ,\]
that is, as the space with a single non-isolated point corresponding to $\Gamma=\baire$ and $\mc A$.
If the family $\mc A$ is almost-disjoint, the corresponding $T_{\mc A}$ is said to be an \emph{AD broom space}.

We define \emph{Talagrand's broom spaces} as
\begin{align*}
\mathbf{T}_\alpha & :=  T_{\mc A^T_{<\alpha}} && \text{for } \alpha \in [1,\omega_1] \text{ and } \\
\mathbf{S}_\alpha & :=  T_{\mc A^T_{<\alpha} \setminus \mc A^T_{<(\alpha-1)}}
	= T_{\mc A^T_{\alpha-1} \setminus \mc A^T_{<(\alpha-1)}}
	&& \text{for non-limit } \alpha \in [1,\omega_1] .
\end{align*}
\emph{The Talagrand's broom space} $\mathbf{T}$ is defined as $\mathbf{T}:=\mathbf{T}_{\omega_1}$ (emphasis on `\emph{the}'). 
\end{definition}

Formally, the main result of Section \ref{section: complexity of brooms} is the following theorem. (But as we mentioned earlier, we consider the methods to be the more interesting part.)

\begin{theorem}[Complexity of Talagrand's broom spaces]\label{theorem: complexity of talagrands brooms}
Talagrand's broom spaces satisfy
\begin{align*}
\alpha \in [2,\omega_1] \implies &
\textnormal{Compl}\left(\mathbf{T}_\alpha\right) = \left[2,\alpha\right] \\
\alpha \in [2,\omega_1] \text{ is non-limit}  \implies &
\textnormal{Compl}\left(\mathbf{S}_\alpha\right) = \left[2,\alpha\right] .
\end{align*}
\end{theorem}

For every topological space $Y\supset T$ and $y\in Y$, we set
\begin{align} \label{equation: S of y}
S(y) := \left\{ s\in\seq | \ \overline{\mc N(s)}^{Y} \ni y \right\}.
\end{align}

\noindent The following computation appears several times throughout the paper, so we formulate it separately.

\begin{lemma}[Sufficient condition for being $\Fa$]\label{lemma: S(y) and complexity of T}
Let $T$ be a broom space, $Y\supset T$ a topological space, and $\alpha\in [2,\omega_1]$. Suppose that for each $y\in Y\setminus T$, $S(y)$ is either finite or it can be covered by finitely many sets $\D(A)$, $A\in \mc A_{<\alpha}$. Then $T\in \mc F_\alpha(Y)$.
\end{lemma}

\noindent The proof of this lemma is the only place where Section \ref{section: complexity of brooms} relies on the concept of regular $\Fa$-representations and the results of Section \ref{section: regular representations}. Reader willing to treat this lemma as a ``black box'' is invited to ignore Section \ref{section: regular representations} altogether, and skip ahead to Proposition \ref{proposition: basic broom complexities}.

In the terminology of Section \ref{section: regular representations}, Lemma \ref{lemma: S(y) and complexity of T} actually shows that $\mc N$ is a regular $\Fa$-representation of $T$ in $Y$, in the sense of Definition \ref{definition: regular representations}. This  immediately follows from the remark just above Corollary \ref{corollary: sufficient condition for Fa}.

To prove Lemma \ref{lemma: S(y) and complexity of T} for a broom space $T$ contained in a topological space $Y$, we first need to have a Suslin scheme $\mc C$ in $T$ which satisfies $\mc A(\overline{\mc C}^Y\!\!) = T$.
It turns out that the basic open sets $\mc N(s) = \{ \sigma\in\baire | \ \sigma \sqsupset s \}$ from the product topology of $\baire$ provide a canonical solution:

\begin{lemma}[$\mc N$ as universally useful Suslin scheme] \label{lemma: A of N}
For every broom space $T$, the Suslin scheme $\mc N = (\mc N(s) )_{s\in\seq}$ on $\baire\subset T$ satisfies $\mc A(\overline{\mc N}^Y\!\!) = T$ for every topological space $Y\supset T$.
\end{lemma}

\begin{proof}
This holds, for example, by \cite[Lemma 5.10]{kovarik2018brooms}.
\end{proof}

Note that $S(\cdot) = S_{\mc N}(\cdot)$, so $S(\cdot)$ is a special case of the general definition in \eqref{equation: S C of y}. By definition of broom topology, the closure of any uncountable subset of $\baire$ contains $\infty$.
It follows that $\overline{\mc N(s)}^{Y} = \overline{\mc N(s)\cup \{\infty\}}^{Y}$ holds for every $s\in \seq$. In other words, we could equally well work with a Suslin scheme $(\mc N(s) \cup \{\infty\})_s$ which covers $X$ (but this scheme would not be disjoint).

\begin{proof}[Proof of Lemma \ref{lemma: S(y) and complexity of T}]
Let $T\subset Y$ and $\alpha$ be as in the lemma.
We shall verify that every $y\in Y \setminus T$ satisfies the assumptions of Corollary \ref{corollary: sufficient condition for Fa}, obtaining $T\in \Fa(Y)$ as a result.
In particular, we shall show that $\D^{\alpha'} (S(y))$ is empty when $\alpha$ is even, resp. that it can only contain the empty sequence when $\alpha$ is odd.

Let $y\in Y\setminus T$.
When $S(y)$ is finite, we have $\D^{\alpha'}(S(x)) = \emptyset$, independently of the parity of $\alpha$ (because $\alpha\geq 2$ implies $\alpha' \geq 1$ and $\D^{\alpha'}(\cdot) \subset \D(\cdot)$, and $\D$-derivative of a finite set is empty).

Suppose that $\alpha$ is odd. We then have
\[ \D^{\alpha'}(S(y)) \subset \D^{\alpha'} ( \bigcup_{\mc A'} \D(A) )
 = \bigcup_{\mc A'} \D^{\alpha'} ( \D(A) ) 
 \overset{\alpha-1}{\underset{\leq \alpha}=}  \bigcup_{\mc A'} \D^{(\alpha-1)'} ( \D(A) ).\]
Since $\mc A' \subset \mc A \subset \mc A = \mc A_{<\alpha} = \mc A_{\alpha-1}$, it follows from $2)$\,$(ii)$ in Lemma \ref{lemma: rank of broom sets} that each $\D^{(\alpha-1)'} ( \D(A) )$ is either empty or only contains the empty sequence.
This gives $\D^{\alpha'}(S(y)) \subset \{ \emptyset \}$.

Suppose that $\alpha$ is even. Let $A\in \mc A'$. Since $\mc A' \subset \mc A_{<\alpha}$, we have $A \in \mc A_{\beta}$ for some $\beta<\alpha$.
We have $\beta' < \alpha'$ (because $\alpha$ is even), which implies $\D^{\alpha'}(\cdot) \subset \D (\D^{\beta'} (\cdot))$.
Since $\D^{\beta'} (\D (A))$ is finite (by $2)$\,$(i)$ from Lemma \ref{lemma: rank of broom sets}), $\D(\D^{\beta'} (\D (A)))$ is empty, and thus $D^{\alpha'} ( \D (A) )$ is empty as well.
As in the case of odd $\alpha$, this gives $\D^{\alpha'}(S(y)) = \emptyset$.
\end{proof}

\noindent Actually, the proof of Lemma \ref{lemma: S(y) and complexity of T} did not use any special properties of $T$ and $S(\cdot)$, so the lemma holds when $T$ is replaced by an abstract topological space $X$ and $S(\cdot)$ is replaced by $S_{\mc C}(\cdot)$ for some complete Suslin scheme $\mc C$ on $X$.

\bigskip
We are now in a position to obtain the following (upper) bounds on the complexity of broom spaces:

\begin{proposition}[Basic broom space complexity results]\label{proposition: basic broom complexities}
Any broom space $T_{\mc A}$ satisfies:
\begin{enumerate}[(i)]
\item \cite{talagrand1985choquet} $T_{\mc A}$ is $\fsd$ in $\beta T_{\mc A}$, but it is not $\sigma$-compact;
\item If $\mc A \subset \mc A_{<\alpha}$ holds for some $\alpha\in[2,\omega_1]$, then $T_{\mc A}$ is an absolute $\Fa$ space.
\end{enumerate}
\end{proposition}

(In the language of Section \ref{section: regular representations}, (ii) even proves that $\mc N$ is a universal regular representation of $T_{\mc A}$, in the sense of Definition \ref{definition: regular representations}. This follows from the remark below Lemma \ref{lemma: S(y) and complexity of T}.)

\begin{proof}
(i) is proven in \cite[p.\,197]{talagrand1985choquet}, so it remains to prove (ii).
For even $\alpha$, the result is proven in \cite[Prop.\,5.13]{kovarik2018brooms} (but the proof below also applies).

Suppose we have $\mc A \subset \mc A_{<\alpha}$ for $\alpha\in[2,\omega_1]$ and let $Y\supset T_{\mc A}$ be a topological space.
Let $y\in Y \setminus T_{\mc A}$.
By Lemma \ref{lemma: S(y) and complexity of T}, it suffices to prove that $S(y)$ can be covered by finitely many sets $\D(A)$, $A\in \mc A \subset \mc A_{<\alpha}$.

Since $y\neq \infty$, there must be some open neighborhood $V$ of $y$ in $Y$ which satisfies $\infty \notin\overline{V}^Y$.
By definition of topology $\tau (\mc A)$, $\overline{V}^Y \cap T_{\mc A}$ can be covered by a union of some finite $\mc A' \subset \mc A$ and a finite set $F \subset \baire$. It follows that $S(y) = \bigcup_{\mc A'} \D(A)$ holds for some finite $\mc A' \subset \mc A$:
\begin{align} \label{equation: S of y is D of A}
& s \in S(y)
 \overset{\eqref{equation: S of y}}\iff y \in \overline{\mc N(s)}^{Y}
 \iff y \in \overline{\mc N(s)}^{Y} \cap \overline{V}^Y
 \overset{\text{def.}}{\underset{\text{of }\tau(\mc A)}\iff}
 y \in \nonumber \overline{\mc N(s)\cap \left(\bigcup \mc A' \cup F \right)}^{Y} \\
& \overset{y\notin T}\implies \mc N(s)\cap \left(\bigcup \mc A' \cup F \right) \textrm{ is infinite} \nonumber
 \iff \mc N(s)\cap A \textrm{ is inf. for some } A \in \mc A' \\
& \iff \text{some } A\in \mc A' \textrm{ contains infinitely many incomparable extensions of } s \nonumber \\
& \overset{\eqref{equation: rewriting derivative of A}}{\iff} s \in \D(A) \text{ for some } A\in \mc A'
\iff s \in \bigcup_{\mc A'} \D(A).
\end{align}
\end{proof}

\subsection{Talagrand's lemma}\label{section: Talagrand lemma}

In this subsection, we consider the following general problem:
We are given a $\mc K$-analytic space $X$, and a space $Y$ which contains it. We ``know'' $X$, but not its complexity in $Y$ -- we would like to obtain a lower bound on $\Compl{X}{Y}$.
To solve the problem, we take an arbitrary ``unkown'' set $Z\subset Y$, about which we however \emph{do} know that it belongs to $\Fa(Y)$. The mission shall be successful if we show that any such $Z$ which contains a big enough part of $X$ must also contain a part of $Y\setminus X$.
(Of course, the values of $\alpha$ for which this can work depend on the topology of $X$ and $Y$, so some more assumptions will be needed. In Sections \ref{section: cT} and \ref{section: dT}, we consider some suitable combinations of $X$ and $Y$.)

The method we adopt closely mimics the tools used in \cite{talagrand1985choquet} (namely, Lemma 1 and Lemma 3), and refines their conclusions. The author of the present article has found it hard to grasp the meaning of these tools. From this reason, Notations \ref{notation: witness} and \ref{notation: B corresponding to Z} introduce some auxiliary notions, which -- we hope -- will make the statements easier to parse.

The setting we consider is as follows:
\begin{itemize}
\item $X$ is a $\mc K$-analytic space of cardinality at least continuum.
	\begin{itemize}
	\item We assume that $X$ contains $\baire$ as a \emph{subset}.
	\item By default, $\baire$ is equipped with the subspace topology inherited from $X$. 	This subspace topology need not have any relation with the product topology $\tau_p$ 		of $\baire$; whenever $\tau_p$ is used, it shall be explicitly mentioned.
	\end{itemize}
\item $Y$ is a topological space containing $X$.
\item $Z$ is a subset of $Y$.
\end{itemize}

To get a quicker grasp of the notions, we can just imagine that
\[ \baire \subset X = Z \in \Fa(Y) \ \text{ for some } \alpha<\omega_1 ,\]
and we are aiming to arrive at a contradiction somewhere down the road.

The first auxiliary notion is that of ``witnessing'':\footnote{The author is aware of the fact that the word ``witness'' from Notation~\ref{notation: witness} is neither very original, nor particularly illuminating. Any ideas for a more fitting terminology would be appreciated. The same remark applies to the notion of ``correspondence'' from Notation~\ref{notation: B corresponding to Z}.}

\begin{notation}[Witnessing]\label{notation: witness}
$W\subset Y$ is a \emph{witness of $Z$ in $Y$} if there exists an indexing set $I$ and a family $\{ L_i | \ i\in I \}$ of closed subsets of $Y$ satisfying
\begin{enumerate}[(i)]
\item $\bigcap_{i\in I} L_i \subset Z$;
\item $\left( \forall i \in I \right) : L_i \cap W \text{ is infinite}$.
\end{enumerate}
\end{notation}

Of course, any infinite $Z$ which is closed in $Y$ is its own witness, since we can just set $I := \{i_0\}$ and $L_{i_0} := Z$.
But when the complexity of $Z$ is higher, we might need a much larger collection.
In general, there might even be no witnesses at all -- such as when $Z$ is discrete and $Y$ is its one-point compactification.

The key property of this notion is the following observation (used in \cite{talagrand1985choquet}). It shows that sets with certain properties cannot be witnesses, because they force the existence of points outside of $Z$.

\begin{lemma}[No discrete witnesses in certain spaces]\label{lemma: closed discrete witnesses}
Let $W$ be a closed discrete subset of $Z$. If $Y$ is s.t. $\overline{W}^Y$ is the one-point compactification of $W$, then $W$ cannot be a witness for $Z$ in $Y$.
\end{lemma}

For any discrete space $D$, we denote its one-point compactification as $\alpha D =: D\cup \{x_D\}$. We have
\begin{equation} \label{equation: closure in alpha D}
F\subset D \text{ is infinite } \implies \overline{F}^{\alpha D} = F \cup \{x_D\} .
\end{equation}

\begin{proof}[Proof of Lemma \ref{lemma: closed discrete witnesses}]
Suppose that $W$ is closed and discrete in $Z$ and $\overline{W}^Y$ is the one-point compactification of $W$ (and hence $x_W \notin Z$).
For contradiction, assume that there exists a family $\{ L_i \, | \ i\in I \}$ as in Notation~\ref{notation: witness}. Since each $L_i \cap W$ is infinite by $(ii)$, we have
\[ x_W \overset{\eqref{equation: closure in alpha D}}{\in} \overline{L_i \cap W}^Y
\subset \overline{L_i}^Y = L_i .\]
By $(i)$, we have $x_W \in \bigcap_I L_i \subset Z$ -- a contradiction.
\end{proof}

Note that whenever $W\subset Y$ is a witness of $Z$ in $Y$, then so is any $\widetilde W$ with $W\subset \widetilde W \subset Y$. However, it is more practical to have small witnesses, as it gives us more control over which elements appear in $\overline{L_i \cap W}^Y$.

Regarding positive results, we have the following refinement of \cite[Lemma 3]{talagrand1985choquet}:

\begin{lemma}[Existence of broom witnesses]\label{lemma: Talagrand lemma 3}
If $X\in \Fa(Y)$, then $X$ has an $\mc A^T_\alpha$-witness in $Y$.
\end{lemma}

As we will see later (in Lemma \ref{lemma: c T - lower bound} and Lemma \ref{lemma: d T - lower bound}), combining Lemma \ref{lemma: Talagrand lemma 3} and Lemma \ref{lemma: closed discrete witnesses} yields a lower bound on $\Compl{X}{Y}$ for certain $X$ and $Y$ (more precisely, for those $X\subset Y$ where each $A\in \mc A^T_\alpha$ is closed discrete in $X$, and satisfies $\overline{A}^Y = \alpha A$).
Before giving the proof of Lemma \ref{lemma: Talagrand lemma 3}, we first need some technical results.

Setting $W=\baire$ and $I=B\subset \seq$ in Notation \ref{notation: witness}, and strengthening the condition (ii), we obtain the following notion of ``correspondence''.
Its purpose is to allow the construction of $\mc A^T$-witnesses via Lemma \ref{lemma: properties of Talagrands brooms}\,\eqref{case: A T is rich}.

\begin{notation}[Correspondence]\label{notation: B corresponding to Z}
A set $B\subset \seq$ \emph{corresponds to $Z$} (in $Y$) if there exists a family $\{ L_s | \ s\in B \}$ of closed subsets of $Y$ satisfying
\begin{enumerate}[(i)]
\item $\bigcap_{s\in B} L_s \subset Z$;
\item $\left( \forall s \in B \right) : L_s \cap \mc N(s) \text{ is $\tau_p$-dense in } \mc N(s)$.
\end{enumerate}
\end{notation}

Talagrand's Lemma 1 says that when $X$ is $\mc F$-Borel in $Y$, there is a $\mc B^T$ broom set which corresponds to $X$. We refine this result to include the exact relation between complexity of $X$ and the ``rank'' of the corresponding broom set. The proof itself is identical to the one used in \cite{talagrand1985choquet} -- the non-trivial part was finding the ``right definitions'' for $\Fa$ and $\mc B_\alpha$ such that the correspondence holds.

\begin{lemma}[Exicence of corresponding brooms]\label{lemma: Talagrand lemma 1}
Let $\alpha<\omega_1$. If $X \in \Fa(Y)$, then there is some $B\in \mc B^T_\alpha$ which corresponds to $X$ in $Y$.
\end{lemma}

\begin{proof}
Recall that if $B\in \mc B^T_{\alpha+1}$ holds for even $\alpha$ and $h_B= \emptyset$, we actually have $B\in \mc B^T_{\alpha}$.
We shall prove the following stronger result. (Setting $Z:= X$ and $h:=\emptyset$ in Claim \ref{claim: Talagrand lemma 1} gives the conclusion of Lemma \ref{lemma: Talagrand lemma 1}.)

\begin{claim}\label{claim: Talagrand lemma 1}
Suppose that $Z\in\Fa(Y)$ and $Z\cap \mc N(h)$ is $\tau_p$-dense in $\mc N(h)$ for some $h\in \seq$.
\begin{enumerate}[(i)]
\item For odd $\alpha$, there is $B\in \mc B^T_\alpha$ with $h_B \sqsupset h$ which corresponds to $Z$ in $Y$.
\item For even $\alpha$, there is $B\in \mc B^T_{\alpha+1}$ with $h_B = h$ which corresponds to $Z$ in $Y$.
\end{enumerate}
\end{claim}

Let $Z$ and $h$ be as in the assumption of the claim. We shall prove the conclusion by transfinite induction.
For $\alpha=0$, $Z$ is closed, so we set $B:=\{h\}$ and $L_h := Z$.

Suppose that $\alpha$ is odd and the claim holds for $\alpha-1$. We have $Z=\bigcup_n Z_n$ for some $Z_n \in \mc F_{\alpha-1}(Y)$.
By the Baire theorem, some $Z_{n_0}$ is non-meager in $(\mc N(h),\tau_p)$. It follows that $Z_{n_0}$ is $\tau_p$ dense in $\mc N(h_0)$ for some $h_0 \sqsupset h$.
By the induction hypothesis, there is some $B\in \mc B^T_{(\alpha-1)+1} = \mc B^T_{\alpha}$ with $h_B = h_0 \sqsupset h$ which corresponds to $Z_{n_0}$. In particular, this $B$ also corresponds to $Z\supset Z_{n_0}$.

Suppose that $\alpha\in (0,\omega_1)$ is even and the the claim holds for every $\beta < \alpha$.
We have $Z=\bigcap_n Z_n$, where $Z_n \in \mc F_{\alpha_n}(Y)$ for some odd $\alpha_n < \alpha$.
Suppose we have already constructed $(f_i)_{i<n}$, $(B'_i)_{i<n}$ and $(B_i)_{i<n}$ for some $n\in\omega$.
Let $f_n := \varphi_n((B'_i)_{i<n})$ (where $\varphi_n$ is the function from Lemma \ref{lemma: properties of Talagrands brooms}). By the induction hypothesis, there is some $B_n \in \mc B^T_{\alpha_n}$ which corresponds to $Z_n$ and satisfies $h_{B_n} \sqsupset h\ext f_n$.
As noted in Section \ref{section: broom sets}, $B_n$ can be rewritten as $h\ext f_n \ext B'_n$, where $B'_n \in \mc B^T_{\alpha_n} \subset \mc B^T_{<\alpha}$.

Once we have $f_n$ and $B'_n$ for every $n\in\omega$, we get a broom set $B' := \bigcup_n f_n\ext B'_n$. Since we have both $B' \in \mc B_\alpha$ and $B'\in \mc B^T$, $B$ belongs to $\mc B^T_\alpha$. It follows that $B:=h\ext B' \in \mc B^T_{\alpha+1}$.

It remains to prove that $B$ corresponds to $Z$. Since each $B_n$ corresponds to $Z_n$, there are some closed sets $L_s$, $s\in B_n$, such that $\bigcap_{B_n} L_s \subset Z_n$ (and $(ii)$ from Notation \ref{notation: B corresponding to Z} holds). Since $B=\bigcup_n h\ext f_n \ext B'_n = \bigcup_n B_n$, we have
\[ \bigcap_{s\in B} L_s = \bigcap_{n\in \omega} \bigcap_{s\in B_n} L_s \subset \bigcap_{n\in\omega} Z_n \subset Z ,\]
which shows that $B$ corresponds to $Z$.
\end{proof}

Lemma \ref{lemma: Talagrand lemma 3} is now a simple corollary of Lemma \ref{lemma: Talagrand lemma 1}:

\begin{proof}[Proof of Lemma \ref{lemma: Talagrand lemma 3}]
Let $B \in \mc B^T_\alpha$ be the set which corresponds to $X$ in $Y$ by Lemma \ref{lemma: Talagrand lemma 1}. Let $L_s$, $s\in B$, be the closed subsets of $Y$ as in Notation \ref{notation: B corresponding to Z}.
By Lemma \ref{lemma: properties of Talagrands brooms}, there is an $\mc A^T$-extension $A$ of $B$, such that each $L_s \cap A$ is infinite.
Setting $I := B$, we see that $A$ is a witness for $X$ in $Y$. Since we also have $A\in \mc A_\alpha$ (by definition of $\mc A_\alpha$), we have $A\in \mc A^T_\alpha$ and the proof is complete.
\end{proof}

We will also need the following technical version of Lemma \ref{lemma: Talagrand lemma 3}. As it contains no particularly novel ideas, we recommend skipping it on the first reading.

\begin{lemma}[Existence of broom witnesses -- technical version]\label{lemma: technical version of Talagrand lemma 3}
Suppose that $X\in \mc F_\eta (Y)$ holds for some $\eta < \omega_1$. Then for every $\delta \in [\eta,\omega_1)$, there is some $A\in \mc A^T_\delta \setminus A^T_{<\delta}$ and $h_0\in \seq$ s.t. $A \cap \mc N(h_0) \in \mc A^T_\eta$ is a witness for $X$ in $Y$.
\end{lemma}

\begin{proof}
Assume that $X\in \mc F_\eta (Y)$ and $\eta \leq \delta < \omega_1$.
First, construct a set $B\in \mc B^T_\delta \setminus \mc B^T_{<\delta}$ and $h_0 \in \seq$, such that the following set $B_0$ belongs to $\mc B^T_\eta$ and corresponds to $X$:
\[ B_0 := \{ s\in B| \ s\sqsupset h_0 \} .\]

Denote $h_{\textnormal{even}}:=\emptyset$ and choose an arbitrary $h_{\textnormal{odd}} \in \seq \setminus \{\emptyset\}$ (this notation is chosen merely so that it is simpler to explain how the construction differs depending on the parity of $\delta$).

Suppose first that $\delta$ is even.
Recall that $\varphi_n$, $n\in\omega$, are the functions from Lemma \ref{lemma: properties of Talagrands brooms}.
Set $f_0 := \varphi_0 (\emptyset)$, $h_0 := h_{\textnormal{even}}\ext f_0$.
By Claim \ref{claim: Talagrand lemma 1}, there is some $B_0 \in \mc B^T_\eta$ which corresponds to $Z$ and satisfies $h_{B_0} \sqsupset h_0$. Denote by $B'_0$ the $\mc B^T_\eta$-set satisfying $B_0 = h\ext f_0 \ext B'_0$.
We either have $B_0 \in \mc B^T_\delta \setminus \mc B^T_{<\delta}$ (in which case we set $B := B_0$ and the construction is complete), or $B_0, B'_0 \in \mc B^T_{<\delta}$.

Assume the second variant is true.
Let $(B'_n)_{n=1}^\infty$ be such that $\delta$ is the smallest ordinal s.t. each $B'_n$, is contained in $\mc B^T_{<\delta}$.
For $n\geq 1$, we set $f_n := \varphi_n ( (B'_i)_{i<n})$.
By \eqref{equation: construction of B T}, the following set $B$ belongs to $\mc B^T_{\delta}$:
\[ B := \bigcup_n f_n \ext B'_n = \bigcup_n h_{\textnormal{even}}\ext f_n \ext B'_n .\]
The choice of $(B'_n)_{n=1}^\infty$ ensures that $B \in \mc B^T_{\delta} \setminus \mc B^T_{<\delta}$.

To get the result for odd $\delta=\tilde \delta +1$, we just repeat the above process with $h_{\textnormal{odd}}$ in place of $h_{\textnormal{even}}$ and $\tilde \delta$ in place of $\delta$.

We now construct $A$ with the desired properties.
Let $L_s$, $s\in B_0$, be some sets which ensure that $B_0$ corresponds to $X$ and denote $L_s := \overline{X}^Y$ for $s\in B\setminus B_0$.
By Lemma~\ref{lemma: properties of Talagrands brooms}, there is some $\mc A^T$-extension $A$ of $B$ such that each $L_s \cap A$ is infinite.
By definition of $\mc A^T_{(\cdot)}$, we have $A\in \mc A^T_\delta \setminus A^T_{<\delta}$.
Moreover, $A \cap \mc N(h_0)$ is a broom-extension of $B_0\in \mc B^T_\eta$, which gives $A \cap \mc N(h_0) \in \mc A^T_\eta$.
Finally, $L_s \cap A = L_s \cap (A\cap \mc N(h_0))$ is infinite for each $s\in B_0$, and $\bigcap L_s \subset X$ holds even when the intersection is taken over $s\in B_0$. This proves that $A \cap \mc N(h_0)$ is a witness of $X$ in $Y$.
\end{proof}

In the~next part, we study broom spaces and their compactifications in a~more abstract setting. The purpose is to isolate the~few key properties which are required to obtain the~results we need, while ignoring all the~other details.

\subsection{Amalgamation spaces} \label{section: amalgamation spaces}

In Section~\ref{section: amalgamation spaces} (and only here), we retract our standing assumption that every topological space is Tychonoff.

In Section~\ref{section: zoom spaces}, we were able to take a~space $Z(cY,\mc X)$ and construct its compactification $Z(cY,c\mc X)$ by separately extending each $X_i$ into $cX_i$. Recall that this was easily doable, since the~spaces $X_i$ were pairwise disjoint and clopen in $Z(cY,X_i)$.
Our goal is to take an AD broom space $T_{\mc A}$, extend each $A\in \mc A$ separately into a~compactification $cA$, and thus obtain a~compactification of $T_{\mc A}$. However, the~family $\mc A$ is not disjoint, so we need a~generalization of the~approach from Section~\ref{section: zoom spaces}.

We will only need the~following properties of broom spaces:

\begin{lemma}[Example: Broom spaces]\label{lemma: broom spaces and A 1-4}
Any AD broom space $X=T_{\mc A}$ satisfies the~following four conditions:
\begin{itemize}
	\item[$(\mc A 1)$] $\mc A$ consists of clopen subsets of $X$.
	\item[$(\mc A 2)$] For distinct $A,A'\in \mc A$, the~intersection $A\cap A'$ is compact.
	\item[$(\mc A 3)$] $K := X \setminus \bigcup \mc A$ is compact.
	\item[$(\mc A 4)$] Whenever $\mc U$ is a~collection of open subsets of $X$ which covers $K$, there exists a~finite family $\mc A' \subset \mc A$, s.t. for every $A\in \mc A \setminus \mc A'$, we have $U \cup \bigcup \mc A' \supset A$ for some $U\in \mc U$.
\end{itemize}
\end{lemma}

To summarize the~properties, we can say that the~family $\mc A$ consists of clopen sets with small intersections. The third and fourth condition then ensure, in somewhat technical manner, that the~only parts where the~space $X$ is not compact are the~sets $A\in \mc A$.	

\begin{proof}
Let $T_{\mc A}$ be an AD broom space.
$(\mc A 1)$ holds by the~definition of topology $\tau(\mc A)$ on $T_{\mc A}$.
$(\mc A 2)$ because the~intersection of two distinct elements of $\mc A$ is, in fact, even finite.

$(\mc A 3)$: Often, $\mc A$ will cover the~whole space $T_{\mc A}$ except for $\infty$. In this case, $(\mc A3)$ holds trivially.
However even in the~non-trivial case where $K = T_{\mc A} \setminus \bigcup \mc A$ is infinite, the~subspace topology on $K = T_{\mc A} \setminus \bigcup \mc A$ coincides with the~topology of one-point compactification of $\baire \setminus \bigcup \mc A$ (by Definition~\ref{definition: 1 non isolated point}). This shows that $K$ is compact.

$(\mc A 4)$: Whenever $\mc U$ is an open cover of $K$, there is some $U \in \mc U$ which contains $\infty$. By definition of $\tau(\mc A)$, $U$ contains some basic open set $T_{\mc A} \setminus (\bigcup \mc A' \cup F)$, where $\mc A' \subset \mc A$ and $F \subset K$ are finite.
The fact that $T_{\mc A}$ satisfies $(\mc A4)$ for $\mc U$ is then witnessed by $\mc A'$ and $U$. Note that $U$ is a~universal witness, as we have $U \supset A \setminus \bigcup \mc A'$ for \emph{every} $A\in \mc A\setminus A'$.
\end{proof}

In the~remainder of this section, we will work with an abstract topological space $X$ and a~fixed family $\mc A \subset \mc P(X)$ s.t. the~conditions $(\mc A1)-(\mc A4)$ from Lemma~\ref{lemma: broom spaces and A 1-4} hold.

Let $\mc E=\left( E(A)\right)_{A\in \mc A}$ be a~collection of topological spaces such that for each $A\in \mc A$, $A$ is a~dense subset of $E(A)$.
Informally speaking, our goal is to find a~space whose ``local behavior'' is ``$E(A)$-like'', but the~``global properties'' are similar to those of $X$.

Without yet defining any topology on it, we set $\Amg{X}{\mc E} := X\cup\bigcup_{\mc A} E(A)$, assuming that each $A$ is ``extended into $E(A)$ separately'':
\begin{align} \label{equation: amalgamation set}
(\forall A \in \mc A) : 	& \ X\cap E(A) = A 	\nonumber \\
(\forall A, A' \in \mc A) : & \ A\neq A' \implies E(A)\cap E(A') = A \cap A'
\end{align}

To define the~topology on $\Amg{X}{\mc E}$, we first need the~following lemma.

\begin{lemma}[Largest open set with given trace] \label{lemma: W A U}
Let $P\subset Q$ be topological spaces and $G,G'$ open subsets of $P$.
Denote by $W^Q_P(G)$ the~largest open subset $W$ of $Q$ which satisfies $W \cap P = G \cap P$.
\begin{enumerate}[(i)]
\item $W^Q_P(G)$ is well defined. \label{case: W A U well defined}
\item $G\subset G' \implies W^Q_P(G) \subset W^Q_P(G')$. \label{case: W A U monotonicity}
\item $W^Q_P(G\cap G') = W^Q_P(G) \cap W^Q_P(G')$ .\label{case: W A U intersection}
\item If $P$ is dense in $Q$, we have $W^Q_P(G) \subset \textnormal{Int}_Q \, \overline{G}^Q$.
	\label{case: W A U is subset of ...}
\item For any compact $C\subset P$, we have $W^Q_P(P\setminus C) = Q \setminus C$.
	\label{case: W A U with compact complement}
\end{enumerate}
\end{lemma}

\begin{proof}
Let $P,Q,G$ and $G'$ be as in the~statement.

\eqref{case: W A U well defined}: Since $P$ is a~topological subspace of $Q$, there always exists \emph{some} open subset $W$ which satisfies $W \cap P = G$.
Consequently, we can define $W^Q_P(G)$ as
\begin{equation}\label{equation: W A U formula}
W^Q_P(G) := \bigcup \{ W \subset Q \text{ open } | \ W \cap A = G \} .
\end{equation}

\eqref{case: W A U monotonicity}:
Suppose that $G\subset G'$. The set $W' := W^Q_P(G) \cup W^Q_P(G')$ is open in $Q$ and satisfies
\[ W' \cap P = (W^Q_P(G) \cap P ) \cup (W^Q_P(G') \cap P)
= G \cup G' = G' .\]
Applying \eqref{equation: W A U formula} to $G'$, we have $W' \subset W^Q_P(G')$. It follows that $W^Q_P(G) \subset W^Q_P(G')$.

\eqref{case: W A U intersection}: 
``$\subset$'' follows from \eqref{case: W A U monotonicity}.
``$\supset$'' holds by \eqref{equation: W A U formula}, since $W^Q_P(G) \cap W^Q_P(G')$ is open in $Q$ and satisfies
\[ W^Q_P(G) \cap W^Q_P(G') \cap P = (W^Q_P(G) \cap P) \cap ( W^Q_P(G') \cap P) = G \cap G' .\]

\eqref{case: W A U is subset of ...}:
$P$ is dense in $Q$ and $W^Q_P(G) \subset Q$ is open.
Consequently, $W^Q_P(G) \cap P$ is dense in $W^Q_P(G)$ and we get
\[ W^Q_P(G) \subset \overline{W^Q_P(G) \cap P}^Q = \overline{G}^Q .\]
Since $W^Q_P(G)$ is open, the~conclusion follows.

\eqref{case: W A U with compact complement}: This is immediate, because $Q \setminus C$ is open in $Q$.
\end{proof}

For an open subset $U$ of $X$, we define
\[ V_U := U \cup \bigcup \{ W^{E(A)}_A(U\cap A) \, | \ A\in \mc A \} .\]
It follows from Lemma~\ref{lemma: W A U} that these sets satisfy
\begin{align}
U, U' \subset X \text{ are open in } X \implies \label{equation: V U cap}
& V_{U\cap U'} = V_U \cap V_{U'} .
\end{align}

\begin{definition}[Amalgamation space]\label{definition: amalgamation space}
The \emph{amalgamation of $X$ and $\mc E$} is defined as the~set $\Amg{X}{\mc E}$, equipped with the~topology whose basis\footnote{We do not claim that it is obvious that $\mc B$ is a~basis of topology. This is the~content of Lemma~\ref{lemma: Amg definition is correct}.} $\mc B$ consists of all sets of the~form
\begin{itemize}
\item $W\subset E(A)$, where $W$ is open in $E(A)$ and $A\in \mc A$;
\item $V_U$, where $U\subset X$ is open in $X$.
\end{itemize}
\end{definition}

In Lemma~\ref{lemma: Amg definition is correct}, we show that the~system $\mc B$ is closed under intersections, and therefore the~topology of $\Amg{X}{\mc E}$ is defined correctly.
We then follow with Lemma~\ref{lemma: basic properties of amalgamation spaces}, which captures the~basic and ``local'' properties of $\Amg{X}{\mc E}$. The ``global properties'' of amalgamation spaces are used implicitly in Proposition~\ref{proposition: compactifications of amalgamations}, where amalgamations are used to compactify $X$.

\begin{lemma}\label{lemma: Amg definition is correct}
The topology of $\Amg{X}{\mc E}$ is correctly defined.
\end{lemma}

\begin{proof}
Let $\Amg{X}{\mc E}$ be an amalgamation space and $\mc B$ the~system above, which we claim is a~basis of topology.
Since $\mc B$ obviously covers $\Amg{X}{\mc E}$, it remains to show that intersection of any two elements of $\mc B$ is again in $\mc B$.

This trivially holds when $W_0,W_1$ are two open subsets of the~same $E(A)$.
When $W$ is an open subset of $E(A)$ and $V_U$ corresponds to some open $U\subset X$, we have $V_U \cap E(A) = W^A_U$, so $W\cap V_U$ is again an open subset of $E(A)$.
When $V_U$ and $V_{U'}$ correspond to some opens subsets $U,U'$ of $X$, we have $V_U \cap V_{U'} = V_{U\cap U'} \in \mc B$ by \eqref{equation: V U cap}.

It remains to consider the~situation when $W$ is an open subset of some $E(A)$ and $W'$ is an open some $E(A')$, $A'\neq A$.
We need the~following claim:

\begin{claim}\label{claim: A cap A' is clopen}
For every distinct $A,A'\in \mc A$, $A'\cap E(A)$ is clopen in $E(A)$.
\end{claim}

\begin{proof}[Proof of the~claim]
Note that $A'\cap E(A) =A'\cap A$ is open in $X$, hence in $A$.
Further, we have
\[ A'\cap A \subset W_A^{E(A)} (A'\cap A)
\overset{L\ref{lemma: W A U}}{\underset{\eqref{case: W A U is subset of ...}}\subset}
\overline{A'\cap A}^{E(A)}
\overset{(\mc A2)} = A'\cap A .\]
Thus $A'\cap A$ is open in $E(A)$. It is also closed, as it is compact by $(\mc A2)$.
\end{proof}

Since $W'$ is open in $E(A)$, $W' \cap E(A)$ is open in $E(A') \cap E(A) = A' \cap A$. This set is, in turn, open in $E(A)$ (by the~claim). It follows that $W' \cap E(A)$ is open in $E(A')$, and thus $W\cap W' \in \mc B$.
\end{proof}

\begin{lemma}[Basic properties of amalgamations]\label{lemma: basic properties of amalgamation spaces}
The space $\Amg{X}{\mc E}$ has the~following properties:
\begin{enumerate}[(i)]
\item Each $E(A)$ is clopen in $\Amg{X}{\mc E}$. \label{case: E(A) is clopen in amalgamation}
\item The subspace topologies (inherited from $\Amg{X}{\mc E}$) on $X$ and $E(A)$ for $A\in \mc A$ coincide with the~original topologies of these spaces. \label{case: Amg and E(A) and X}
\item In particular, $X=\Amg{X}{\mc A}$. \label{case: X as Amg}
\item The space $\Amg{X}{\mc E}$ satisfies conditions $(\mc A1)-(\mc A4)$ for $\mc E$. 
	\label{case: Amg and E satisfies A 1 - 4}
\item The space $\Amg{X}{\mc E}$ is Hausdorff, provided that $X$ and each $E\in \mc E$ are Hausdorff \label{case: Amg is T 3}.
\end{enumerate}
\end{lemma}

\begin{proof}
\eqref{case: E(A) is clopen in amalgamation}: Let $A\in \mc A$. $E(A)$ is open in $\Amg{X}{\mc E}$ by definition (since it is open in itself).

To show that $\Amg{X}{\mc E} \setminus E(A)$ is open, we first prove $E(A) = V_A$.
For any $A'\neq A$, we have 
\begin{equation}\label{equation: W of A cap A'}
W^{E(A')}_{A'}(A\cap A') \overset{L\ref{lemma: W A U}}{\underset{\eqref{case: W A U is subset of ...}}\subset}
\overline{A\cap A'}^{E(A)} \overset{(\mc A 2)}{=} A\cap A' \subset A .
\end{equation}
Since $E(A) = W^{E(A)}_A(A)$, it follows that $E(A) = V_A$: 
\[ E(A) = W^{E(A)}_A(A) \subset V_A
= A \cup W^{E(A)}_A (A) \cup \bigcup_{A'\neq A} W^{E(A')}_{A'} (A \cap A')
\overset{\eqref{equation: W of A cap A'}}\subset
A \cup E(A) \cup \bigcup_{A'\neq A} A = E(A) .\]

It remains to show that every point from the~set
\[ \Amg{X}{\mc E} \setminus E(A) = (X\setminus A) \cup \bigcup_{A'\neq A} E(A') \setminus A \]
is contained in some open set disjoint with $E(A)$.
For $x\in E(A') \setminus A$, the~set $E(A') \setminus E(A) = E(A')\setminus A$ is open in $E(A')$ (by Claim~\ref{claim: A cap A' is clopen}), and hence in $\Amg{X}{\mc E}$ as well.
For $x \in X \setminus A$, we have
\[ x \in X\setminus A \subset V_{X\setminus A}
\overset{\eqref{equation: V U cap}}{\subset} \Amg{X}{\mc E} \setminus V_A
 = \Amg{X}{\mc E} \setminus E(A) .\]

\eqref{case: Amg and E(A) and X}: Let $A\in \mc A$. For any basic open set $B$ in $\Amg{X}{\mc E}$, the~intersection $B\cap E(A)$ is open $E(A)$ (by definition of topology of $\Amg{X}{\mc E}$). Moreover, any $W\subset E(A)$ which is open in $E(A)$ is, by definition, also open in $\Amg{X}{\mc E}$. This shows that the~subspace topology of $E(A) \subset \Amg{X}{\mc E}$ coincides with the~original topology of $E(A)$.

We now show that the~subspace topology of $X \subset \Amg{X}{\mc E}$ coincides with the~original topology of $X$. Again, let $B$ be a~basic open subset of $\Amg{X}{\mc E}$.
If $B= V_U$ holds for some open $U\subset X$, we have $B\cap X = U$. When $B$ is an open subset of some $E(A)$, $B\cap X = B \cap A$ is open in $A$, and therefore also open in $X$ (because $A$ is open in $X$). This shows that the~original topology of $X$ is finer than the~subspace topology.

Conversely, for any open subset $U$ of $X$, we have $V_U\cap X=U$, which proves that the~subspace topology of $X$ is finer than the~original topology.

\eqref{case: X as Amg}: By (ii), $X$ is embedded in $\Amg{X}{\mc E}$ for any $\mc E$. Since $\Amg{X}{\mc A}$ adds no new points, this ``canonical'' embedding is a~homeomorphism.

\eqref{case: Amg and E satisfies A 1 - 4}:
$(\mc A1)$ for $\mc E$ (and $\Amg{X}{\mc E}$) is equivalent to \eqref{case: E(A) is clopen in amalgamation}.
By \eqref{equation: amalgamation set}, distinct sets $E(A), E(A') \in \mc E$ satisfy $E(A)\cap E(A') = A \cap A'$. Since $A \cap A'$ is compact by $(\mc A2)$, we get $(\mc A2)$ for $\mc E$ as well.
By \eqref{equation: amalgamation set}, we have
\[ \Amg{X}{\mc E} \setminus \bigcup \mc E = X \setminus \bigcup \mc A = K ,\]
which gives $(\mc A3)$ for $\mc E$.

To prove $(\mc A4)$ for $\mc E$, let $\mc V$ be a~collection of open subsets of $\Amg{X}{\mc E}$ covering $K$.
It suffices to work with a~suitable refinement -- in particular, we can assume that $\mc V$ consists of basic open sets. Denote
\begin{align*}
\mc V_K 		& := \mc V \cap \{ V_U |\, U\subset X \text{ open}\}
.
\end{align*}

By \eqref{equation: amalgamation set}, $\mc V_K$ is an open (in $\Amg{X}{\mc E}$) cover of $K$ and
\[ \mc U_K := \{ U\subset X | \ V_U \in \mc V_K \} \]
is an open (in $X$) cover of $K$. By $(\mc A4)$, there is a~finite family $\mc A' \subset \mc A$ s.t. for every $A\in \mc A \setminus A'$ we have $U_A \supset A \setminus \bigcup \mc A'$ for some $U_A \in \mc U$.
For any $A\in \mc A \setminus \mc A'$, the~set $C_A := A \cap \bigcup \mc A'$ is compact by $(\mc A2)$. Consequently, we have
\[ V_{U_A} \supset W^{E(A)}_A( U_A \cap A)
\overset{L\ref{lemma: W A U}}{\underset{\eqref{case: W A U monotonicity}}\supset}
W^{E(A)}_A(A \setminus C_A)
\overset{L\ref{lemma: W A U}}{\underset{\eqref{case: W A U with compact complement}}=}
E(A) \setminus C_A .\]
It follows that $V := V_{U_A}$ is an element of $\mc V$ satisfying $V \cup \bigcup \mc A' \supset E(A)$, which shows that $(\mc A4)$ holds for $\mc E$.

\eqref{case: Amg is T 3}:
First, we show that $\Amg{X}{\mc E}$ is Hausdorff.
Let $x,y\in\Amg{X}{\mc E}$ be distinct. It suffices to consider the~cases where
\begin{enumerate}[1)]
\item $x\in E(A)$ and $y\notin E(A)$ for some $A\in\mc A$,
\item $x,y \in E(A)$ for some $A\in\mc A$ and
\item $x,y \in K = X \setminus \bigcup \mc A$.
\end{enumerate}
In the~first case, the~open sets separating $x$ from $y$ are $\Amg{X}{\mc E}\setminus E(A)$ and $E(A)$ (by (i)).
In the~second case, we use the~fact that $E(A)$ is Hausdorff to find open $W,W'\subset E(A)$ which separate $x$ and $y$. By (i), $W$ and $W'$ are open in $\Amg{X}{\mc E}$ as well.

In the~last case, we use the~fact that $X$ is Hausdorff to get some disjoint open subsets $U, U'$ of $X$ for which $x \in U$ and $y\in U'$. By \eqref{equation: V U cap}, $V_U$ and $V_{U'}$ are disjoint as well.

To prove that $\Amg{X}{\mc E}$ is Tychonoff, it suffices to show that it is a~subspace of some compact space.
By \eqref{case: Amg and E satisfies A 1 - 4}, we can construct the~amalgamation $C := \Amg{\Amg{X}{\mc E}}{\beta \mc E}$.
By \eqref{case: X as Amg}, $\Amg{X}{\mc E}$ is a~subspace of $C$.
In Proposition~\ref{proposition: compactifications of amalgamations}, we show that the~amalgamation $C$ is compact (obviously, without relying on the~fact that $\Amg{X}{\mc E}$ is Tychonoff).
\end{proof}

\begin{proposition}[Compactifications of amalgamation spaces] \label{proposition: compactifications of amalgamations}
Let $X$ be a~Tychonoff space, $\mc A\subset \mc P(X)$ a~family satisfying $(\mc A1)$-$(\mc A4)$ and $cA$, $dA$ (Hausdorff) compactifications of $A$ for every $A\in\mc A$.
\begin{enumerate}[(i)]
\item For any regular topological space $Z$, a~function $f: \Amg{X}{\mc E} \rightarrow Z$ is continuous if and only if all the~restrictions $f|_X$ and $f|_{E(A)}$, $A\in\mc A$, are continuous. \label{case: continuous functions on amalgamation space}
\item $\Amg{X}{c\mc A}$ is a~compactification of $X$. In particular, $\Amg{X}{\mc A}$ is Tychonoff.
	\label{case: compactification of Amg}
\item $\Amg{X}{c\mc A} \preceq \Amg{C}{d\mc A} $ holds whenever $cA\preceq dA$ for each $A\in\mc A$.
	\label{case: smaller compactifications of Amg}
\item For every compactification $cX$ of $X$, we have $\Amg{X}{\overline{\mc A}^{cX}} \succeq cX$.\footnote{Analogously to Notation~\ref{notation: c A} we define $\overline{\mc A}^{cX} := \{ \overline{A}^{cX} | \ A\in\mc A \}$}
	\label{case: larger compactification of Amg type}
\item In particular, $\beta X = \Amg{X}{\beta \mc A}$.
	\label{case: beta compactification as Amg}
\end{enumerate}
\end{proposition}

\begin{proof}
\eqref{case: continuous functions on amalgamation space}: It remains to prove ``$\Leftarrow$''.
Let $f: \Amg{X}{\mc E} \rightarrow Z$ and suppose that all the~restrictions are continuous. We need to prove that $f$ is continuous at each point of $\Amg{X}{\mc E}$.

Let $x\in E(A)$ for some $A\in \mc A$.
Since $E(A)$ is clopen in $\Amg{X}{\mc E}$ (by \eqref{case: E(A) is clopen in amalgamation} of Lemma~\ref{lemma: basic properties of amalgamation spaces}) and $f|_{E(A)}$ is continuous (by the~assumption), $f$ is continuous at $x$.

Let $x\in K$. Let $G$ be an open neighborhood of $f(x)$ in $Z$ and let $H$ be an open neighborhood of $f(x)$ satisfying $\overline{H} \subset G$.
By continuity of $f|_X$, there is some open neighborhood $U$ of $x$ in $X$ which satisfies $f(U) \subset H$.
In particular, we have $f(U\cap A) \subset H$ for any $A\in \mc A$. Since $\overline{U \cap A} \subset E(A)$ and $f|_{E(A)}$ is continuous, we have
\[ f(\overline{U \cap A}) \subset \overline{ f( U \cap A )} \subset \overline{H} \subset G .\]
It follows that $f( U \cup \bigcup_{\mc A} \overline{U \cap A} ) \subset G$.
This proves that $V_U$ is an open neighborhood of $x$ which is mapped into $G$:
\[ x \in V_U \overset{\text{def.}}{=} U \cup \bigcup_{A \in \mc A} W^{E(A)}_A(U\cap A)
\overset{L\ref{lemma: W A U}}{\underset{\eqref{case: W A U is subset of ...}}\subset}
U \cup \bigcup_{A \in \mc A} \overline{U \cap A}
\ \ \& \ \ f( U \cup \bigcup_{\mc A} \overline{U \cap A} ) \subset G .\]

\eqref{case: compactification of Amg}:
By Lemma~\ref{lemma: basic properties of amalgamation spaces}, we already know that $\Amg{X}{c\mc A}$ is Hausdorff. Once we know that $\Amg{X}{c\mc A}$ is compact, we get that it is Tychonoff for free, which proves the~``in particular'' part.

Let $\mc V$ be an open cover. As in the~proof of Lemma~\ref{lemma: basic properties of amalgamation spaces} \eqref{case: Amg and E satisfies A 1 - 4}, we can assume that $\mc V$ consists of basic open sets, denoting
\begin{align*}
& \mc V_K := \mc V \cap \{ V_U |\, U\subset X \text{ open}\} \text{ and} \\
& \mc V_{cA} := \mc V \cap \{ W | \, W \cap cA \neq \emptyset \}, \ A\in\mc A .
\end{align*}

Clearly, we have $\bigcup \mc V_{cA} \supset cA$ for every $A\in \mc A$, so there exist some finite subfamilies $\mc V'_{cA}$ of $\mc V_{cA}$ satisfying $\bigcup \mc V'_{cA} \supset cA$.
Similarly we have $\bigcup \mc V_K \supset K$ and we denote by $\mc V'_K$ be some finite subfamily of $\mc V_K$ satisfying $\bigcup V'_K \supset K$.

By Lemma~\ref{lemma: basic properties of amalgamation spaces} \eqref{case: Amg and E satisfies A 1 - 4}, $(\mc A4)$ holds for $\Amg{X}{c\mc A}$ and $c\mc A$.
Applying $(\mc A4)$ yields a~finite family $c\mc A' \subset c \mc A$, such that every $cA \in c \mc A \setminus c \mc A'$ satisfies $V \cup \bigcup c\mc A' \supset cA$ for some $V \in \mc V'_K$.
In particular, $\mc V'_K$ covers the~whole space $\Amg{X}{c\mc A}$ except for $\bigcup c\mc A'$.

It follows that $\mc V' := \mc V'_K \cup \bigcup_{c\mc A'} \mc V'_{cA}$ is a~finite subcover of $\Amg{X}{c\mc A}$.

\eqref{case: smaller compactifications of Amg}: For $A\in \mc A$, denote by $q_A$ the~mapping witnessing that $cA\preceq dA$ and define $\varphi: \Amg{T}{d\mc A} \rightarrow \Amg{T}{c\mc A}$ as
\begin{equation*}
\varphi (x) :=
\begin{cases}
x & \text{  } x \in K , \\
q_A(x) & \text{ for } x\in dA, \ A\in \mc A.
\end{cases}
\end{equation*}
Clearly, $\varphi$ satisfies $\varphi|_X = \text{id}_X$.
By \eqref{case: continuous functions on amalgamation space}, the~mapping $\varphi$ is continuous.
This proves that $\varphi$ witnesses  $\Amg{T}{c\mc A} \preceq \Amg{T}{d\mc A}$.

\eqref{case: larger compactification of Amg type}: Let $X$ and $cX$ be as in the~statement. We denote by $i_A : \bar A \rightarrow cX $ the~identity mapping between $\bar{A} \subset \Amg{X}{\overline{\mc A}^{cX}}$ and $\bar{A} \subset cX$. We also denote as $i_X : X \rightarrow cX$ the~identity between $X \subset \Amg{X}{\overline{\mc A}^{cX}}$ and $X \subset cX$. By definition of topology on the~amalgamation space, $i_X$ and each $i_A$ is an embedding. We define a~mapping $\varphi : \Amg{X}{\overline{\mc A}^{cX}} \rightarrow cX$ as
\begin{equation*}
\varphi (x) :=
\begin{cases}
i_X(x) & \text{  } x \in X , \\
i_A(x) & \text{ for } x\in \bar A, \ A\in \mc A.
\end{cases}
\end{equation*}
The mapping $\varphi$ is well-defined and continuous (by \eqref{case: continuous functions on amalgamation space}). In particular, $\varphi$ witnesses that $cX \preceq \Amg{X}{\overline{\mc A}^{cX}}$.

\eqref{case: beta compactification as Amg}: This is an immediate consequence of \eqref{case: larger compactification of Amg type}.
\end{proof}



In the~following two subsections, we show two different ways of constructing compactifications of a~broom space $T$, and compute the~complexity of $T$ in each of them.

\subsection{Compactifications \texorpdfstring{$c_\gamma T$}{cT} and Broom Spaces \texorpdfstring{$\mathbf{T}_\alpha$}{T alpha}}
	\label{section: cT}

In one type of broom space compactifications, which we call $c_\gamma T$, the~closures of each $A\in \mc A$ are either the~smallest possible compactification $A$ (the Alexandroff one-point compactification), or the~largest possible one (the Čech-stone compactification):

\begin{notation}[Compactifications $c_\gamma T$]\label{notation:cT}
Let $\gamma \leq \omega_1$. For $A\in \mc A_{\omega_1}$, we denote
\[ c_\gamma A :=
\begin{cases}
	\alpha A, 	& \textrm{for } A\in \mc A_{<\gamma}\\
	\beta A, 	& \textrm{for } A\in \mc A_{\omega_1} \setminus \mc A_{<\gamma} ,
\end{cases} \]
where $A$ is endowed with the~discrete topology.
For an AD broom space $T = T_{\mc A}$, we set $c_\gamma T := \Amg{T}{c_\gamma \mc A}$.
\end{notation}

By Proposition~\ref{proposition: compactifications of amalgamations}, $c_\gamma T$ is a~compactification of $T$.
Since $\mc A_{<0} = \emptyset$, $c_0(\cdot)$ assigns to each $A$ its Čech-Stone compactification. It follows that $c_0 T = \beta T$:
\[ c_0 T \overset{\text{def.}}{\underset{\text{of }c_0 T}=} 
\Amg{T}{c_0 \mc A} \overset{\mc A_{<0}}{\underset{=\emptyset}=}
\Amg{T}{\beta \mc A} \overset{P\ref{proposition: compactifications of amalgamations}}{\underset{\eqref{case: beta compactification as Amg}}=}
\beta T .\]
On the~other hand, when $\gamma$ is such that the~whole $\mc A$ is contained in $\mc A_{<\gamma}$, every $c_\gamma A$ will be equal to $\alpha A$. This gives the~second identity in the~following observation:\footnote{$\Amg{T_{\mc A}}{\alpha \mc A}$ actually corresponds to the~compactification from in \cite{talagrand1985choquet}, used to prove the~existence of a~non-absolute $\fsd$ space.}
\[							\mc A \subset \mc A_{<\alpha} \ \& \ \gamma \geq \alpha \implies
c_\gamma T					\overset{\text{def.}}{\underset{\text{of }c_\gamma T}=}
\Amg{T}{c_\gamma \mc A}		\overset{\text{def.}}{\underset{\text{of }c_\gamma}=}
\Amg{T}{\alpha \mc A} .\]
By Proposition~\ref{proposition: compactifications of amalgamations}\,\eqref{case: smaller compactifications of Amg}, we obtain the~following chain of compactifications
\[ \beta T = c_0 T \succeq c_1 T \succeq \dots \succeq c_\gamma T
\succeq \dots \succeq c_\alpha T = c_{\alpha+1} T = \dots = c_{\omega_1} T , \]
which stabilizes at the~first ordinal $\alpha$ for which $\mc A \subset \mc A_{<\alpha}$.

The next lemma is the~only part which is specific to Talagrand's broom spaces:

\begin{lemma}[Lower bound on the~complexity in $c_\gamma T$] \label{lemma: c T - lower bound}
Let $\gamma \in [2,\alpha]$. If an AD broom space $T = T_\mc A$ satisfies $\mc A \supset \mc A^T_{<\gamma}$, then we have $T \notin \mc F_{<\gamma} (c_\gamma T)$.
\end{lemma}

\begin{proof}
Assuming for contradiction that $T \in \mc F_{<\gamma}(c_\gamma T)$, we can apply Lemma~\ref{lemma: Talagrand lemma 3} to $X=T$ and $Y=c_\gamma T$. It follows that there is an $\mc A^T_{<\gamma}$-witness $A$ for $T$ in $c_\gamma T$.
Since $A$ belongs to $\mc A^T_{<\gamma} \subset \mc A$, it is closed and discrete in $T$. By definition of $c_\gamma T$, we have $\overline{A}^{c_\gamma T} = c_\gamma A = \alpha A$ -- this contradicts Lemma~\ref{lemma: closed discrete witnesses}.
\end{proof}

In particular, this yields the~following result which we promised in Section~\ref{section:overview}:

\begin{corollary}[Existence of spaces with additive complexity]\label{corollary: T alpha for odd}
For any odd $\beta\in [3,\omega_1)$, there exists a~space $X^\beta_2$ satisfying
\[ \left\{ 2, \beta \right\} \subset \textnormal{Compl}\left( X^\beta_2 \right) \subset [2, \beta] .\]
\end{corollary}

\begin{proof}
Let $\beta\in [3,\omega_1)$ be odd and set $X^\beta_2 := \mathbf{T}_\beta$.
Since $\mathbf{T}_\beta = T_{\mc A_{<\beta}}$ holds by definition, we can apply Proposition~\ref{proposition: basic broom complexities}\,$(ii)$ to get $\textnormal{Compl}(\mathbf{T}_\beta) \subset [0,\beta]$.
By Proposition~\ref{proposition: basic broom complexities}\,$(i)$, we have $2 \in \textnormal{Compl}(\mathbf{T}_\beta)$ and hence (by Proposition~\ref{proposition: basic attainable complexities}\,\eqref{case: compact and sigma compact})
\[ \left\{ 2 \right\} \subset \textnormal{Compl}\left( \mathbf{T}_\beta \right) \subset [2, \beta] .\]
By Lemma~\ref{lemma: c T - lower bound}, we have $\Compl{\mathbf{T}_\beta}{c_\beta \mathbf{T}_\beta} \geq \beta$, which implies that $\textnormal{Compl}(\mathbf{T}_\beta)$ contains $\beta$. This shows that $\textnormal{Compl}\left( \mathbf{T}_\beta \right) \subset [2, \beta]$ and concludes the~proof.
\end{proof}

The next lemma proves an upper estimate on the~complexity of $T$ in $c_\gamma T$.

\begin{lemma}[Upper bound on the~complexity in $c_\gamma T$] \label{lemma: c T - upper bound}
For any AD broom space $T$ and $\gamma\in [2,\omega_1]$, we have $T \in \mc F_\gamma (c_\gamma T)$.
\end{lemma}

\begin{proof}
Let $T=T_{\mc A}$ be an AD broom space and $\gamma\in [2,\omega_1]$.
We shall prove that $T$, $c_{\gamma} T$ and $\gamma$ satisfy the~assumptions of Lemma~\ref{lemma: S(y) and complexity of T} (which gives $T \in \mc F_\gamma (c_\gamma T)$).
To apply Lemma~\ref{lemma: S(y) and complexity of T}, we need to show that for every $x\in c_\gamma T \setminus T$, the~following set $S(x)$ is either finite or it can be covered by finitely sets $\D(A)$, $A\in \mc A_{<\gamma}$:
\begin{align*}
S(x) & = \left\{ s\in\seq | \ \overline{\mc N(s)}^{c_\gamma T} \ni x \right\} 
\end{align*}

First, we observe that $\overline{\mc N(s)}^{c_\gamma T} \cap \overline{A}^{c_\gamma T} = \overline{\mc N(s) \cap A}^{c_\gamma T}$ holds for any $s\in\seq$ and $A\in \mc A$.
The inclusion ``$\supset$'' is trivial.
For the~converse inclusion, let $x\in \overline{\mc N(s)}^{c_\gamma T} \cap \overline{A}^{c_\gamma T}$ and let $U$ be an open neighborhood of $x$ in $c_\gamma T$.
Since $\overline{A}^{c_\gamma T} = c_\gamma A$ is open in $c_\gamma T$ (Lemma~\ref{lemma: basic properties of amalgamation spaces}\,\eqref{case: E(A) is clopen in amalgamation}), $U\cap \overline{A}^{c_\gamma T}$ is an open neighborhood of $x$.
In particular, $x\in \overline{\mc N(s)}^{c_\gamma T}$ gives $U \cap \overline{A}^{c_\gamma T} \cap \mc N(s) \neq \emptyset$.
Since $\overline{A}^{c_\gamma T} \cap \mc N(s) = A \cap \mc N(s)$, it follows that $U$ intersects $A \cap \mc N(s)$.
This proves that $x$ belongs to $\overline{A \cap \mc N(s)}^{c_\gamma T}$.

Recall that $c_\gamma T \setminus T = \bigcup_{\mc A} \overline{A}^{c_\gamma T} \setminus A$. 
As the~first case, we shall assume that $x \in \overline{A}^{c_\gamma T} \setminus A$ for some $A \in \mc A \setminus \mc A_{<\gamma}$, and prove that the~set $S(x)$ is finite.
Suppose that such an $x$ satisfies $x\in \overline{\mc N(s)}^{c_\gamma T}$ and $x\in \overline{\mc N(t)}^{c_\gamma T}$ for two sequences $s$ and $t$.
By definition of $c_\gamma T$, we have $\overline{A}^{c_\gamma T} = c_\gamma A = \beta A$. It follows that
\begin{equation}\label{equation: x and N of s}
x \in \overline{\mc N(s)}^{c_\gamma T} \cap \overline{\mc N(t)}^{c_\gamma T} \cap \overline{A}^{c_\gamma T}
= \overline{\mc N(s)\cap A}^{c_\gamma T} \cap \overline{\mc N(t) \cap A}^{c_\gamma T}
= \overline{\mc N(s)\cap A}^{\beta A} \cap \overline{\mc N(t) \cap A}^{\beta A} .
\end{equation}
Recall that for any normal topological space $X$ and closed $E,F\subset X$, we have $\overline{E}^{\beta X} \cap \overline{F}^{\beta X} = \overline{E \cap F}^{\beta X}$. In particular, this holds for $A$ (which is discrete).
Applying this to \eqref{equation: x and N of s} yields
\[ x \in \overline{\mc N(s) \cap \mc N(t) \cap A}^{\beta A} \subset
 \overline{\mc N(s) \cap \mc N(t) \cap A}^{c_\gamma T} \subset
 \overline{\mc N(s) \cap \mc N(t)}^{c_\gamma T} .\]
Since $x$ does not belong to $T \supset \mc N(s) \cap \mc N(t)$, the~set $\mc N(s) \cap \mc N(t)$ must in particular be non-empty.
The only way this might happen is when the~sequences $s$ and $t$ are comparable.
Consequently, $S(x)$ consists of a~single branch.
Since $x \notin \mc A(\overline{\mc N}^{c_\gamma T}\!)$ by Lemma~\ref{lemma: A of N}, this branch must necessarily be finite (by Lemma~\ref{lemma: IF trees}).

The remaining case is when $x$ belongs to $\overline{A}^{c_\gamma T} \setminus A$ for some $A \in \mc A \cap \mc A_{<\gamma}$, that is, when we have $x=x_A$.
We claim that such $x$ satisfies $S(x_A)=\D(A)$. Indeed, any $s\in\seq$ satisfies
\begin{equation} \label{equation: S of x_A}
s \in S(x_A) \iff x_A\in \overline{\mc N(s)}^{c_\gamma T} \iff x_A\in \overline{\mc N(s) \cap A}^{c_\gamma T} \iff \mc N(s) \cap A \textrm{ is infinite} .
\end{equation}
Clearly, \eqref{equation: S of x_A} is further equivalent to $\mc A$ containing infinitely many distinct extensions of $s$.
This happens precisely when $\cltr{A}$ contains infinitely many incomparable extensions of $s$.
Since each branch of $\cltr{A}$ is infinite, we can assume that each two of these extensions have different lengths. In other words, $s$ belongs to $S(x_A)$ precisely when it belongs to $\D(A)$.
\end{proof}

Finally, we apply all the~results to the~particular case of $T=\mathbf{T}_\alpha$:

\begin{proof}[Proof of the~``$\mathbf{T}_\alpha$'' part of Theorem~\ref{theorem: complexity of talagrands brooms}]
Let $\alpha \in [2,\omega_1]$. Since $\mathbf{T}_\alpha$ corresponds to the~family $\mc A = \mc A^T_{<\alpha} \subset \mc A_{<\alpha}$, we can apply Proposition~\ref{proposition: basic broom complexities} to get
\[ \textnormal{Compl}(\mathbf{T}_\alpha) \subset [2,\alpha] .\]
For any $\gamma \in [2,\alpha]$, $\mathbf{T}_\alpha$ is an $\mc F_\gamma$ subset of $c_\gamma \mathbf{T}_\alpha$ (by Lemma~\ref{lemma: c T - upper bound}), but it is \emph{not} its $\mc F_{<\gamma}$ subset (by Lemma~\ref{lemma: c T - lower bound}). It follows that $\Compl{\mathbf{T}_\alpha}{c_\gamma \mathbf{T}_\alpha} = \gamma$ and therefore $\gamma \in \textnormal{Compl}(\mathbf{T}_\alpha)$.
Since $\gamma \in [2,\alpha]$ was arbitrary, we get the~desired result:
\[ \textnormal{Compl}(\mathbf{T}_\alpha) = [2,\alpha] .\]
\end{proof}

\subsection{Compactifications \texorpdfstring{$d_\gamma T$}{dT} and Broom Spaces \texorpdfstring{$\mathbf{S}_\alpha$}{S alpha}} \label{section: dT}

The construction from Section~\ref{section: cT} is more involved than gluing together pre-existing examples, but nonetheless, it has somewhat similar flavor to the~approach from Section~\ref{section: topological sums}. We will now show that we can get the~same result by ``relying on the~same sets the~whole time''.

For $A\in \mc A_{\omega_1}$ and $h\in \seq$, we shall write
\[ A(h):=\left\{ \sigma \in A | \ \sigma \sqsupset h \right\} .\]

\noindent When $h=\emptyset$, we have $A(\emptyset) = A$ and $A(h)$ is ``as complicated as it can be'', in the~sense that its rank $\rank$ is the~same as the~rank of $A$. Conversely, when $h$ belongs to the~$B$ of which $A$ is a~broom-extension, then $A(h)$ is ``as simple as it can be'' -- it belongs to $\mc A_1$. The following lemma shows that the~intermediate possibilities are also possible, an defines the~corresponding set of ``$<\!\gamma$-handles'' of $A$.

\begin{lemma}[$H_{<\gamma}(\cdot)$, the~set $<\!\gamma$-handles]\label{lemma: H gamma}
Let $A$ be a~broom extension of $B$. For $\gamma\in [2,\omega_1]$, denote \emph{the set of $<\!\gamma$-handles of $A$} as
\[ H_{<\gamma}(A) := \{ h\in \cltr{B} | \ A(h)\in \mc A_{<\gamma} \ \& \text{ no other } s\sqsubset h \text{ satisfies } A(s) \in \mc A_{<\gamma} \}  .\]
\begin{enumerate}[(i)]
\item $A$ is the~disjoint union of sets $A(h)$, $h\in H_{<\gamma}(A)$.
\item If $A(h_0) \in \mc A_{<\gamma}$ holds for some $h_0\in \seq$, then we have $h\sqsubset h_0$ for some $h\in H_{<\gamma}(A)$.
\end{enumerate}
\end{lemma}

\begin{proof}
Let $A$, $B$ and $\gamma$ be as in the~statement.
First, we give a~more practical description of $H_{<\gamma}(A)$.
For each $s\in B$, $A(s)$ is a~broom extension of $\{s\} \in \mc B_1 \subset \mc B_{<\gamma}$. It follows that $A(s) \in \mc A_{<\gamma}$.
Moreover, we have
\[ \left( \forall u \sqsubset v \in \cltr{B} \right)
 \left(\forall \alpha<\omega_1\right): 
 A(u) \in \mc A_\alpha \implies A(v) \in \mc A_\alpha .\]
Indeed, this follows from the~definition of $\mc B_\alpha$.
Consequently, for each $s\in B$, there exists a~minimal $n\in\omega$ for which $A(s|n_s) \in \mc A_{<\gamma}$.
We claim that
\[ H_{<\gamma}(A) = \{ s|n_s \ | \ s\in B \} .\]
Indeed, each $s|n_s$ belongs to $H_{<\gamma}(A)$ by minimality of $n_s$. Conversely, any $h \in \cltr{B}$ satisfies $h \sqsubset s$ for some $s\in B$, so $h$ is either equal to $s|n_s$, or it does not belong to $H_{<\gamma}(A)$.

(i):
For each $\sigma\in A$, there is some $s\in B$ s.t. $s \sqsubset \sigma$.
It follows that $\sigma \in A(s|n_s)$, which proves that $A$ is covered by the~sets $A(h)$, $h\in H_{<\gamma}(A)$.

For distinct $s, t \in B$, the~initial segments $s|n_s $ and $t|n_t$ are either equal, or incomparable (by minimality of $n_s$ and $n_t$).
It follows that distinct $g,h\in H_{<\gamma}(A)$ are incomparable, which means that distinct $A(g)$ and $A(h)$ are disjoint.

(ii) follows from the~minimality of $n_s$ (and the~trivial fact that for each $h\in \cltr{B}$, there is some $s\in B$ with $s \sqsupset h$).
\end{proof}

By default, we equip each $A\in \mc A_{\omega_1}$ with a~discrete topology.
Using the~sets of $<\!\gamma$-handles, we define new compactifications of broom spaces. Unlike the~compactifications from $c_\gamma T$ from Section~\ref{section: cT}, these are no longer ``all or nothing'', but there are many intermediate steps between $\alpha A$ on one end and $\beta A$ on the~other:

\begin{notation}[Compactifications $d_\gamma T$]
For $\gamma \in [2,\omega_1]$ and $\mc A \in \mc A_{\omega_1}$, we define
\[ E_\gamma(A) := \bigoplus_{h\in H_{<\gamma}(A)} \alpha A(h)
 = \bigoplus_{h\in H_{<\gamma}(A)} \left( A(h) \cup \{ x_{A(h)} \} \right). \]
For an AD broom space $T$, we define a~compactification $d_\gamma T$ as
\[ d_\gamma T := \Amg{T}{d_\gamma \mc A} \text{, where } d_\gamma A := \beta E_\gamma(A) .\]
\end{notation}

By Lemma~\ref{lemma: H gamma}(ii), we have $E_\gamma(A) \supset A$, as well as the~fact that for distinct $g,h\in H_{<\gamma}(A)$, we have $x_{A(h)}\neq x_{A(g)}$.

We now show that the~compactifications $d_\gamma T$ have similar properties as $c_\gamma T$.

\begin{lemma}[Lower bound on the~complexity in $d_\gamma T$]\label{lemma: d T - lower bound}
Let $\gamma \in [2,\omega_1)$. If an AD system $\mc A\subset \mc A_{\omega_1}$ contains $\mc A^T_{<\alpha} \setminus \mc A^T_{<\alpha-1}$ for some non-limit $\alpha\geq \gamma$, then we have $T\notin \mc F_{<\gamma}(d_\gamma T)$.
\end{lemma}

\begin{proof}
Assuming for contradiction that $T \in \mc F_\eta (d_\gamma T)$ holds for some $\eta < \gamma$, we can apply Lemma~\ref{lemma: technical version of Talagrand lemma 3} (with `$\delta$'$=\alpha-1$).
It follows that there is some $A \in \mc A^T_{\alpha-1} \setminus \mc A^T_{<\alpha-1} \subset \mc A$ and $h_0\in \seq$, such that $A(h_0) \in \mc A^T_\eta \subset \mc A^T_{<\gamma}$ is a~witness for $T$ in $d_\gamma T$.

By Lemma~\ref{lemma: H gamma} (iii), there is some $h\in H_{<\gamma}(A)$ s.t. $h_0 \sqsupset h$.
Since $A$ belongs to $\mc A$, the~sets $A(h_0) \subset A(h) \subset A$ are all closed discrete in $T$. Moreover, we have $\overline{A(h)}^{d_\gamma T} = \alpha A(h)$ by definition of $d_\gamma T$. Because $A(h_0)$ is infinite, we get
\[ \overline{A(h_0)}^{d_\gamma T} = \overline{A(h_0)}^{\alpha A(h)} = A(h_0) \cup \{ x_{A(h)} \} .\]

We have found a~closed (in $T$) discrete witness for $T$ in $d_\gamma T$, whose closure in $d_\gamma T$ is homeomorphic to its one-point compactification -- a~contradiction with Lemma~\ref{lemma: closed discrete witnesses}.
\end{proof}

\begin{lemma}[Upper bound on the~complexity in $d_\gamma T$] \label{lemma: d T - upper bound}
For any AD broom space $T$ and $\gamma\in [2,\omega_1]$, $T \in \mc F_\gamma (d_\gamma T)$.
\end{lemma}

\begin{proof}
Let $x\in d_\gamma T \setminus T$.
When $x$ belongs to $\beta E_\gamma(A) \setminus  E_\gamma(A)$ for some $A\in \mc A$, we use the~exact same method as in Lemma~\ref{lemma: c T - upper bound} to prove that $S(x)$ is finite.

When $x$ belongs to $E_\gamma(A) \setminus A$ for some $A\in \mc A$, we have $x=x_{A(h)}$ for some $h\in H_{<\gamma } (A)$. By definition of $H_{<\gamma}(A)$, we have $A(h) \in \mc A_{<\gamma}$.
The approach from Lemma~\ref{lemma: c T - upper bound} yields $ S(x_{A(h)}) = \D(A\left(h\right))$.

We have verified the~assumptions of Lemma~\ref{lemma: S(y) and complexity of T}, which gives $T \in \mc F_\gamma (d_\gamma T)$.
\end{proof}

We have all the~ingredients necessary to finish the~proof of Theorem~\ref{theorem: complexity of talagrands brooms}:

\begin{proof}[Proof of the~``$\mathbf{S}_\alpha$'' part of Theorem~\ref{theorem: complexity of talagrands brooms}]
Let $\alpha \in [2,\omega_1)$ be a~non-limit ordinal. Since $\mathbf{S}_\alpha$ corresponds to the~family
\[ \mc A = \mc A^T_{<\alpha} \setminus \mc A^T_{<\alpha-1} \subset \mc A_{<\alpha} ,\]
we can apply Proposition~\ref{proposition: basic broom complexities} to get
\[ \textnormal{Compl}(\mathbf{S}_\alpha) \subset [2,\alpha] .\]
For any $\gamma \in [2,\alpha]$, $\mathbf{S}_\alpha$ satisfies $\Compl{\mathbf{S}_\alpha}{d_\gamma \mathbf{S}_\alpha}\leq\gamma $ (by Lemma~\ref{lemma: d T - upper bound}) and $\Compl{\mathbf{S}_\alpha}{d_\gamma \mathbf{S}_\alpha}\geq\gamma $ (by Lemma~\ref{lemma: d T - lower bound}). It follows that $\Compl{\mathbf{S}_\alpha}{d_\gamma \mathbf{S}_\alpha} = \gamma$ and therefore $\gamma \in \textnormal{Compl}(\mathbf{S}_\alpha)$.
Since $\gamma \in [2,\alpha]$ was arbitrary, we get
\[ \textnormal{Compl}(\mathbf{S}_\alpha) \supset [2,\alpha] ,\]
which concludes the~proof.
\end{proof}
%

\section*{Acknowledgment}
I would like to thank my supervisor, Ondřej Kalenda, for numerous very helpful
suggestions and fruitful consultations regarding this paper.
I am grateful to Adam Bartoš for discussions related to this paper.
This work was supported by the research grants  GAČR 17-00941S and  GA UK No. 915.


\end{document}